\documentclass[11pt]{amsart}
\usepackage{geometry}                 
\usepackage{graphicx}
\usepackage{amssymb}
\usepackage{epstopdf}
\usepackage{amsmath}
\usepackage{color}
\usepackage{hyperref}
\usepackage{pdfsync}
 \usepackage{tikz}
 
\usepackage[all]{xy}
\xyoption{matrix}
\xyoption{arrow}

\usetikzlibrary{arrows,decorations.pathmorphing,backgrounds,positioning,fit,petri}

 \tikzset{help lines/.style={step=#1cm,very thin, color=gray},
help lines/.default=.5} 
\tikzset{thick grid/.style={step=#1cm,thick, color=gray},
thick grid/.default=1} 

\newtheorem{thm}{Theorem}[subsection]
\newtheorem{lem}[thm]{Lemma}
\newtheorem{cor}[thm]{Corollary}

\theoremstyle{definition}
\newtheorem{defn}[thm]{Definition}
\newtheorem{eg}[thm]{Example}

\theoremstyle{remark}
\newtheorem{rem}[thm]{Remark}
 
\numberwithin{equation}{section}

\newcommand{\mat}[1]{\ensuremath{
\left[\begin{matrix}#1
\end{matrix}\right]
}}

\newcommand{\vs}[1]{\vskip .#1 cm} 

 \newcommand{\onto}{\twoheadrightarrow}
 \newcommand{\cof}{\rightarrowtail}

\newcommand{\smallcoprod}{\,{\textstyle{\coprod}}\,}

\DeclareMathOperator{\coker}{coker}
\DeclareMathOperator{\Hom}{Hom}%
\DeclareMathOperator{\Ext}{Ext}%
\DeclareMathOperator{\End}{End}%
\DeclareMathOperator{\add}{add} 

\DeclareMathOperator{\undim}{\underline{dim}}
 
\newcommand{\field}[1]{\mathbb{#1}}
\newcommand{\ZZ}{\ensuremath{{\field{Z}}}}

\newcommand{\RR}{\ensuremath{{\field{R}}}}

\newcommand{\NNN}{\ensuremath{{\field{N}}}}

\newcommand{\commentout}[1]{}

\newcommand{\cA}{\ensuremath{{\mathcal{A}}}}

\newcommand{\cC}{\ensuremath{{\mathcal{C}}}}
\newcommand{\cD}{\ensuremath{{\mathcal{D}}}}

\newcommand{\cF}{\ensuremath{{\mathcal{F}}}}

\newcommand{\cG}{\ensuremath{{\mathcal{G}}}}
\newcommand{\cH}{\ensuremath{{\mathcal{H}}}}

\newcommand{\cP}{\ensuremath{{\mathcal{P}}}}
\newcommand{\cQ}{\ensuremath{{\mathcal{Q}}}}

\newcommand{\maskA}[1]{\draw[thin,fill,color=#1] (-1.5,0.5) circle[radius=2.8cm];}
\newcommand{\maskB}[1]{\draw[thin,fill,color=#1] (1.5,0.5) circle[radius=2.8cm];}
\newcommand{\maskC}[1]{\draw[thin,fill,color=#1] (0,-2.5) circle[radius=2.8cm];}
\newcommand{\maskCC}[1]{
	\begin{scope}
		\clip (0,-2.5) circle[radius=2.8cm];
		\draw[thin,fill,color=#1] 
		(-1.5,0.5) circle[radius=2.8cm];
	\end{scope}
}
\newcommand{\maskZ}[1]{\draw[thin,fill,color=#1] (0,-.42) circle[radius=2.43cm];}
\newcommand{\maskX}[1]{
	\begin{scope}[rotate=-63,xshift=.55cm]
		\draw[thin,fill,color=#1] (0,-1.14) ellipse[x 			radius=3cm,y radius=2.235cm];
	\end{scope}
}
\newcommand{\maskY}[1]{ 
	\begin{scope}[rotate=63,xshift=-.55cm]
		\draw[thin,fill,color=#1] (0,-1.14) ellipse[x 			radius=3cm,y radius=2.235cm];
	\end{scope}
}

\newcommand{\maskAa}[1]{
\draw[thin,fill,color=#1] (-3.5,0) rectangle (0,-3);
}
\newcommand{\maskAb}[1]{
\draw[thin,fill,color=#1] (-3.5,0) rectangle (0,3);
}
\newcommand{\maskAe}[1]{
\draw[thin,fill,color=#1] (3.5,0) rectangle (0,3);
}
\newcommand{\maskAc}[1]{
\begin{scope}
\clip (3.5,0) rectangle (0,-3);
\begin{scope}[rotate=-45]
\draw[thin,fill,color=#1] (4.3,0) rectangle (0,-3);
\end{scope}
\end{scope}
}
\newcommand{\maskAd}[1]{
\begin{scope}
\clip (3.5,0) rectangle (0,-2.5);
\begin{scope}[rotate=45]
\draw[thin,fill,color=#1] (3,0) rectangle (0,-5);
\end{scope}
\end{scope}
}

\newcommand{\VerticalFan}{\begin{scope}[xscale=.7,yscale=.6,rotate=-5]
\coordinate (nI3) at (-1,1.8);
\coordinate (I3) at (1.58,-2.32);
\coordinate (I1) at (-1.58,-2.32);
\coordinate (nI1) at (1,1.8);
\coordinate (S2) at (0,-3.85);
\coordinate (nS2) at (0,1.4);
\coordinate (P2) at (0,-1.85);
\coordinate (nP2) at (0,2.85);
\coordinate (P1) at (-1.25,0);
\coordinate (nP1) at (2.75,-2);
\coordinate (P3) at (1.25,0);
\coordinate (nP3) at (-2.75,-2);
\maskA{gray!27!}
\maskX{gray!40}
\maskB{white}
\maskC{white}
\begin{scope}
\clip (-1,1.8) rectangle (1,2.1);
\draw[very thick,color=black] (0,-.42) circle[radius=2.43cm];
\end{scope}
\begin{scope}[rotate=-63,xshift=.55cm]
\clip (0,-4) rectangle (-3.1,1.5);
\draw[color=black,very thick] (0,-1.14) ellipse[x radius=3cm,y radius=2.235cm];
\end{scope}
\begin{scope}[rotate=63,xshift=-.55cm]
\clip (0,-4) rectangle (3.1,1.5);
\draw[color=black,very thick] (0,-1.14) ellipse[x radius=3cm,y radius=2.235cm];
\end{scope}
    	\draw[very thick] (-1.5,0.5) circle[radius=2.8cm];
	\draw[very thick] (1.5,0.5) circle[radius=2.8cm];
	\draw[very thick] (0,-2.5) circle[radius=2.8cm];
\draw[fill] (P1) ellipse[x radius=4.5pt,y radius=6pt] +(.8,-.3)node{\tiny$P_1[2]$};
\draw[fill] (P2) ellipse[x radius=3pt,y radius=4pt] +(0,-.2)node[below]{\tiny$P_2[2]$};
\draw[fill] (P3) ellipse[x radius=3pt,y radius=4pt] +(0,.2)node[ right]{\tiny$P_3[2]$};
\draw[fill] (nP1) ellipse[x radius=3pt,y radius=4pt] node[below right]{\tiny$P_1[1]$};
\draw[fill] (nP2) ellipse[x radius=3pt,y radius=4pt] +(0,.2)node[above]{\tiny$P_2[1]$};
\draw[fill] (nP3) ellipse[x radius=4.5pt,y radius=6pt] +(-.7,-.3)node{\tiny$P_3[1]$};
\draw[fill] (nI1) ellipse[x radius=3pt,y radius=4pt];
\draw[fill] (nI3) ellipse[x radius=4.5pt,y radius=6pt] +(-.3,.5)node{\tiny $I_3[1]$};
\draw[fill] (nS2) ellipse[x radius=3pt,y radius=4pt] ;
\draw (nI3) +(-1.5,.5)node{$Y$};
\draw (P1) +(-1,.5)node{$X$};
\end{scope}
} 

\newcommand{\allblack}{\begin{scope}[xscale=.7,yscale=.45,rotate=-5]
\coordinate (nI3) at (-1,1.8);
\coordinate (I3) at (1.58,-2.32);
\coordinate (I1) at (-1.58,-2.32);
\coordinate (nI1) at (1,1.8);
\coordinate (S2) at (0,-3.85);
\coordinate (nS2) at (0,1.4);
\coordinate (P2) at (0,-1.85);
\coordinate (nP2) at (0,2.85);
\coordinate (P1) at (-1.25,0);
\coordinate (nP1) at (2.75,-2);
\coordinate (P3) at (1.25,0);
\coordinate (nP3) at (-2.75,-2);
\draw (0,-6.5) node{$A_3\times {\color{blue}0}$};
\begin{scope}
\clip (1.5,0.5) circle[radius=2.8cm];
\clip (-1.5,0.5) circle[radius=2.8cm];
\maskC{gray!40} 
\end{scope}
\begin{scope}
\clip (-2,-2.3) rectangle (2,-3);
\draw[very thick,color=black] (0,-.45) circle[radius=2.43cm];
\end{scope}
\begin{scope}[rotate=-63,xshift=.55cm]
\clip (0,-4) rectangle (3.1,1.5);
\draw[color=black,very thick] (0,-1.14) ellipse[x radius=3cm,y radius=2.235cm];
\end{scope}
\begin{scope}[rotate=63,xshift=-.55cm]
\clip (0,-4) rectangle (-3.1,1.5);
\draw[color=black,very thick] (0,-1.14) ellipse[x radius=3cm,y radius=2.235cm];
\end{scope}
	\draw[very thick] (-1.5,0.5) circle[radius=2.8cm];
	\draw[very thick] (1.5,0.5) circle[radius=2.8cm];
	\draw[very thick] (0,-2.5) circle[radius=2.8cm];
\end{scope}
} 

\newcommand{\allblackB}{\begin{scope}[xscale=.7,yscale=.45,rotate=-5]
\coordinate (nI3) at (-1,1.8);
\coordinate (I3) at (1.58,-2.32);
\coordinate (I1) at (-1.58,-2.32);
\coordinate (nI1) at (1,1.8);
\coordinate (S2) at (0,-3.85);
\coordinate (nS2) at (0,1.4);
\coordinate (P2) at (0,-1.85);
\coordinate (nP2) at (0,2.85);
\coordinate (P1) at (-1.25,0);
\coordinate (nP1) at (2.75,-2);
\coordinate (P3) at (1.25,0);
\coordinate (nP3) at (-2.75,-2);
\begin{scope}
\clip (1.5,0.5) circle[radius=2.8cm];
\clip (-1.5,0.5) circle[radius=2.8cm];
\maskC{gray!50} 
\end{scope}
\begin{scope}
\clip (-2,-2.3) rectangle (2,-3);
\draw[very thick,color=black] (0,-.45) circle[radius=2.43cm];
\end{scope}
\begin{scope}[rotate=-63,xshift=.55cm]
\clip (0,-4) rectangle (3.1,1.5);
\draw[color=black,very thick] (0,-1.14) ellipse[x radius=3cm,y radius=2.235cm];
\end{scope}
\begin{scope}[rotate=63,xshift=-.55cm]
\clip (0,-4) rectangle (-3.1,1.5);
\draw[color=black,very thick] (0,-1.14) ellipse[x radius=3cm,y radius=2.235cm];
\end{scope}
	\draw[very thick] (-1.5,0.5) circle[radius=2.8cm];
	\draw[very thick] (1.5,0.5) circle[radius=2.8cm];
	\draw[very thick] (0,-2.5) circle[radius=2.8cm];
\end{scope}
} 

\newcommand{\allblackC}{\begin{scope}[xscale=.7,yscale=.45,rotate=-5]
\coordinate (nI3) at (-1,1.8);
\coordinate (I3) at (1.58,-2.32);
\coordinate (I1) at (-1.58,-2.32);
\coordinate (nI1) at (1,1.8);
\coordinate (S2) at (0,-3.85);
\coordinate (nS2) at (0,1.4);
\coordinate (P2) at (0,-1.85);
\coordinate (nP2) at (0,2.85);
\coordinate (P1) at (-1.25,0);
\coordinate (nP1) at (2.75,-2);
\coordinate (P3) at (1.25,0);
\coordinate (nP3) at (-2.75,-2);
\begin{scope}
\clip (0,-2.5) circle[radius=2.8cm];
\begin{scope}[rotate=63,xshift=-.55cm]
\draw[thin,fill,color=gray!20!white] (0,-1.14) ellipse[x radius=3cm,y radius=2.235cm];
\end{scope}
\end{scope}
\begin{scope}
\draw[thin,fill,color=white] 
(0,.5) ellipse[x radius=3cm,y radius=2.35cm]
(-1.5,0.5) circle[radius=2.8cm]; 
\end{scope}
\begin{scope}
\clip (0,-3) rectangle (4,3); 
\draw[very thick] 
(0,.5) ellipse[x radius=3cm,y radius=2.35cm]; \end{scope}
\begin{scope}
\clip (.8,-2.2) rectangle (2.35,-1); 
\draw[very thick] (-.5,-.3) ellipse[x radius=3cm,y radius=2.15cm]; 
\end{scope}
\begin{scope}
\clip (1.5,0.5) circle[radius=2.8cm];
\clip (-1.5,0.5) circle[radius=2.8cm];
\end{scope}
\begin{scope}
\clip (0,-4) rectangle (2.5,-1);
\draw[very thick] (0,-1) ellipse[x radius=2.35cm,y radius=2.85cm];
\end{scope}
\begin{scope}
\clip (-1.6,-4) rectangle (2.15,-2);
\draw[very thick] (0.25,-.25) ellipse[x radius=2.5cm,y radius=3cm];
\end{scope}
\begin{scope}
\clip (-1.6,-3.1) rectangle (1.7,-1.5);
\draw[very thick] (-.5,-1.1) ellipse[x radius=2.2cm,y radius=1.4cm];
\end{scope}
\begin{scope}[rotate=-63,xshift=.55cm]
\clip (0,-4) rectangle (3.1,1.5);
\draw[color=black,very thick] (0,-1.14) ellipse[x radius=3cm,y radius=2.235cm];
\end{scope}
\begin{scope}[rotate=63,xshift=-.55cm]
\clip (0,-4) rectangle (-3.1,1.5);
\draw[color=black,very thick] (0,-1.14) ellipse[x radius=3cm,y radius=2.235cm];
\end{scope}
	\draw[very thick] (-1.5,0.5) circle[radius=2.8cm];
	\draw[very thick] (1.5,0.5) circle[radius=2.8cm];
	\draw[very thick] (0,-2.5) circle[radius=2.8cm];
\begin{scope} 
\clip (-1.6,-3.1) rectangle (3,-1);
\draw[very thick,color=red] (0,-.45) circle[radius=2.43cm];
\end{scope}
\draw[very thick,dashed,color=green!70!black,->] (5,0) --(.3,-.8);
\draw[color=green!70!black] (5,0) node[right]{A};
\draw[color=green!70!black] (4.5,2.5) node[right]{B};
\draw[color=green!70!black] (-4.5,2.5) node[left]{C};
\draw[color=green!70!black] (-5,-1.5) node[left]{D};
\draw[color=green!70!black] (-4,-4) node[left]{E};
\draw[very thick,dashed,color=green!70!black,->] (-5,-1.5) --(-.3,-.8);
\draw[very thick,dashed,color=green!70!black,->] (4.5,2.5)..controls (2,2) and (1,1.5)..(.5,-.3);
\draw[very thick,dashed,color=green!70!black,->] (-4.5,2.5)..controls (-2,2) and (-1,1.5)..(-.5,-.3);
\draw[very thick,dashed,color=green!70!black,->] (-4,-4)..controls (-2,-2) and (-1,-1.3)..(-.2,-1.2);
\draw[very thick, color=blue,dotted] (.4,-1.2) -- (3.7,-4.3) node[right]{F};
\end{scope}
} 

\newcommand{\PonePerp}{\begin{scope}[xscale=.7,yscale=.45,rotate=-5]
\coordinate (nI3) at (-1,1.8);
\coordinate (I3) at (1.58,-2.32);
\coordinate (I1) at (-1.58,-2.32);
\coordinate (nI1) at (1,1.8);
\coordinate (S2) at (0,-3.85);
\coordinate (nS2) at (0,1.4);
\coordinate (P2) at (0,-1.85);
\coordinate (nP2) at (0,2.85);
\coordinate (P1) at (-1.25,0);
\coordinate (nP1) at (2.75,-2);
\coordinate (P3) at (1.25,0);
\coordinate (nP3) at (-2.75,-2);
\draw (0.5,-7) node{$P_1^\perp\times {\color{blue}P_1}$};
\begin{scope}
\clip (1.5,0.5) circle[radius=2.8cm];
\maskC{gray!40} 
\maskA{white}
\end{scope}
\begin{scope}[rotate=63,xshift=-.55cm]
\clip (0,-4) rectangle (-3.1,1.5);
\draw[color=black,very thick] (0,-1.14) ellipse[x radius=3cm,y radius=2.235cm];
\end{scope}
	\draw[color=blue,very thick] (-1.5,0.5) circle[radius=2.8cm];
	\draw[very thick] (1.5,0.5) circle[radius=2.8cm];
	\draw[very thick] (0,-2.5) circle[radius=2.8cm];
\draw[color=blue,fill] (P1) ellipse[x radius=9pt, y radius=13pt] +(.2,-.3)node[above left]{\tiny$P_1[0]$};
\draw[color=blue,fill] (nP1) ellipse[x radius=9pt, y radius=13pt] +(-.2,.1)node[below right]{\tiny$P_1[1]$};
\end{scope}
} 

\newcommand{\StwoPerp}{\begin{scope}[xscale=.7,yscale=.45,rotate=-5]
\coordinate (nI3) at (-1,1.8);
\coordinate (I3) at (1.58,-2.32);
\coordinate (I1) at (-1.58,-2.32);
\coordinate (nI1) at (1,1.8);
\coordinate (S2) at (0,-3.85);
\coordinate (nS2) at (0,1.4);
\coordinate (P2) at (0,-1.85);
\coordinate (nP2) at (0,2.85);
\coordinate (P1) at (-1.25,0);
\coordinate (nP1) at (2.75,-2);
\coordinate (P3) at (1.25,0);
\coordinate (nP3) at (-2.75,-2);
\draw (0,-6.5) node{$S_2^\perp\times {\color{blue}S_2}$};
\begin{scope}[rotate=-63,xshift=.55cm]
\clip (0,-1.14) ellipse[x radius=3cm,y radius=2.235cm];
\begin{scope}[rotate=63,xshift=-.55cm,rotate=63]
\draw[fill,color=gray!40] (0,-1.14) ellipse[x radius=3cm,y radius=2.235cm];
\end{scope}
\begin{scope}[xshift=-.55cm,rotate=63]
\draw[fill,color=white] (0,-2.5) circle[radius=2.8cm];
\end{scope}
\end{scope}
\begin{scope}[rotate=-63,xshift=.55cm]
\draw[color=black,very thick] (0,-1.14) ellipse[x radius=3cm,y radius=2.235cm];
\end{scope}
\begin{scope}[rotate=63,xshift=-.55cm]
\draw[color=black,very thick] (0,-1.14) ellipse[x radius=3cm,y radius=2.235cm];
\end{scope}
	\draw[very thick, dotted, color=red] (-1.5,0.5) circle[radius=2.8cm];
	\draw[very thick, dotted, color=red] (1.5,0.5) circle[radius=2.8cm];
	\draw[color=blue,very thick] (0,-2.5) circle[radius=2.8cm];
\draw[color=blue,fill] (S2) ellipse[x radius=9pt, y radius=13pt] node[above]{\tiny$S_2[0]$};
\draw[color=white,fill] (nS2)+(-.2,1.8) ellipse[x radius=35pt, y radius=20pt];
\draw[color=blue,fill] (nS2) ellipse[x radius=9pt, y radius=13pt] ;
\draw[color=blue] (nP2) node{\tiny $S_2[1]$};
\end{scope}
} 

\newcommand{\PthreePerp}{\begin{scope}[xscale=.7,yscale=.45,rotate=-5]
\coordinate (nI3) at (-1,1.8);
\coordinate (I3) at (1.58,-2.32);
\coordinate (I1) at (-1.58,-2.32);
\coordinate (nI1) at (1,1.8);
\coordinate (S2) at (0,-3.85);
\coordinate (nS2) at (0,1.4);
\coordinate (P2) at (0,-1.85);
\coordinate (nP2) at (0,2.85);
\coordinate (P1) at (-1.25,0);
\coordinate (nP1) at (2.75,-2);
\coordinate (P3) at (1.25,0);
\coordinate (nP3) at (-2.75,-2);
\draw (0.5,-7) node{$P_3^\perp\times {\color{blue}P_3}$};
\begin{scope}
\clip (-1.5,0.5) circle[radius=2.8cm];
\maskC{gray!40} 
\maskB{white}
\end{scope}
\begin{scope}[rotate=-63,xshift=.55cm]
\clip (0,-4) rectangle (3.1,1.5);
\draw[color=black,very thick] (0,-1.14) ellipse[x radius=3cm,y radius=2.235cm];
\end{scope}
	\draw[very thick] (-1.5,0.5) circle[radius=2.8cm];
	\draw[very thick,color=blue] (1.5,0.5) circle[radius=2.8cm];
	\draw[very thick] (0,-2.5) circle[radius=2.8cm];
\draw[color=blue,fill] (P3) ellipse[x radius=9pt, y radius=13pt] +(0,-.5)node[above right]{\tiny$P_3[0]$};
\draw[color=blue,fill] (nP3) ellipse[x radius=9pt, y radius=13pt] +(0,0.5)node[below left]{\tiny$P_3[1]$};
\end{scope}
} 

\newcommand{\PthreePerpB}{\begin{scope}[xscale=.7,yscale=.45,rotate=-5]
\coordinate (nI3) at (-1,1.8);
\coordinate (I3) at (1.58,-2.32);
\coordinate (I1) at (-1.58,-2.32);
\coordinate (nI1) at (1,1.8);
\coordinate (S2) at (0,-3.85);
\coordinate (nS2) at (0,1.4);
\coordinate (P2) at (0,-1.85);
\coordinate (nP2) at (0,2.85);
\coordinate (P1) at (-1.25,0);
\coordinate (nP1) at (2.75,-2);
\coordinate (P3) at (1.25,0);
\coordinate (nP3) at (-2.75,-2);
\begin{scope}
	\clip (-1.5,0.5) circle[radius=2.8cm];
	\maskC{gray!50} 
	\maskB{white}
\end{scope}
\begin{scope}[rotate=-63,xshift=.55cm]
	\clip (0,-4) rectangle (3.1,1.5);
	\draw[color=black,very thick] (0,-1.14) ellipse[x radius=3cm,y radius=2.235cm];
\end{scope}
	\draw[very thick] (-1.5,0.5) circle[radius=2.8cm];
	\draw[very thick,color=blue] (1.5,0.5) circle[radius=2.8cm];
	\draw[very thick] (0,-2.5) circle[radius=2.8cm];
\draw[color=blue,fill] (P3) ellipse[x radius=6pt, y radius=9pt] +(0,-.5)node[above right]{\tiny$P_3[0]$};
\draw[color=blue,fill] (nP3) ellipse[x radius=6pt, y radius=9pt] +(0,0.5)node[below left]{\tiny$P_3[1]$};
\end{scope}
} 

\newcommand{\IonePerp}{\begin{scope}[xscale=.7,yscale=.45,rotate=-5]
\coordinate (nI3) at (-1,1.8);
\coordinate (I3) at (1.58,-2.32);
\coordinate (I1) at (-1.58,-2.32);
\coordinate (nI1) at (1,1.8);
\coordinate (S2) at (0,-3.85);
\coordinate (nS2) at (0,1.4);
\coordinate (P2) at (0,-1.85);
\coordinate (nP2) at (0,2.85);
\coordinate (P1) at (-1.25,0);
\coordinate (nP1) at (2.75,-2);
\coordinate (P3) at (1.25,0);
\coordinate (nP3) at (-2.75,-2);
\draw (0.5,-7) node{$I_1^\perp\times {\color{blue}I_1}$};
\begin{scope}
\clip (-1.5,0.5) circle[radius=2.8cm];
\maskY{gray!40} 
\maskX{white}
\end{scope}
\begin{scope}
\clip (-2,-2.3) rectangle (2,-3);
\draw[very thick,color=black] (0,-.45) circle[radius=2.43cm];
\end{scope}
\begin{scope}
\clip (1,2) rectangle (3,-2.31);
\draw[very thick,color=black] (0,-.42) circle[radius=2.44cm];
\end{scope}
\begin{scope}[rotate=-63,xshift=.55cm]
\draw[color=blue,very thick] (0,-1.14) ellipse[x radius=3cm,y radius=2.235cm];
\end{scope}
\begin{scope}[rotate=63,xshift=-.55cm]
\draw[color=black,very thick] (0,-1.14) ellipse[x radius=3cm,y radius=2.235cm];
\end{scope}
	\draw[very thick, dotted, color=red] (1.5,0.5) circle[radius=2.8cm];
	\draw[very thick] (-1.5,0.5) circle[radius=2.8cm];
	\draw[very thick, dotted, color=red] (0,-2.5) circle[radius=2.8cm];
\draw[color=blue,fill] (I1) ellipse[x radius=9pt, y radius=13pt] node[below left]{\tiny$I_1[0]$};
\draw[color=blue,fill] (nI1) ellipse[x radius=9pt, y radius=13pt] ;\end{scope}
} 

\newcommand{\IthreePerp}{\begin{scope}[xscale=.7,yscale=.45,rotate=-5]
\coordinate (nI3) at (-1,1.8);
\coordinate (I3) at (1.58,-2.32);
\coordinate (I1) at (-1.58,-2.32);
\coordinate (nI1) at (1,1.8);
\coordinate (S2) at (0,-3.85);
\coordinate (nS2) at (0,1.4);
\coordinate (P2) at (0,-1.85);
\coordinate (nP2) at (0,2.85);
\coordinate (P1) at (-1.25,0);
\coordinate (nP1) at (2.75,-2);
\coordinate (P3) at (1.25,0);
\coordinate (nP3) at (-2.75,-2);
\draw (0.5,-7) node{$I_3^\perp\times {\color{blue}I_3}$};
\begin{scope}
\clip (1.5,0.5) circle[radius=2.8cm];
\maskX{gray!40} 
\maskY{white}
\end{scope}
\begin{scope}
\clip (-2,-2.3) rectangle (2,-3);
\draw[very thick,color=black] (0,-.45) circle[radius=2.43cm];
\end{scope}
\begin{scope}
\clip (-1,2) rectangle (-3,-2.31);
\draw[very thick,color=black] (0,-.42) circle[radius=2.44cm];
\end{scope}
\begin{scope}[rotate=-63,xshift=.55cm]
\draw[color=black,very thick] (0,-1.14) ellipse[x radius=3cm,y radius=2.235cm];
\end{scope}
\begin{scope}[rotate=63,xshift=-.55cm]
\draw[color=blue,very thick] (0,-1.14) ellipse[x radius=3cm,y radius=2.235cm];
\end{scope}
	\draw[very thick, dotted, color=red] (-1.5,0.5) circle[radius=2.8cm];
	\draw[very thick] (1.5,0.5) circle[radius=2.8cm];
	\draw[very thick, dotted, color=red] (0,-2.5) circle[radius=2.8cm];
\draw[color=blue,fill] (I3) ellipse[x radius=9pt, y radius=13pt] node[below right]{\tiny$I_3[0]$};
\draw[color=blue,fill] (nI3) ellipse[x radius=9pt, y radius=13pt] ;\end{scope}
} 

\newcommand{\Ione}{\begin{scope}[xscale=.7,yscale=.45,rotate=-5]
\coordinate (nI3) at (-1,1.8);
\coordinate (I3) at (1.58,-2.32);
\coordinate (I1) at (-1.58,-2.32);
\coordinate (nI1) at (1,1.8);
\coordinate (S2) at (0,-3.85);
\coordinate (nS2) at (0,1.4);
\coordinate (P2) at (0,-1.85);
\coordinate (nP2) at (0,2.85);
\coordinate (P1) at (-1.25,0);
\coordinate (nP1) at (2.75,-2);
\coordinate (P3) at (1.25,0);
\coordinate (nP3) at (-2.75,-2);
\draw (0.5,-6.5) node{$I_1\times {\color{blue}\,^\perp I_1}$};
\maskX{gray!40} 
\maskB{white}
\maskC{white}
\begin{scope}[rotate=-63,xshift=.55cm]
\draw[color=black,very thick] (0,-1.14) ellipse[x radius=3cm,y radius=2.235cm];
\end{scope}
\begin{scope}[rotate=63,xshift=-.55cm]
\clip (0,-4) rectangle (-3.1,1.5);
\draw[color=blue,very thick] (0,-1.14) ellipse[x radius=3cm,y radius=2.235cm];
\end{scope}
	\draw[color=blue,very thick] (1.5,0.5) circle[radius=2.8cm];
	\draw[very thick, dotted, color=red] (-1.5,0.5) circle[radius=2.8cm];
	\draw[color=blue,very thick] (0,-2.5) circle[radius=2.8cm];
\draw[fill] (P1) ellipse[x radius=9pt, y radius=13pt];
\draw[fill] (nP1) ellipse[x radius=9pt, y radius=13pt] node[below right]{\tiny$P_1[3]$};
\end{scope}
} 

\newcommand{\IoneB}{\begin{scope}[xscale=.7,yscale=.45,rotate=-5]
\coordinate (nI3) at (-1,1.8);
\coordinate (I3) at (1.58,-2.32);
\coordinate (I1) at (-1.58,-2.32);
\coordinate (nI1) at (1,1.8);
\coordinate (S2) at (0,-3.85);
\coordinate (nS2) at (0,1.4);
\coordinate (P2) at (0,-1.85);
\coordinate (nP2) at (0,2.85);
\coordinate (P1) at (-1.25,0);
\coordinate (nP1) at (2.75,-2);
\coordinate (P3) at (1.25,0);
\coordinate (nP3) at (-2.75,-2);
\maskX{gray!35} 
\maskB{white}
\maskC{white}
\begin{scope}[rotate=-63,xshift=.55cm]
\draw[color=black,very thick] (0,-1.14) ellipse[x radius=3cm,y radius=2.235cm];
\end{scope}
\begin{scope}[rotate=63,xshift=-.55cm]
\clip (0,-4) rectangle (-3.1,1.5);
\draw[color=blue,very thick] (0,-1.14) ellipse[x radius=3cm,y radius=2.235cm];
\end{scope}
	\draw[color=blue,very thick] (1.5,0.5) circle[radius=2.8cm];
	\draw[very thick, dotted, color=red] (-1.5,0.5) circle[radius=2.8cm];
	\draw[color=blue,very thick] (0,-2.5) circle[radius=2.8cm];
\draw[fill] (P1) ellipse[x radius=6pt, y radius=9pt];
\draw[fill] (nP1) ellipse[x radius=6pt, y radius=9pt] node[below right]{\tiny$P_1[3]$}; 
\end{scope}
} 

\newcommand{\Ithree}{\begin{scope}[xscale=.7,yscale=.45,rotate=-5]
\coordinate (nI3) at (-1,1.8);
\coordinate (I3) at (1.58,-2.32);
\coordinate (I1) at (-1.58,-2.32);
\coordinate (nI1) at (1,1.8);
\coordinate (S2) at (0,-3.85);
\coordinate (nS2) at (0,1.4);
\coordinate (P2) at (0,-1.85);
\coordinate (nP2) at (0,2.85);
\coordinate (P1) at (-1.25,0);
\coordinate (nP1) at (2.75,-2);
\coordinate (P3) at (1.25,0);
\coordinate (nP3) at (-2.75,-2);
\draw (0.5,-6.5) node{$I_3\times {\color{blue}\,^\perp I_3}$};
\maskY{gray!40} 
\maskA{white}
\maskC{white}
\begin{scope}[rotate=-63,xshift=.55cm]
\clip (0,-4) rectangle (3.1,1.5);
\draw[color=blue,very thick] (0,-1.14) ellipse[x radius=3cm,y radius=2.235cm];
\end{scope}
\begin{scope}[rotate=63,xshift=-.55cm]
\draw[color=black,very thick] (0,-1.14) ellipse[x radius=3cm,y radius=2.235cm];
\end{scope}
	\draw[very thick,color=blue] (-1.5,0.5) circle[radius=2.8cm];
	\draw[very thick, dotted, color=red] (1.5,0.5) circle[radius=2.8cm];
	\draw[very thick,color=blue] (0,-2.5) circle[radius=2.8cm];
\draw[fill] (P3) ellipse[x radius=9pt, y radius=13pt] ;
\draw[fill] (nP3) ellipse[x radius=9pt, y radius=13pt] +(0,.2)node[below left]{\tiny$P_3[3]$};
\end{scope}
} 

\newcommand{\Ptwo}{\begin{scope}[xscale=.7,yscale=.45,rotate=-5]
\coordinate (nI3) at (-1,1.8);
\coordinate (I3) at (1.58,-2.32);
\coordinate (I1) at (-1.58,-2.32);
\coordinate (nI1) at (1,1.8);
\coordinate (S2) at (0,-3.85);
\coordinate (nS2) at (0,1.4);
\coordinate (P2) at (0,-1.85);
\coordinate (nP2) at (0,2.85);
\coordinate (P1) at (-1.25,0);
\coordinate (nP1) at (2.75,-2);
\coordinate (P3) at (1.25,0);
\coordinate (nP3) at (-2.75,-2);
\draw (0.5,-7.2) node{$P_2\times {\color{blue}\,^\perp P_2}$};
\maskZ{gray!40} 
\maskY{white}
\maskX{white}
\begin{scope}
\draw[very thick,color=black] (0,-.45) circle[radius=2.43cm];
\end{scope}
\begin{scope}[rotate=-63,xshift=.55cm]
\clip (0,-4) rectangle (3.1,1.5);
\draw[color=blue,very thick] (0,-1.14) ellipse[x radius=3cm,y radius=2.235cm];
\end{scope}
\begin{scope}[rotate=-63,xshift=.55cm]
\clip (0,-4) rectangle (-3.1,1.5);
\draw[color=blue,very thick] (0,-1.14) ellipse[x radius=3cm,y radius=2.235cm];
\end{scope}
\begin{scope}[rotate=63,xshift=-.55cm]
\clip (0,-4) rectangle (-3.1,1.5);
\draw[color=blue,very thick] (0,-1.14) ellipse[x radius=3cm,y radius=2.235cm];
\end{scope}
\begin{scope}[rotate=63,xshift=-.55cm]
\clip (0,-4) rectangle (3.1,1.5);
\draw[color=blue,very thick] (0,-1.14) ellipse[x radius=3cm,y radius=2.235cm];
\end{scope}
	\draw[very thick, dotted, color=red] (-1.5,0.5) circle[radius=2.8cm] (1.5,0.5) circle[radius=2.8cm] (0,-2.5) circle[radius=2.8cm];
\draw[fill] (S2) ellipse[x radius=9pt, y radius=13pt] node[below]{\tiny$S_2[2]$};
\draw[fill] (nS2) ellipse[x radius=9pt, y radius=13pt];
\end{scope}
} 

\newcommand{\PtwoPerp}{\begin{scope}[xscale=.7,yscale=.45,rotate=-5]
\coordinate (nI3) at (-1,1.8);
\coordinate (I3) at (1.58,-2.32);
\coordinate (I1) at (-1.58,-2.32);
\coordinate (nI1) at (1,1.8);
\coordinate (S2) at (0,-3.85);
\coordinate (nS2) at (0,1.4);
\coordinate (P2) at (0,-1.85);
\coordinate (nP2) at (0,2.85);
\coordinate (P1) at (-1.25,0);
\coordinate (nP1) at (2.75,-2);
\coordinate (P3) at (1.25,0);
\coordinate (nP3) at (-2.75,-2);
\draw (0.3,-6.5) node{$P_2^\perp\times {\color{blue} P_2}$};
\begin{scope}
\clip (-1.5,0.5) circle[radius=2.8cm];
\maskB{gray!40} 
\maskZ{white}
\end{scope}
\begin{scope}
\draw[very thick,color=blue] (0,-.45) circle[radius=2.43cm];
\end{scope}
	\draw[very thick] (-1.5,0.5) circle[radius=2.8cm];
	\draw[very thick] (1.5,0.5) circle[radius=2.8cm];
	\draw[very thick, dotted, color=red] (0,-2.5) circle[radius=2.8cm];
\draw[color=blue,fill] (nP2) ellipse[x radius=9pt, y radius=13pt] ;\draw[color=blue,fill] (P2) ellipse[x radius=9pt, y radius=13pt] +(0,-.4)node[below]{\tiny$P_2[0]$};
\end{scope}
} 

\newcommand{\Pthree}{\begin{scope}[xscale=.7,yscale=.45,rotate=-5]
\coordinate (nI3) at (-1,1.8);
\coordinate (I3) at (1.58,-2.32);
\coordinate (I1) at (-1.58,-2.32);
\coordinate (nI1) at (1,1.8);
\coordinate (S2) at (0,-3.85);
\coordinate (nS2) at (0,1.4);
\coordinate (P2) at (0,-1.85);
\coordinate (nP2) at (0,2.85);
\coordinate (P1) at (-1.25,0);
\coordinate (nP1) at (2.75,-2);
\coordinate (P3) at (1.25,0);
\coordinate (nP3) at (-2.75,-2);
\draw (0.3,-6.7) node{$P_3\times {\color{blue}\,^\perp P_3}$};
\maskB{gray!40} 
\maskA{white}
\maskY{white}
\begin{scope}
\clip (-2,-2.3) rectangle (2,-3);
\draw[very thick,color=blue] (0,-.45) circle[radius=2.43cm];
\end{scope}
\begin{scope}
\clip (1,2) rectangle (3,-2.31);
\draw[very thick,color=blue] (0,-.42) circle[radius=2.44cm];
\end{scope}
\begin{scope}[rotate=63,xshift=-.55cm]
\clip (0,-4) rectangle (-3.1,1.5);
\draw[color=blue,very thick] (0,-1.14) ellipse[x radius=3cm,y radius=2.235cm];
\end{scope}
\begin{scope}[rotate=63,xshift=-.55cm]
\clip (0,-4) rectangle (3.1,1.5);
\draw[color=blue,very thick] (0,-1.14) ellipse[x radius=3cm,y radius=2.235cm];
\end{scope}
	\draw[very thick,color=blue] (-1.5,0.5) circle[radius=2.8cm];
	\draw[very thick] (1.5,0.5) circle[radius=2.8cm];
	\draw[very thick, dotted, color=red] (0,-2.5) circle[radius=2.8cm];
\draw[fill] (I1) ellipse[x radius=9pt, y radius=13pt] +(.4,0)node[below left]{\tiny$I_1[2]$};
\draw[fill] (nI1) ellipse[x radius=9pt, y radius=13pt] ;
\end{scope}
} 

\newcommand{\Pone}{\begin{scope}[xscale=.7,yscale=.45,rotate=-5]
\coordinate (nI3) at (-1,1.8);
\coordinate (I3) at (1.58,-2.32);
\coordinate (I1) at (-1.58,-2.32);
\coordinate (nI1) at (1,1.8);
\coordinate (S2) at (0,-3.85);
\coordinate (nS2) at (0,1.4);
\coordinate (P2) at (0,-1.85);
\coordinate (nP2) at (0,2.85);
\coordinate (P1) at (-1.25,0);
\coordinate (nP1) at (2.75,-2);
\coordinate (P3) at (1.25,0);
\coordinate (nP3) at (-2.75,-2);
\draw (1,-6.5) node{$P_1\times {\color{blue}\,^\perp P_1}$};
\maskA{gray!40} 
\maskB{white}
\maskX{white}
\begin{scope}
\clip (-2,-2.3) rectangle (2,-3);
\draw[very thick,color=blue] (0,-.45) circle[radius=2.43cm];
\end{scope}
\begin{scope}
\clip (-1,2) rectangle (-3,-2.31);
\draw[very thick,color=blue] (0,-.42) circle[radius=2.44cm];
\end{scope}
\begin{scope}[rotate=-63,xshift=.55cm]
\draw[color=blue,very thick] (0,-1.14) ellipse[x radius=3cm,y radius=2.235cm];
\end{scope}%
	\draw[very thick] (-1.5,0.5) circle[radius=2.8cm];
	\draw[very thick,color=blue] (1.5,0.5) circle[radius=2.8cm];
	\draw[very thick, dotted, color=red] (0,-2.5) circle[radius=2.8cm];
\draw[fill] (I3) ellipse[x radius=9pt, y radius=13pt] +(-.5,0.1)node[below right]{\tiny$I_3[2]$};
\draw[fill] (nI3) ellipse[x radius=9pt, y radius=13pt];
\end{scope}
} 

\newcommand{\PoneB}{\begin{scope}[xscale=.7,yscale=.45,rotate=-5]
\coordinate (nI3) at (-1,1.8);
\coordinate (I3) at (1.58,-2.32);
\coordinate (I1) at (-1.58,-2.32);
\coordinate (nI1) at (1,1.8);
\coordinate (S2) at (0,-3.85);
\coordinate (nS2) at (0,1.4);
\coordinate (P2) at (0,-1.85);
\coordinate (nP2) at (0,2.85);
\coordinate (P1) at (-1.25,0);
\coordinate (nP1) at (2.75,-2);
\coordinate (P3) at (1.25,0);
\coordinate (nP3) at (-2.75,-2);
\maskA{gray!45} 
\maskB{white}
\maskX{white}
\begin{scope}
\clip (-2,-2.3) rectangle (2,-3);
\draw[very thick,color=blue] (0,-.45) circle[radius=2.43cm];
\end{scope}
\begin{scope}
\clip (-1,2) rectangle (-3,-2.31);
\draw[very thick,color=blue] (0,-.42) circle[radius=2.44cm];
\end{scope}
\begin{scope}[rotate=-63,xshift=.55cm]
\draw[color=blue,very thick] (0,-1.14) ellipse[x radius=3cm,y radius=2.235cm];
\end{scope}%
	\draw[very thick] (-1.5,0.5) circle[radius=2.8cm];
	\draw[very thick,color=blue] (1.5,0.5) circle[radius=2.8cm];
	\draw[very thick, dotted, color=red] (0,-2.5) circle[radius=2.8cm];
\draw[fill] (I3) ellipse[x radius=6pt, y radius=9pt] +(-.5,0.1)node[below right]{\tiny$I_3[2]$}; 
\draw[fill] (nI3) ellipse[x radius=6pt, y radius=9pt];
\end{scope}
} 

\newcommand{\Stwo}{\begin{scope}[xscale=.7,yscale=.45,rotate=-5]
\coordinate (nI3) at (-1,1.8);
\coordinate (I3) at (1.58,-2.32);
\coordinate (I1) at (-1.58,-2.32);
\coordinate (nI1) at (1,1.8);
\coordinate (S2) at (0,-3.85);
\coordinate (nS2) at (0,1.4);
\coordinate (P2) at (0,-1.85);
\coordinate (nP2) at (0,2.85);
\coordinate (P1) at (-1.25,0);
\coordinate (nP1) at (2.75,-2);
\coordinate (P3) at (1.25,0);
\coordinate (nP3) at (-2.75,-2);
\draw (0.5,-7) node{$S_2\times {\color{blue}\,^\perp S_2}$};
\maskC{gray!40} 
\maskA{white}
\maskB{white}
	\draw[very thick, color=blue] (-1.5,0.5) circle[radius=2.8cm];
	\draw[very thick, color=blue] (1.5,0.5) circle[radius=2.8cm];
	\draw[very thick] (0,-2.5) circle[radius=2.8cm];
\draw[fill] (P2) ellipse[x radius=9pt, y radius=13pt] node[below]{\tiny$P_2[2]$};
\draw[fill] (nP2) ellipse[x radius=9pt, y radius=13pt];
\end{scope}
} 

\newcommand{\allblue}{\begin{scope}[xscale=.7,yscale=.45,rotate=-5]
\coordinate (nI3) at (-1,1.8);
\coordinate (I3) at (1.58,-2.32);
\coordinate (I1) at (-1.58,-2.32);
\coordinate (nI1) at (1,1.8);
\coordinate (S2) at (0,-3.85);
\coordinate (nS2) at (0,1.4);
\coordinate (P2) at (0,-1.85);
\coordinate (nP2) at (0,2.85);
\coordinate (P1) at (-1.25,0);
\coordinate (nP1) at (2.75,-2);
\coordinate (P3) at (1.25,0);
\coordinate (nP3) at (-2.75,-2);
\draw[fill,color=gray!40] (-5,-6) rectangle (5,4);
\maskA{white}
\maskB{white}
\maskC{white}
\draw (0.7,-6.5) node{$0\times {\color{blue}A_3}$};
\begin{scope}
\clip (-2,-2.3) rectangle (2,-3);
\draw[very thick,color=blue] (0,-.45) circle[radius=2.43cm];
\end{scope}
\begin{scope}[rotate=-63,xshift=.55cm]
\clip (0,-4) rectangle (3.1,1.5);
\draw[color=blue,very thick] (0,-1.14) ellipse[x radius=3cm,y radius=2.235cm];
\end{scope}
\begin{scope}[rotate=63,xshift=-.55cm]
\clip (0,-4) rectangle (-3.1,1.5);
\draw[color=blue,very thick] (0,-1.14) ellipse[x radius=3cm,y radius=2.235cm];
\end{scope}
	\draw[color=blue,very thick] (-1.5,0.5) circle[radius=2.8cm];
	\draw[color=blue,very thick] (1.5,0.5) circle[radius=2.8cm];
	\draw[color=blue,very thick] (0,-2.5) circle[radius=2.8cm];
\end{scope}
} 

\newcommand{\allblueB}{\begin{scope}[xscale=.7,yscale=.45,rotate=-5]
\coordinate (nI3) at (-1,1.8);
\coordinate (I3) at (1.58,-2.32);
\coordinate (I1) at (-1.58,-2.32);
\coordinate (nI1) at (1,1.8);
\coordinate (S2) at (0,-3.85);
\coordinate (nS2) at (0,1.4);
\coordinate (P2) at (0,-1.85);
\coordinate (nP2) at (0,2.85);
\coordinate (P1) at (-1.25,0);
\coordinate (nP1) at (2.75,-2);
\coordinate (P3) at (1.25,0);
\coordinate (nP3) at (-2.75,-2);
\draw[fill,color=gray!35] (-5,-6) rectangle (5,4);
\maskA{white}
\maskB{white}
\maskC{white}
\begin{scope}
\clip (-2,-2.3) rectangle (2,-3);
\draw[very thick,color=blue] (0,-.45) circle[radius=2.43cm];
\end{scope}
\begin{scope}[rotate=-63,xshift=.55cm]
\clip (0,-4) rectangle (3.1,1.5);
\draw[color=blue,very thick] (0,-1.14) ellipse[x radius=3cm,y radius=2.235cm];
\end{scope}
\begin{scope}[rotate=63,xshift=-.55cm]
\clip (0,-4) rectangle (-3.1,1.5);
\draw[color=blue,very thick] (0,-1.14) ellipse[x radius=3cm,y radius=2.235cm];
\end{scope}
	\draw[color=blue,very thick] (-1.5,0.5) circle[radius=2.8cm];
	\draw[color=blue,very thick] (1.5,0.5) circle[radius=2.8cm];
	\draw[color=blue,very thick] (0,-2.5) circle[radius=2.8cm];
\end{scope}
} 

\newcommand{\IoneBig}{\begin{scope}[xscale=.7,yscale=.5,rotate=-5]
\coordinate (nI3) at (-1,1.8);
\coordinate (I3) at (1.58,-2.32);
\coordinate (I1) at (-1.58,-2.32);
\coordinate (nI1) at (1,1.8);
\coordinate (S2) at (0,-3.85);
\coordinate (nS2) at (0,1.4);
\coordinate (P2) at (0,-1.85);
\coordinate (nP2) at (0,2.85);
\coordinate (P1) at (-1.25,0);
\coordinate (nP1) at (2.75,-2);
\coordinate (P3) at (1.25,0);
\coordinate (nP3) at (-2.75,-2);
\draw (0.5,-6.5) node{$I_1\times {\color{blue}\,^\perp I_1}$};
\maskX{gray!40} 
\maskB{white}
\maskC{white}
\begin{scope}[rotate=-63,xshift=.55cm]
\draw[color=black,very thick] (0,-1.14) ellipse[x radius=3cm,y radius=2.235cm];
\end{scope}
\begin{scope}[rotate=63,xshift=-.55cm]
\clip (0,-4) rectangle (-3.1,1.5);
\draw[color=blue,very thick] (0,-1.14) ellipse[x radius=3cm,y radius=2.235cm];
\end{scope}
	\draw[color=blue,very thick] (1.5,0.5) circle[radius=2.8cm];
	\draw[very thick, dotted, color=red] (-1.5,0.5) circle[radius=2.8cm];
	\draw[color=blue,very thick] (0,-2.5) circle[radius=2.8cm];
\draw[fill] (P1) ellipse[x radius=6pt, y radius=9pt] +(-1,.5)node{$X$};
\draw (nI3) +(.4,-0.8)node{$Z$};
\draw[fill] (nP1) ellipse[x radius=6pt, y radius=9pt] node[below right]{\small$P_1[3]$};
\draw (P1) +(1,-.5)node{\tiny$P_1[2]$};
\draw[fill,color=blue] (P3) ellipse[x radius=4pt,y radius=6pt] +(0,.2)node[ right]{\tiny$P_3[0]$};
\draw[fill,color=blue] (nP3) ellipse[x radius=4pt,y radius=6pt] +(-.8,-.5)node{\tiny$P_3[1]$};
\draw[fill, color=blue] (nI3) ellipse[x radius=4pt, y radius=6pt] +(-.4,0)node[above]{\tiny$I_3[1]$};
\end{scope}
} 

\newcommand{\PoneBig}{\begin{scope}[xscale=.7,yscale=.5,rotate=-5]
\coordinate (nI3) at (-1,1.8);
\coordinate (I3) at (1.58,-2.32);
\coordinate (I1) at (-1.58,-2.32);
\coordinate (nI1) at (1,1.8);
\coordinate (S2) at (0,-3.85);
\coordinate (nS2) at (0,1.4);
\coordinate (P2) at (0,-1.85);
\coordinate (nP2) at (0,2.85);
\coordinate (P1) at (-1.25,0);
\coordinate (nP1) at (2.75,-2);
\coordinate (P3) at (1.25,0);
\coordinate (nP3) at (-2.75,-2);
\draw (1,-6.5) node{$P_1\times {\color{blue}\,^\perp P_1}$};
\maskA{gray!27} 
\maskB{white}
\maskX{white}
\begin{scope}
\clip (-2,-2.3) rectangle (2,-3);
\draw[very thick,color=blue] (0,-.45) circle[radius=2.43cm];
\end{scope}
\begin{scope}
\clip (-1,2) rectangle (-3,-2.31);
\draw[very thick,color=blue] (0,-.42) circle[radius=2.44cm];
\end{scope}
\begin{scope}[rotate=-63,xshift=.55cm]
\draw[color=blue,very thick] (0,-1.14) ellipse[x radius=3cm,y radius=2.235cm];
\end{scope}%
	\draw[very thick] (-1.5,0.5) circle[radius=2.8cm];
	\draw[very thick,color=blue] (1.5,0.5) circle[radius=2.8cm];
	\draw[very thick, dotted, color=red] (0,-2.5) circle[radius=2.8cm];
\draw[fill] (I3) ellipse[x radius=6pt, y radius=9pt] +(.7,0)node[below]{\tiny$I_3[2]$};
\draw[fill] (nI3) ellipse[x radius=6pt, y radius=9pt] +(.4,-.2)node[below]{\tiny$I_3[1]$};
\draw (nI3) +(-1.5,.5)node{$Y$};
\draw[fill,color=blue] (P3) ellipse[x radius=4pt,y radius=6pt] +(0,.2)node[ right]{\tiny$P_3[0]$};
\draw[fill,color=blue] (nP3) ellipse[x radius=4pt,y radius=6pt] +(-.8,-.5)node{\tiny$P_3[1]$};
\draw[fill,color=blue] (nP2) ellipse[x radius=4pt,y radius=6pt] +(0,.2)node[above]{\tiny$P_2[1]$};
\end{scope}
} 

\newcommand{\AtwoLRc}[1]{ 
\begin{scope}
\maskAc{#1}
	\draw[thick] (0,-3)--(0,3) (0,2.6);
	\draw[thick,color=blue!80!white] (3,-3)--(-2.8,2.8);
	\draw (0,-3.5) node{$S_1\times \color{blue}P_2$};
\end{scope} 
\begin{scope}[scale=.5, xshift=4.5cm,yshift=3cm]
\coordinate (S1) at (0,0);
\coordinate (P2) at (0.5,.8);
\coordinate (S2) at (1,0);
\coordinate (S11) at (-1.5,.8);
\coordinate (S21) at (-.5,.8);
\coordinate (P21) at (-1,0);
\coordinate (S12) at (-3,0);
\coordinate (P22) at (-2.5,.8);
\coordinate (S22) at (-2,0);
\coordinate (L1) at (-2.1,1);
\coordinate (L2) at (-1.4,-.2);
\coordinate (L3) at (-.6,-.2);
\coordinate (L4) at (0.1,1);
	\draw[color=red] (L1)--(L2) (L3)--(L4);
	\draw[thick,fill,color=red] (S1) circle[radius=6pt] ; 
	\draw[thick,fill] (P2) circle[radius=5pt]; 
	\draw[thick,fill] (S2) circle[radius=5pt]; 
	\draw[thick] (S11) circle[radius=5pt] ; 
	\draw[thick,fill] (P21) circle[radius=5pt]; 
	\draw[thick,fill] (S21) circle[radius=5pt]; 
	\draw[thick] (S12) circle[radius=5pt] ; 
	\draw[thick,fill,color=red] (P22) circle[radius=6pt]; 
	\draw[thick] (S22) circle[radius=5pt]; 
\end{scope}
\begin{scope}[scale=.5, xshift=-3cm,yshift=-3.5cm]
\coordinate (S1) at (0,0);
\coordinate (P2) at (0.5,.8);
\coordinate (S2) at (1,0);
\coordinate (S11) at (-1.5,.8);
\coordinate (S21) at (-.5,.8);
\coordinate (P21) at (-1,0);
\coordinate (S12) at (-3,0);
\coordinate (P22) at (-2.5,.8);
\coordinate (S22) at (-2,0);
\coordinate (L1) at (-2.1,1);
\coordinate (L2) at (-1.4,-.2);
\coordinate (L3) at (-.6,-.2);
\coordinate (L4) at (0.1,1);
	\draw[color=red] (L1)--(L2) (L3)--(L4);
	\draw[thick] (S1) circle[radius=5pt] ; 
	\draw[thick,fill] (P2) circle[radius=5pt]; 
	\draw[thick,fill] (S2) circle[radius=5pt]; 
	\draw[thick] (S11) circle[radius=5pt] ; 
	\draw[thick,fill,color=red] (P21) circle[radius=6pt]; 
	\draw[thick] (S21) circle[radius=5pt]; 
	\draw[thick] (S12) circle[radius=5pt] ; 
	\draw[thick] (P22) circle[radius=5pt]; 
	\draw[thick] (S22) circle[radius=5pt]; 
\end{scope}
\begin{scope}[scale=.5, xshift=3.5cm,yshift=-5.5cm]
\coordinate (S1) at (0,0);
\coordinate (P2) at (0.5,.8);
\coordinate (S2) at (1,0);
\coordinate (S11) at (-1.5,.8);
\coordinate (S21) at (-.5,.8);
\coordinate (P21) at (-1,0);
\coordinate (S12) at (-3,0);
\coordinate (P22) at (-2.5,.8);
\coordinate (S22) at (-2,0);
\coordinate (L1) at (-2.1,1);
\coordinate (L2) at (-1.4,-.2);
\coordinate (L3) at (-.6,-.2);
\coordinate (L4) at (0.1,1);
	\draw[color=red] (L1)--(L2) (L3)--(L4);
	\draw[thick,fill,color=red] (S1) circle[radius=6pt] ; 
	\draw[thick,fill] (P2) circle[radius=5pt]; 
	\draw[thick,fill] (S2) circle[radius=5pt]; 
	\draw[thick] (S11) circle[radius=5pt] ; 
	\draw[thick,fill,color=red] (P21) circle[radius=6pt]; 
	\draw[thick,fill,color=red] (S21) circle[radius=6pt]; 
	\draw[thick] (S12) circle[radius=5pt] ; 
	\draw[thick] (P22) circle[radius=5pt]; 
	\draw[thick] (S22) circle[radius=5pt]; 
\end{scope}
\begin{scope}[scale=.5, xshift=-1.5cm,yshift=5cm]
\coordinate (S1) at (0,0);
\coordinate (P2) at (0.5,.8);
\coordinate (S2) at (1,0);
\coordinate (S11) at (-1.5,.8);
\coordinate (S21) at (-.5,.8);
\coordinate (P21) at (-1,0);
\coordinate (S12) at (-3,0);
\coordinate (P22) at (-2.5,.8);
\coordinate (S22) at (-2,0);
\coordinate (L1) at (-2.1,1);
\coordinate (L2) at (-1.4,-.2);
\coordinate (L3) at (-.6,-.2);
\coordinate (L4) at (0.1,1);
	\draw[color=red] (L1)--(L2) (L3)--(L4);
	\draw[thick] (S1) circle[radius=5pt] ; 
	\draw[thick,fill] (P2) circle[radius=5pt]; 
	\draw[thick,fill] (S2) circle[radius=5pt]; 
	\draw[thick] (S11) circle[radius=5pt] ; 
	\draw[thick,fill] (P21) circle[radius=5pt]; 
	\draw[thick] (S21) circle[radius=5pt]; 
	\draw[thick] (S12) circle[radius=5pt] ; 
	\draw[thick,fill,color=red] (P22) circle[radius=6pt]; 
	\draw[thick] (S22) circle[radius=5pt]; 
\end{scope}
\begin{scope}[scale=.5] 
	\draw (-.7,2) node{$5$};
	\draw (1.1,1.2) node{$b^+$};
	\draw (.7,-2) node{$2$};
	\draw (-1,-1) node{$b$};
\end{scope} 
}
\newcommand{\AtwoLRb}[1]{
\begin{scope}
\maskAd{#1} 
	\draw[thick,color=blue!80!white] (-3.5,0)--(3.5,0) (2.5,0);
	\draw[thick] (2.5,-2.5)--(-2.5,2.5);
	\draw (0,-3) node{$P_2\times \color{blue}S_2$};
\end{scope} 
\begin{scope}[scale=.5, xshift=4cm,yshift=3cm]
\coordinate (S1) at (0,0);
\coordinate (P2) at (0.5,.8);
\coordinate (S2) at (1,0);
\coordinate (S11) at (-1.5,.8);
\coordinate (S21) at (-.5,.8);
\coordinate (P21) at (-1,0);
\coordinate (S12) at (-3,0);
\coordinate (P22) at (-2.5,.8);
\coordinate (S22) at (-2,0);
\coordinate (L1) at (-2.1,1);
\coordinate (L2) at (-1.4,-.2);
\coordinate (L3) at (-.6,-.2);
\coordinate (L4) at (0.1,1);
	\draw[color=red] (L1)--(L2) (L3)--(L4);
	\draw[thick,fill] (S1) circle[radius=5pt] ; 
	\draw[thick,fill,color=red] (P2) circle[radius=6pt]; 
	\draw[thick,fill] (S2) circle[radius=5pt]; 
	\draw[thick] (S11) circle[radius=5pt] ; 
	\draw[thick] (P21) circle[radius=5pt]; 
	\draw[thick,fill] (S21) circle[radius=5pt]; 
	\draw[thick] (S12) circle[radius=5pt] ; 
	\draw[thick] (P22) circle[radius=5pt]; 
	\draw[thick,fill,color=red] (S22) circle[radius=6pt]; 
\end{scope}
\begin{scope}[scale=.5, xshift=-2cm,yshift=-3.5cm]
\coordinate (S1) at (0,0);
\coordinate (P2) at (0.5,.8);
\coordinate (S2) at (1,0);
\coordinate (S11) at (-1.5,.8);
\coordinate (S21) at (-.5,.8);
\coordinate (P21) at (-1,0);
\coordinate (S12) at (-3,0);
\coordinate (P22) at (-2.5,.8);
\coordinate (S22) at (-2,0);
\coordinate (L1) at (-2.1,1);
\coordinate (L2) at (-1.4,-.2);
\coordinate (L3) at (-.6,-.2);
\coordinate (L4) at (0.1,1);
	\draw[color=red] (L1)--(L2) (L3)--(L4);
	\draw[thick] (S1) circle[radius=5pt] ; 
	\draw[thick] (P2) circle[radius=5pt]; 
	\draw[thick,fill] (S2) circle[radius=5pt]; 
	\draw[thick] (S11) circle[radius=5pt] ; 
	\draw[thick] (P21) circle[radius=5pt]; 
	\draw[thick,fill,color=red] (S21) circle[radius=6pt]; 
	\draw[thick] (S12) circle[radius=5pt] ; 
	\draw[thick] (P22) circle[radius=5pt]; 
	\draw[thick] (S22) circle[radius=5pt]; 
\end{scope}
\begin{scope}[scale=.5, xshift=5.5cm,yshift=-2cm]
\coordinate (S1) at (0,0);
\coordinate (P2) at (0.5,.8);
\coordinate (S2) at (1,0);
\coordinate (S11) at (-1.5,.8);
\coordinate (S21) at (-.5,.8);
\coordinate (P21) at (-1,0);
\coordinate (S12) at (-3,0);
\coordinate (P22) at (-2.5,.8);
\coordinate (S22) at (-2,0);
\coordinate (L1) at (-2.1,1);
\coordinate (L2) at (-1.4,-.2);
\coordinate (L3) at (-.6,-.2);
\coordinate (L4) at (0.1,1);
	\draw[color=red] (L1)--(L2) (L3)--(L4); 
	\draw[thick,fill,color=red] (S1) circle[radius=6pt] ; 
	\draw[thick,fill,color=red] (P2) circle[radius=6pt]; 
	\draw[thick,fill] (S2) circle[radius=5pt]; 
	\draw[thick] (S11) circle[radius=5pt] ; 
	\draw[thick] (P21) circle[radius=5pt]; 
	\draw[thick,fill,color=red] (S21) circle[radius=6pt]; 
	\draw[thick] (S12) circle[radius=5pt] ; 
	\draw[thick] (P22) circle[radius=5pt]; 
	\draw[thick] (S22) circle[radius=5pt]; 
\end{scope} 
\begin{scope}[scale=.5, xshift=-4.5cm,yshift=1.5cm]
\coordinate (S1) at (0,0);
\coordinate (P2) at (0.5,.8);
\coordinate (S2) at (1,0);
\coordinate (S11) at (-1.5,.8);
\coordinate (S21) at (-.5,.8);
\coordinate (P21) at (-1,0);
\coordinate (S12) at (-3,0);
\coordinate (P22) at (-2.5,.8);
\coordinate (S22) at (-2,0);
\coordinate (L1) at (-2.1,1);
\coordinate (L2) at (-1.4,-.2);
\coordinate (L3) at (-.6,-.2);
\coordinate (L4) at (0.1,1);
	\draw[color=red] (L1)--(L2) (L3)--(L4);
	\draw[thick] (S1) circle[radius=5pt] ; 
	\draw[thick] (P2) circle[radius=5pt]; 
	\draw[thick,fill] (S2) circle[radius=5pt]; 
	\draw[thick] (S11) circle[radius=5pt] ; 
	\draw[thick] (P21) circle[radius=5pt]; 
	\draw[thick,fill] (S21) circle[radius=5pt]; 
	\draw[thick] (S12) circle[radius=5pt] ; 
	\draw[thick] (P22) circle[radius=5pt]; 
	\draw[thick,fill,color=red] (S22) circle[radius=6pt]; 
\end{scope}
\begin{scope}[scale=.5] 
	\draw (-2.2,1) node{$4$};
	\draw (.8,1.3) node{$a^+$};
	\draw (2.1,-1) node{$2$};
	\draw (-.8,-1.3) node{$a$};
\end{scope} 
}
\newcommand{\AtwoUR}[1]{
\begin{scope}
\maskAa{#1} 
	\draw[thick,color=blue!80!white] (-3.5,0)--(3.5,0) (2.5,0);
	\draw[thick,color=blue!80!white] (0,-3)--(0,3) (0,2.6);
	\draw[thick,color=blue!80!white] (0,0)--(-2.5,2.5);
	\draw (0,-3.5) node{$0\times \color{blue}A_2$};
\end{scope}
\begin{scope}[scale=.5, xshift=-2.5cm,yshift=-3.5cm]
\coordinate (S1) at (0,0);
\coordinate (P2) at (0.5,.8);
\coordinate (S2) at (1,0);
\coordinate (S11) at (-1.5,.8);
\coordinate (S21) at (-.5,.8);
\coordinate (P21) at (-1,0);
\coordinate (S12) at (-3,0);
\coordinate (P22) at (-2.5,.8);
\coordinate (S22) at (-2,0);
\coordinate (L1) at (-2.1,1);
\coordinate (L2) at (-1.4,-.2);
\coordinate (L3) at (-.6,-.2);
\coordinate (L4) at (0.1,1);
	\draw[color=red] (L1)--(L2) (L3)--(L4);
	\draw[thick,fill] (S1) circle[radius=5pt] ; 
	\draw[thick,fill] (P2) circle[radius=5pt]; 
	\draw[thick,fill] (S2) circle[radius=5pt]; 
	\draw[thick,fill,color=red] (S11) circle[radius=6pt] ; 
	\draw[thick,fill,color=red] (P21) circle[radius=6pt]; 
	\draw[thick,fill,color=red] (S21) circle[radius=6pt]; 
	\draw[thick] (S12) circle[radius=5pt] ; 
	\draw[thick] (P22) circle[radius=5pt]; 
	\draw[thick] (S22) circle[radius=5pt]; 
\end{scope}
\begin{scope}[scale=.5, xshift=4.5cm,yshift=-3.5cm]
\coordinate (S1) at (0,0);
\coordinate (P2) at (0.5,.8);
\coordinate (S2) at (1,0);
\coordinate (S11) at (-1.5,.8);
\coordinate (S21) at (-.5,.8);
\coordinate (P21) at (-1,0);
\coordinate (S12) at (-3,0);
\coordinate (P22) at (-2.5,.8);
\coordinate (S22) at (-2,0);
\coordinate (L1) at (-2.1,1);
\coordinate (L2) at (-1.4,-.2);
\coordinate (L3) at (-.6,-.2);
\coordinate (L4) at (0.1,1);
	\draw[color=red] (L1)--(L2) (L3)--(L4);
	\draw[thick,fill] (S1) circle[radius=5pt] ; 
	\draw[thick,fill] (P2) circle[radius=5pt]; 
	\draw[thick,fill] (S2) circle[radius=5pt]; 
	\draw[thick,fill] (S11) circle[radius=5pt] ; 
	\draw[thick,fill] (P21) circle[radius=5pt]; 
	\draw[thick,fill,color=red] (S21) circle[radius=6pt]; 
	\draw[thick,fill,color=red] (S12) circle[radius=6pt] ; 
	\draw[thick] (P22) circle[radius=5pt]; 
	\draw[thick] (S22) circle[radius=5pt]; 
\end{scope}
\begin{scope}[scale=.5, xshift=4.5cm,yshift=3cm]
\coordinate (S1) at (0,0);
\coordinate (P2) at (0.5,.8);
\coordinate (S2) at (1,0);
\coordinate (S11) at (-1.5,.8);
\coordinate (S21) at (-.5,.8);
\coordinate (P21) at (-1,0);
\coordinate (S12) at (-3,0);
\coordinate (P22) at (-2.5,.8);
\coordinate (S22) at (-2,0);
\coordinate (L1) at (-2.1,1);
\coordinate (L2) at (-1.4,-.2);
\coordinate (L3) at (-.6,-.2);
\coordinate (L4) at (0.1,1);
	\draw[color=red] (L1)--(L2) (L3)--(L4);
	\draw[thick,fill] (S1) circle[radius=5pt] ; 
	\draw[thick,fill] (P2) circle[radius=5pt]; 
	\draw[thick,fill] (S2) circle[radius=5pt]; 
	\draw[thick,fill] (S11) circle[radius=5pt] ; 
	\draw[thick,fill] (P21) circle[radius=5pt]; 
	\draw[thick,fill] (S21) circle[radius=5pt]; 
	\draw[thick,fill,color=red] (S12) circle[radius=6pt] ; 
	\draw[thick,fill,color=red] (P22) circle[radius=6pt]; 
	\draw[thick,fill,color=red] (S22) circle[radius=6pt]; 
\end{scope}
\begin{scope}[scale=.5, xshift=-1.5cm,yshift=5cm]
\coordinate (S1) at (0,0);
\coordinate (P2) at (0.5,.8);
\coordinate (S2) at (1,0);
\coordinate (S11) at (-1.5,.8);
\coordinate (S21) at (-.5,.8);
\coordinate (P21) at (-1,0);
\coordinate (S12) at (-3,0);
\coordinate (P22) at (-2.5,.8);
\coordinate (S22) at (-2,0);
\coordinate (L1) at (-2.1,1);
\coordinate (L2) at (-1.4,-.2);
\coordinate (L3) at (-.6,-.2);
\coordinate (L4) at (0.1,1);
	\draw[color=red] (L1)--(L2) (L3)--(L4);
	\draw[thick,fill] (S1) circle[radius=5pt] ; 
	\draw[thick,fill] (P2) circle[radius=5pt]; 
	\draw[thick,fill] (S2) circle[radius=5pt]; 
	\draw[thick,fill,color=red] (S11) circle[radius=6pt] ; 
	\draw[thick,fill] (P21) circle[radius=5pt]; 
	\draw[thick,fill] (S21) circle[radius=5pt]; 
	\draw[thick] (S12) circle[radius=5pt] ; 
	\draw[thick,fill,color=red] (P22) circle[radius=6pt]; 
	\draw[thick,fill,color=red] (S22) circle[radius=6pt]; 
\end{scope}
\begin{scope}[scale=.5, xshift=-4.5cm,yshift=1.5cm]
\coordinate (S1) at (0,0);
\coordinate (P2) at (0.5,.8);
\coordinate (S2) at (1,0);
\coordinate (S11) at (-1.5,.8);
\coordinate (S21) at (-.5,.8);
\coordinate (P21) at (-1,0);
\coordinate (S12) at (-3,0);
\coordinate (P22) at (-2.5,.8);
\coordinate (S22) at (-2,0);
\coordinate (L1) at (-2.1,1);
\coordinate (L2) at (-1.4,-.2);
\coordinate (L3) at (-.6,-.2);
\coordinate (L4) at (0.1,1);
	\draw[color=red] (L1)--(L2) (L3)--(L4);
	\draw[thick,fill] (S1) circle[radius=5pt] ; 
	\draw[thick,fill] (P2) circle[radius=5pt]; 
	\draw[thick,fill] (S2) circle[radius=5pt]; 
	\draw[thick,fill,color=red] (S11) circle[radius=6pt] ; 
	\draw[thick,fill,color=red] (P21) circle[radius=6pt]; 
	\draw[thick,fill] (S21) circle[radius=5pt]; 
	\draw[thick] (S12) circle[radius=5pt] ; 
	\draw[thick] (P22) circle[radius=5pt]; 
	\draw[thick,fill,color=red] (S22) circle[radius=6pt]; 
\end{scope}
\begin{scope}[scale=.5] 
	\draw (-2.2,1) node{$a^+$};
	\draw (-.7,2) node{$b^+$};
	\draw (1,1.2) node{$3$};
	\draw (1,-1.2) node{$c^+$};
	\draw (-1,-1.3) node{$2$};
\end{scope} 
}
\newcommand{\AtwoUL}[1]{
\begin{scope}
\maskAb{#1} 
	\draw[thick] (-3.5,0)--(3.5,0) (2.5,0);
	\draw[thick,color=blue!80!white] (0,-3)--(0,3) (0,2.6);
	\draw (0,-3.5) node{$S_2\times \color{blue}S_1$};
\end{scope}
\begin{scope}[scale=.5, xshift=-2.5cm,yshift=-3.5cm]
\coordinate (S1) at (0,0);
\coordinate (P2) at (0.5,.8);
\coordinate (S2) at (1,0);
\coordinate (S11) at (-1.5,.8);
\coordinate (S21) at (-.5,.8);
\coordinate (P21) at (-1,0);
\coordinate (S12) at (-3,0);
\coordinate (P22) at (-2.5,.8);
\coordinate (S22) at (-2,0);
\coordinate (L1) at (-2.1,1);
\coordinate (L2) at (-1.4,-.2);
\coordinate (L3) at (-.6,-.2);
\coordinate (L4) at (0.1,1);
	\draw[color=red] (L1)--(L2) (L3)--(L4);
	\draw[thick,fill] (S1) circle[radius=5pt] ; 
	\draw[thick] (P2) circle[radius=5pt]; 
	\draw[thick] (S2) circle[radius=5pt]; 
	\draw[thick,fill,color=red] (S11) circle[radius=6pt] ; 
	\draw[thick] (P21) circle[radius=5pt]; 
	\draw[thick] (S21) circle[radius=5pt]; 
	\draw[thick] (S12) circle[radius=5pt] ; 
	\draw[thick] (P22) circle[radius=5pt]; 
	\draw[thick] (S22) circle[radius=5pt]; 
\end{scope}
\begin{scope}[scale=.5, xshift=4.5cm,yshift=-3.5cm]
\coordinate (S1) at (0,0);
\coordinate (P2) at (0.5,.8);
\coordinate (S2) at (1,0);
\coordinate (S11) at (-1.5,.8);
\coordinate (S21) at (-.5,.8);
\coordinate (P21) at (-1,0);
\coordinate (S12) at (-3,0);
\coordinate (P22) at (-2.5,.8);
\coordinate (S22) at (-2,0);
\coordinate (L1) at (-2.1,1);
\coordinate (L2) at (-1.4,-.2);
\coordinate (L3) at (-.6,-.2);
\coordinate (L4) at (0.1,1);
	\draw[color=red] (L1)--(L2) (L3)--(L4);
	\draw[thick,fill] (S1) circle[radius=5pt] ; 
	\draw[thick] (P2) circle[radius=5pt]; 
	\draw[thick] (S2) circle[radius=5pt]; 
	\draw[thick,fill] (S11) circle[radius=5pt] ; 
	\draw[thick] (P21) circle[radius=5pt]; 
	\draw[thick] (S21) circle[radius=5pt]; 
	\draw[thick,fill,color=red] (S12) circle[radius=6pt] ; 
	\draw[thick] (P22) circle[radius=5pt]; 
	\draw[thick] (S22) circle[radius=5pt]; 
\end{scope}
\begin{scope}[scale=.5, xshift=-2.5cm,yshift=3.5cm]
\coordinate (S1) at (0,0);
\coordinate (P2) at (0.5,.8);
\coordinate (S2) at (1,0);
\coordinate (S11) at (-1.5,.8);
\coordinate (S21) at (-.5,.8);
\coordinate (P21) at (-1,0);
\coordinate (S12) at (-3,0);
\coordinate (P22) at (-2.5,.8);
\coordinate (S22) at (-2,0);
\coordinate (L1) at (-2.1,1);
\coordinate (L2) at (-1.4,-.2);
\coordinate (L3) at (-.6,-.2);
\coordinate (L4) at (0.1,1);
	\draw[color=red] (L1)--(L2) (L3)--(L4);
	\draw[thick,fill] (S1) circle[radius=5pt] ; 
	\draw[thick,fill] (P2) circle[radius=5pt]; 
	\draw[thick,fill,color=red] (S2) circle[radius=6pt]; 
	\draw[thick,fill,color=red] (S11) circle[radius=6pt] ; 
	\draw[thick] (P21) circle[radius=5pt]; 
	\draw[thick] (S21) circle[radius=5pt]; 
	\draw[thick] (S12) circle[radius=5pt] ; 
	\draw[thick] (P22) circle[radius=5pt]; 
	\draw[thick] (S22) circle[radius=5pt]; 
\end{scope}
\begin{scope}[scale=.5, xshift=4.5cm,yshift=3.5cm]
\coordinate (S1) at (0,0);
\coordinate (P2) at (0.5,.8);
\coordinate (S2) at (1,0);
\coordinate (S11) at (-1.5,.8);
\coordinate (S21) at (-.5,.8);
\coordinate (P21) at (-1,0);
\coordinate (S12) at (-3,0);
\coordinate (P22) at (-2.5,.8);
\coordinate (S22) at (-2,0);
\coordinate (L1) at (-2.1,1);
\coordinate (L2) at (-1.4,-.2);
\coordinate (L3) at (-.6,-.2);
\coordinate (L4) at (0.1,1);
	\draw[color=red] (L1)--(L2) (L3)--(L4);
	\draw[thick,fill] (S1) circle[radius=5pt] ; 
	\draw[thick,fill] (P2) circle[radius=5pt]; 
	\draw[thick,fill,color=red] (S2) circle[radius=6pt]; 
	\draw[thick,fill] (S11) circle[radius=5pt] ; 
	\draw[thick] (P21) circle[radius=5pt]; 
	\draw[thick] (S21) circle[radius=5pt]; 
	\draw[thick,fill,color=red] (S12) circle[radius=6pt] ; 
	\draw[thick] (P22) circle[radius=5pt]; 
	\draw[thick] (S22) circle[radius=5pt]; 
\end{scope}
\begin{scope}[scale=.5]
	\draw (-1,1.2) node{$2$};
	\draw (1.3,1.4) node{$c^+$};
	\draw (1,-1.5) node{$6$};
	\draw (-1,-1.5) node{$c$};
\end{scope} 
}
\newcommand{\AtwoLL}[1]{ 
\begin{scope}
\maskAe{#1} 
	\draw[thick] (-3.5,0)--(3.5,0) (2.5,0);
	\draw[thick] (0,-3)--(0,3) (0,2.6);
	\draw[thick] (0,0)--(-2.5,2.5);
	\draw (0,-3.5) node{$A_2\times \color{blue}0$};
\begin{scope}[scale=.5,xshift=-3cm,yshift=-4cm]
	\draw[thick] (0,0) circle[radius=5pt];
	\draw[thick] (0.5,.8) circle[radius=5pt];
	\draw[thick] (1,0) circle[radius=5pt];
\coordinate (S11) at (-1.5,.8);
\coordinate (S21) at (-.5,.8);
\coordinate (P21) at (-1,0);
\coordinate (S12) at (-3,0);
\coordinate (P22) at (-2.5,.8);
\coordinate (S22) at (-2,0);
\coordinate (L1) at (-2.1,1);
\coordinate (L2) at (-1.4,-.2);
\coordinate (L3) at (-.6,-.2);
\coordinate (L4) at (0.1,1);
	\draw[color=red] (L1)--(L2) (L3)--(L4);
	\draw[thick] (S11) circle[radius=5pt] ; %
	\draw[thick] (P21) circle[radius=5pt]; %
	\draw[thick] (S21) circle[radius=5pt]; %
	\draw[thick] (S12) circle[radius=5pt] ; %
	\draw[thick] (P22) circle[radius=5pt]; %
	\draw[thick] (S22) circle[radius=5pt]; 
\end{scope} 
\begin{scope}[scale=.5,xshift=4.5cm,yshift=-4cm]
	\draw[thick,fill,color=red] (0,0) circle[radius=6pt];
	\draw[thick] (0.5,.8) circle[radius=5pt];
	\draw[thick] (1,0) circle[radius=5pt];
\coordinate (S11) at (-1.5,.8);
\coordinate (S21) at (-.5,.8);
\coordinate (P21) at (-1,0);
\coordinate (S12) at (-3,0);
\coordinate (P22) at (-2.5,.8);
\coordinate (S22) at (-2,0);
\coordinate (L1) at (-2.1,1);
\coordinate (L2) at (-1.4,-.2);
\coordinate (L3) at (-.6,-.2);
\coordinate (L4) at (0.1,1);
	\draw[color=red] (L1)--(L2) (L3)--(L4);
	\draw[thick] (S11) circle[radius=5pt] ; %
	\draw[thick] (P21) circle[radius=5pt]; %
	\draw[thick] (S21) circle[radius=5pt]; %
	\draw[thick] (S12) circle[radius=5pt] ; %
	\draw[thick] (P22) circle[radius=5pt]; %
	\draw[thick] (S22) circle[radius=5pt]; 
\end{scope} 
\begin{scope}[scale=.5,xshift=4.5cm,yshift=3cm]
	\draw[thick,fill,color=red] (0,0) circle[radius=6pt];
	\draw[thick,fill,color=red] (0.5,.8) circle[radius=6pt];
	\draw[thick,fill,color=red] (1,0) circle[radius=6pt];
\coordinate (S11) at (-1.5,.8);
\coordinate (S21) at (-.5,.8);
\coordinate (P21) at (-1,0);
\coordinate (S12) at (-3,0);
\coordinate (P22) at (-2.5,.8);
\coordinate (S22) at (-2,0);
\coordinate (L1) at (-2.1,1);
\coordinate (L2) at (-1.4,-.2);
\coordinate (L3) at (-.6,-.2);
\coordinate (L4) at (0.1,1);
	\draw[color=red] (L1)--(L2) (L3)--(L4);
	\draw[thick] (S11) circle[radius=5pt] ; %
	\draw[thick] (P21) circle[radius=5pt]; %
	\draw[thick] (S21) circle[radius=5pt]; %
	\draw[thick] (S12) circle[radius=5pt] ; %
	\draw[thick] (P22) circle[radius=5pt]; %
	\draw[thick] (S22) circle[radius=5pt]; 
\end{scope} 
\begin{scope}[scale=.5,xshift=-1.5cm,yshift=5cm]
	\draw[thick] (0,0) circle[radius=5pt];
	\draw[thick,fill,color=red] (0.5,.8) circle[radius=6pt];
	\draw[thick,fill,color=red] (1,0) circle[radius=6pt];
\coordinate (S11) at (-1.5,.8);
\coordinate (S21) at (-.5,.8);
\coordinate (P21) at (-1,0);
\coordinate (S12) at (-3,0);
\coordinate (P22) at (-2.5,.8);
\coordinate (S22) at (-2,0);
\coordinate (L1) at (-2.1,1);
\coordinate (L2) at (-1.4,-.2);
\coordinate (L3) at (-.6,-.2);
\coordinate (L4) at (0.1,1);
	\draw[color=red] (L1)--(L2) (L3)--(L4);
	\draw[thick] (S11) circle[radius=5pt] ; %
	\draw[thick] (P21) circle[radius=5pt]; %
	\draw[thick] (S21) circle[radius=5pt]; %
	\draw[thick] (S12) circle[radius=5pt] ; %
	\draw[thick] (P22) circle[radius=5pt]; %
	\draw[thick] (S22) circle[radius=5pt]; 
\end{scope} 
\begin{scope}[scale=.5,xshift=-5cm,yshift=1.5cm]
	\draw[thick] (0,0) circle[radius=5pt];
	\draw[thick] (0.5,.8) circle[radius=5pt];
	\draw[thick,fill,color=red] (1,0) circle[radius=6pt];
\coordinate (S11) at (-1.5,.8);
\coordinate (S21) at (-.5,.8);
\coordinate (P21) at (-1,0);
\coordinate (S12) at (-3,0);
\coordinate (P22) at (-2.5,.8);
\coordinate (S22) at (-2,0);
\coordinate (L1) at (-2.1,1);
\coordinate (L2) at (-1.4,-.2);
\coordinate (L3) at (-.6,-.2);
\coordinate (L4) at (0.1,1);
	\draw[color=red] (L1)--(L2) (L3)--(L4);
	\draw[thick] (S11) circle[radius=5pt] ; %
	\draw[thick] (P21) circle[radius=5pt]; %
	\draw[thick] (S21) circle[radius=5pt]; %
	\draw[thick] (S12) circle[radius=5pt] ; %
	\draw[thick] (P22) circle[radius=5pt]; %
	\draw[thick] (S22) circle[radius=5pt]; 
\end{scope} 
\begin{scope}[scale=.5] 
	\draw (-2.2,1) node{$a$};
	\draw (-.7,2) node{$b$};
	\draw (1,1.2) node{$2$};
	\draw (1,-1.5) node{$c$};
	\draw (-1,-1.5) node{$1$};
\end{scope}
\end{scope} 
}

\newcommand{\AtwoVert}[1]{ 
\draw[thin,fill,color=#1] (-3.5,-3) rectangle (3.5,3);
\draw[thick] (-3.5,0)--(3.5,0) (0,-3)--(0,3) (2.5,-2.5)--(0,0);
\begin{scope}[scale=.5,xshift=4cm,yshift=3cm]
\coordinate (S11) at (-1.5,.8);
\coordinate (S21) at (-.5,.8);
\coordinate (P21) at (-1,0);
\coordinate (S12) at (-3,0);
\coordinate (P22) at (-2.5,.8);
\coordinate (S22) at (-2,0);
\coordinate (L1) at (-2.1,1);
\coordinate (L2) at (-1.4,-.2);
\coordinate (L3) at (-.6,-.2);
\coordinate (L4) at (0.1,1);
	\draw[thick] (S12) circle[radius=5pt] ; 
	\draw[thick] (P22) circle[radius=5pt]; 
	\draw[thick] (S22) circle[radius=5pt]; 
	\draw[thick,fill,color=red] (0,0) circle[radius=6pt];
	\draw[thick,fill,color=red] (0.5,.8) circle[radius=6pt];
	\draw[thick,fill,color=red] (1,0) circle[radius=6pt];
	\draw[thick] (S11) circle[radius=5pt];
	\draw[thick] (S21) circle[radius=5pt];
	\draw[thick] (P21) circle[radius=5pt];
	\draw[color=red] (L1)--(L2) (L3)--(L4);
\end{scope}
\begin{scope}[scale=.5, xshift=-2.5cm,yshift=3cm]
\coordinate (S1) at (0,0);
\coordinate (P2) at (0.5,.8);
\coordinate (S2) at (1,0);
\coordinate (S11) at (-1.5,.8);
\coordinate (S21) at (-.5,.8);
\coordinate (P21) at (-1,0);
\coordinate (S12) at (-3,0);
\coordinate (P22) at (-2.5,.8);
\coordinate (S22) at (-2,0);
\coordinate (L1) at (-2.1,1);
\coordinate (L2) at (-1.4,-.2);
\coordinate (L3) at (-.6,-.2);
\coordinate (L4) at (0.1,1);
	\draw[color=red] (L1)--(L2) (L3)--(L4);
	\draw[thick,fill] (S1) circle[radius=5pt] ; 
	\draw[thick,fill] (P2) circle[radius=5pt]; 
	\draw[thick,fill,color=red] (S2) circle[radius=6pt]; 
	\draw[thick,fill,color=red] (S11) circle[radius=6pt] ; 
	\draw[thick] (P21) circle[radius=5pt]; 
	\draw[thick] (S21) circle[radius=5pt]; 
	\draw[thick] (S12) circle[radius=5pt] ; 
	\draw[thick] (P22) circle[radius=5pt]; 
	\draw[thick] (S22) circle[radius=5pt]; 
\end{scope}
\begin{scope}[scale=.5, xshift=-2.5cm,yshift=-3.5cm]
\coordinate (S1) at (0,0);
\coordinate (P2) at (0.5,.8);
\coordinate (S2) at (1,0);
\coordinate (S11) at (-1.5,.8);
\coordinate (S21) at (-.5,.8);
\coordinate (P21) at (-1,0);
\coordinate (S12) at (-3,0);
\coordinate (P22) at (-2.5,.8);
\coordinate (S22) at (-2,0);
\coordinate (L1) at (-2.1,1);
\coordinate (L2) at (-1.4,-.2);
\coordinate (L3) at (-.6,-.2);
\coordinate (L4) at (0.1,1);
	\draw[color=red] (L1)--(L2) (L3)--(L4);
	\draw[thick,fill] (S1) circle[radius=5pt] ; 
	\draw[thick,fill] (P2) circle[radius=5pt]; 
	\draw[thick,fill] (S2) circle[radius=5pt]; 
	\draw[thick,fill,color=red] (S11) circle[radius=6pt] ; 
	\draw[thick,fill,color=red] (P21) circle[radius=6pt]; 
	\draw[thick,fill,color=red] (S21) circle[radius=6pt]; 
	\draw[thick] (S12) circle[radius=5pt] ; 
	\draw[thick] (P22) circle[radius=5pt]; 
	\draw[thick] (S22) circle[radius=5pt]; 
\end{scope}
\begin{scope}[scale=.5, xshift=5.5cm,yshift=-2cm]
\coordinate (S1) at (0,0);
\coordinate (P2) at (0.5,.8);
\coordinate (S2) at (1,0);
\coordinate (S11) at (-1.5,.8);
\coordinate (S21) at (-.5,.8);
\coordinate (P21) at (-1,0);
\coordinate (S12) at (-3,0);
\coordinate (P22) at (-2.5,.8);
\coordinate (S22) at (-2,0);
\coordinate (L1) at (-2.1,1);
\coordinate (L2) at (-1.4,-.2);
\coordinate (L3) at (-.6,-.2);
\coordinate (L4) at (0.1,1);
	\draw[color=red] (L1)--(L2) (L3)--(L4); 
	\draw[thick,fill,color=red] (S1) circle[radius=6pt] ; 
	\draw[thick,fill,color=red] (P2) circle[radius=6pt]; 
	\draw[thick,fill] (S2) circle[radius=5pt]; 
	\draw[thick] (S11) circle[radius=5pt] ; 
	\draw[thick] (P21) circle[radius=5pt]; 
	\draw[thick,fill,color=red] (S21) circle[radius=6pt]; 
	\draw[thick] (S12) circle[radius=5pt] ; 
	\draw[thick] (P22) circle[radius=5pt]; 
	\draw[thick] (S22) circle[radius=5pt]; 
\end{scope}
\begin{scope}[scale=.5, xshift=3.5cm,yshift=-5.5cm]
\coordinate (S1) at (0,0);
\coordinate (P2) at (0.5,.8);
\coordinate (S2) at (1,0);
\coordinate (S11) at (-1.5,.8);
\coordinate (S21) at (-.5,.8);
\coordinate (P21) at (-1,0);
\coordinate (S12) at (-3,0);
\coordinate (P22) at (-2.5,.8);
\coordinate (S22) at (-2,0);
\coordinate (L1) at (-2.1,1);
\coordinate (L2) at (-1.4,-.2);
\coordinate (L3) at (-.6,-.2);
\coordinate (L4) at (0.1,1);
	\draw[color=red] (L1)--(L2) (L3)--(L4);
	\draw[thick,fill,color=red] (S1) circle[radius=6pt] ; 
	\draw[thick,fill] (P2) circle[radius=5pt]; 
	\draw[thick,fill] (S2) circle[radius=5pt]; 
	\draw[thick] (S11) circle[radius=5pt] ; 
	\draw[thick,fill,color=red] (P21) circle[radius=6pt]; 
	\draw[thick,fill,color=red] (S21) circle[radius=6pt]; 
	\draw[thick] (S12) circle[radius=5pt] ; 
	\draw[thick] (P22) circle[radius=5pt]; 
	\draw[thick] (S22) circle[radius=5pt]; 
\end{scope}
}

\title{Horizontal and vertical mutation fans}
\author{Kiyoshi Igusa}
\address{Department of Mathematics, Brandeis University, Waltham, MA 02454}\email{igusa@brandeis.edu}

\subjclass[2010]{
18E30:16G20}
 
 \keywords{exceptional sequence, $m$-cluster category, semi-invariant pictures, $t$-structures, $c$-vectors, $g$-vectors, silting objects, simple minded collections}

\begin{document}

\begin{abstract} We introduce diagrams for $m$-cluster categories which we call ``horizontal'' and ``vertical'' mutation fans. These are analogous to the mutation fans (also known as ``semi-invariant pictures'' or ``scattering diagrams'') for the standard ($m=1$) cluster case which are dual to the poset of finitely generated torsion classes. The purpose of these diagrams is to visualize mutations and analogues of maximal green sequences in the $m$-cluster category with special emphasis on the $c$-vectors (the ``brick'' labels).
\end{abstract}

\maketitle


%
%

\section{Introduction}

These are annotated notes from my lecture at Workshop on Cluster Algebras and Related Topics held at the Chern Institute of Mathematics July 10-13, 2017. Preliminaries discussing the standard ``pictures'' used to visualize maximal green sequences are added. Also, there are additional comments to address a question of Zhe Han right after my talk: Do the horizontal fans correspond to torsion classes? I said ``yes'', but I will answer this more completely in these notes using bounded $t$-structures and the ``spots'' notation of \cite{ST} as illustrated in the lecture of Osamu Iyama on torsion classes and support $\tau$-tilting modules. Finally, these notes end with a list and short description of my other papers and comments about the history of stability conditions and maximal green sequences.

These notes begin with a ``preview'' of the ``horizontal'' and ``vertical'' mutation fans. In the lecture I used $A_2$ as the preview. Here I use a rank 3 example: $A_3$ (see Figure \ref{preview}). The horizontal fans are viewed as the ``floors'' of a building. In rank 3, each floor is subdivided into triangular ``rooms'' (in rank 2 each room has only two walls). Each wall of each room has a door to the next room and at most one set of stairs going either up or down. The vertical fans are sets of rooms connected only be stairs.

All figures for quivers with 3 vertices (Figures \ref{Figure00}, \ref{FigA2tilde}, \ref{A3: vertical fan 0 x A3 x 0}, \ref{A3: horizontal fans H(X), H(Y)}, \ref{Fig:A3 Horizontal fans}) are drawn in perspective. The three coordinate hyperplanes become great circles when intersected with the unit sphere $S^2\subseteq \RR^3$. The stereographic projection to $\RR^2$ gives three overlapping circles. When the plane is drawn in perspective, all circles become ellipses. These figures should be interpreted as patterns on the ground viewed from the side.

\begin{figure}[htbp] 
\begin{center}
\begin{tikzpicture}
\begin{scope}[yscale=.45,yshift=-5cm] 
	\draw[fill,color=white] (0,-.5) ellipse [x radius=3.3cm,y radius=2.8cm];
	\allblackB 
	\begin{scope}[xscale=.7,yscale=.45,rotate=-5]
		\draw[color=black,fill] (P3) ellipse[x radius=6pt, y radius=9pt] +(0,-.5)node[above right]{\tiny$P_3[2]$};
		\draw[color=black,fill] (nP3) ellipse[x radius=6pt, y radius=9pt] +(0,0.5)node[below left]{\tiny$P_3[3]$};
	\end{scope}
	\begin{scope}[xscale=.7,yscale=.45,rotate=-5]
		\draw[fill] (P1) ellipse[x radius=6pt, y radius=9pt] node[above left]{\tiny$P_1[2]$};
		\draw[fill] (nP1) ellipse[x radius=6pt, y radius=9pt] node[below right]{\tiny$P_1[3]$};
	\end{scope}
\end{scope}
\begin{scope}[yscale=.45,yshift=0cm]
	\draw[fill,color=white] (0,-.5) ellipse [x radius=3.3cm, y radius=2.8cm];
	\PthreePerpB 
\end{scope}
\begin{scope}[yscale=.45,yshift=5cm]
	\draw[fill,color=white] (0,-.5) ellipse [x radius=3.3cm, y radius=3cm];
	\IoneB 
		\draw[thick,color=black] (-1.4,.4) node{$X$};
\end{scope}
\begin{scope}[yscale=.45,yshift=10cm]
	\draw[fill,color=white] (0,-.5) ellipse [x radius=3.3cm, y radius=3cm];
	\PoneB 
		\draw[thick,color=black] (-2,.9)node{$Y$};
\end{scope}
\begin{scope}[yscale=.45,yshift=15cm]
	\draw[fill,color=white] (0,-.5) ellipse [x radius=3.3cm, y radius=2.8cm];
	\allblueB 
\end{scope}
\begin{scope}[yscale=.45, yshift=-5.45cm, xshift=-.6cm]
	\draw[very thick, color=green!80!black,dashed,->] (0,0)--(0,5);
\end{scope}
\begin{scope}[yscale=.45, yshift=-.16cm, xshift=-1.4cm]
	\draw[very thick, color=green!80!black,dashed,->] (0,0)--(0,5);
\end{scope}
\begin{scope}[yscale=.45, yshift=5.6cm, xshift=-1.8cm]
	\draw[very thick, color=green!80!black,dashed,->] (0,0)--(0,5);
\end{scope}
\begin{scope}[yscale=.45, yshift=11.25cm, xshift=-2.3cm]
	\draw[very thick, color=green!80!black,dashed,->] (0,0)--(0,5);
\end{scope}
\end{tikzpicture}
\begin{tikzpicture}[yscale=.5]
\draw[color=white] (-3,-8) rectangle (-2,-12);
{\begin{scope}[xscale=.7,yscale=.6,rotate=-5]
\coordinate (nI3) at (-1,1.8);
\coordinate (I3) at (1.58,-2.32);
\coordinate (I1) at (-1.58,-2.32);
\coordinate (nI1) at (1,1.8);
\coordinate (S2) at (0,-3.85);
\coordinate (nS2) at (0,1.4);
\coordinate (P2) at (0,-1.85);
\coordinate (nP2) at (0,2.85);
\coordinate (P1) at (-1.25,0);
\coordinate (nP1) at (2.75,-2);
\coordinate (P3) at (1.25,0);
\coordinate (nP3) at (-2.75,-2);
\draw (.5,-7.5) node{vertical fan};
\draw[fill,color=gray!30] (-5,-6) rectangle (5,4);
\maskA{gray!45}
\maskX{gray!35}
\maskB{white}
\maskC{white}
\maskCC{gray!50}
\begin{scope}
\clip (-1,1.8) rectangle (1,2.1);
\draw[very thick,color=black] (0,-.42) circle[radius=2.43cm];
\end{scope}
\begin{scope}[rotate=-63,xshift=.55cm]
\clip (0,-4) rectangle (-3.1,1.5);
\draw[color=black,very thick] (0,-1.14) ellipse[x radius=3cm,y radius=2.235cm];
\end{scope}
\begin{scope}[rotate=63,xshift=-.55cm]
\clip (0,-4) rectangle (3.1,1.5);
\draw[color=black,very thick] (0,-1.14) ellipse[x radius=3cm,y radius=2.235cm];
\end{scope}
	\draw[very thick] (-1.5,0.5) circle[radius=2.8cm];
	\draw[very thick] (1.5,0.5) circle[radius=2.8cm];
	\draw[very thick] (0,-2.5) circle[radius=2.8cm];
\draw[fill] (P1) ellipse[x radius=4.5pt,y radius=6pt] +(.8,-.3)node{\tiny$P_1[2]$};
\draw[fill] (P2) ellipse[x radius=4.5pt,y radius=6pt] +(0,-.2)node[below]{\tiny$P_2[2]$};
\draw[fill] (P3) ellipse[x radius=4.5pt,y radius=6pt] +(0,.2)node[ right]{\tiny$P_3[2]$};
\draw[fill] (nP1) ellipse[x radius=4.5pt,y radius=6pt] node[below right]{\tiny$P_1[1]$};
\draw[fill] (nP2) ellipse[x radius=4.5pt,y radius=6pt] +(0,.2)node[above]{\tiny$P_2[1]$};
\draw[fill] (nP3) ellipse[x radius=4.5pt,y radius=6pt] +(-.7,-.3)node{\tiny$P_3[1]$};
\draw[fill] (nI1) ellipse[x radius=4.5pt,y radius=6pt];
\draw[fill] (nI3) ellipse[x radius=4.5pt,y radius=6pt] +(-.3,.5)node{\tiny $I_3[1]$};
\draw[fill] (nS2) ellipse[x radius=4.5pt,y radius=6pt] ;
\draw (nI3) +(-1.5,.5)node{$Y$};
\draw (P1) +(-1,.5)node{$X$};
\draw[very thick, color=green!80!black,dashed,->] (P2)+(0,.7) .. controls  (-2,-1.8) and (-3,-1.5)..(-4,3.5); 
\end{scope}
}
\end{tikzpicture}
\caption{Horizontal fans are like floors of a building. Each floor is divided into triangular rooms (for $n=3$). Each room is labeled with a bounded $t$-structure, the corners are labeled with components of the corresponding silting object, the walls are labelled with corresponding simple-minded components. Red dotted circles are absent walls with simple labels. Green dashed lines are ``stairs'' connecting shaded rooms which form 5 out of 14 rooms in the vertical fan shown on right. The wall at the top of each flight of stairs is blue.}
\label{Figure00}\label{preview}
\end{center}
\end{figure}
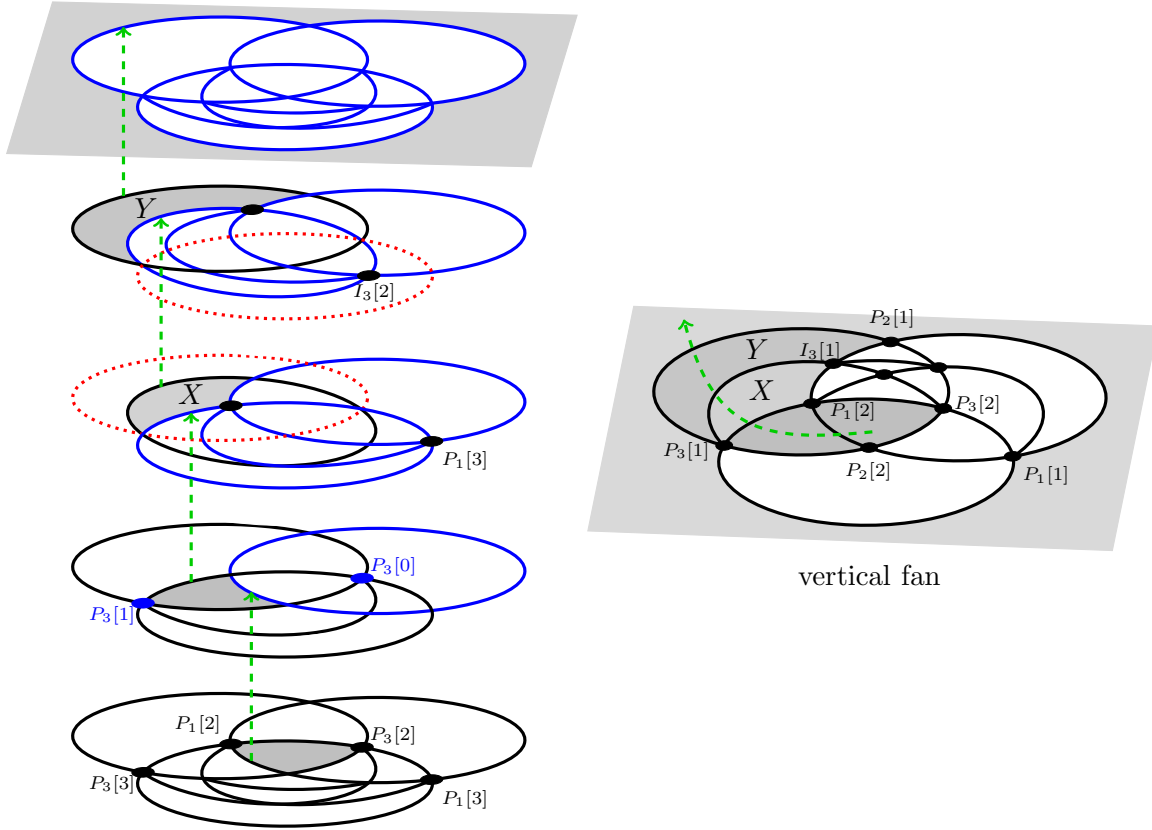

I would like to thank the Chern Institute and organizers of the Workshop for their hospitality and the participants and referees for their inspiring comments.

\section{Preliminaries}

The goal of this project is to visualize ``$m$-maximal green sequences''. First, we will go over several definitions of a standard maximal green sequence and demonstrate some valuable insights derived from the ``semi-invariant picture'' which are more difficult to see in the Hasse diagram of the poset of torsion classes which contains the same information. This point of view is expanded upon in my other papers which are listed at the end of these notes.

\subsection{Basic definitions}

Let $\Lambda$ be a finite dimensional hereditary algebra over a field $K$. We assume $\Lambda$ is basic so that the dimension of the simple module $S_i$ is equal to the dimension of the endomorphism ring of $S_i$ and $P_i$, its projective cover. Call this $f_i$.
\[
	f_i:=\dim_KS_i=\dim_K\End_\Lambda(P_i).
\]
Let $D$ be the $n\times n$ diagonal matrix with diagonal entries $f_i$. In the lecture, I assumed, for simplicity of notation, that $K$ is algebraically closed and $D=I_n$, the identity matrix.

For any finitely generated right $\Lambda$-module $M$ let $\undim M\in\NNN^n$ be the \emph{dimension vector} of $M$. The $i$th coordinate of $\undim M$ is the number of times that $S_i$ occurs in the composition series of $M$. 

Suppose that $M$ has a minimal projective presentation:
\[
	\coprod P_i^{b_i}\to \coprod P_i^{a_i}\to M.
\]
Then, the \emph{$g$-vector} of $M$ is the integer vector $g(M)\in \ZZ^n$ whose $i$th coordinate is $a_i-b_i$. For any two modules $X,Y$, the \emph{Euler-Ringel pairing} is given by the dot product:
\[
	Dg(X)\cdot \undim Y=g(X)^tD\undim Y=\dim_K\Hom_\Lambda(X,Y)-\dim_K\Ext^1_\Lambda(X,Y).
\]
One easy example is $g(P_i)=e_i$, the $i$th unit vector. The indecomposable objects in the bounded derived category of $mod\text-\Lambda$ are $M[k]$ where $M$ is an indecomposable module and $k\in\ZZ$. The $g$-vector of $M[k]$ is defined to be
\[
	g(M[k])=(-1)^kg(M).
\]
In particular, $g(P_i[1])=-g(P_i)=-e_i$.

\begin{defn}
A module $M$ is called \emph{Schurian} (or a \emph{brick}) if its endomorphism ring $\End_\Lambda(M)$ is a division algebra. Given such a module, the \emph{stability set} ${\bf D}_\Lambda(M)$ (also called \emph{semi-invariant domain} \cite{IOTW2}) is the subset of $\RR^n$ given by
\[
	{\bf D}_\Lambda(M):=\{x\in\RR^n\,|\, x\cdot \undim M=0,\, x\cdot \undim M'\le 0\,\text{ for all }M'\subsetneq M\}.
\]
$M$ is called \emph{exceptional} if it is Schurian and rigid (\emph{rigid} means without self-extensions). In the case when $\Lambda$ is the path algebra $\Lambda=KQ$ of a Dynkin quiver $Q$, all indecomposable modules are exceptional.
\end{defn}

These sets ${\bf D}_\Lambda(M)$, sometimes called ``walls'', play a key role in visualizing cluster-tilting objects in the cluster category of $\Lambda$. These walls divide $\RR^n$ into compartments corresponding to clusters and two adjacent compartments differ by a single mutation. I will explain this with an example.

\begin{eg}\label{eg: A2 example, part 1}
Let $\Lambda=KQ$ where $Q$ is the quiver of type $A_2$: $1\leftarrow 2$ and $f_1=f_2=1$. There are three indecomposable modules arranged in the Auslander-Reiten quiver by:
\[
\xymatrixrowsep{15pt}\xymatrixcolsep{10pt}
\xymatrix{
& P_2 \ar[rd]\\
	S_1 \ar[ru] \ar@{--}[rr]&&
	S_2
	}
\]
The stability sets ${\bf D}_\Lambda(S_1), {\bf D}_\Lambda(S_1)\subset \RR^2$ are the $y$ and $x$-axes respectively:
\[
	{\bf D}_\Lambda(S_1)=\{(x,y)\in \RR^2\,|\, (x,y)\cdot \undim S_1=(x,y)\cdot(1,0)=x=0\}=(1,0)^\perp
\]
\[
	{\bf D}_\Lambda(S_2)=\{(x,y)\in \RR^2\,|\, (x,y)\cdot \undim S_2=y=0\}=(0,1)^\perp.
\]
However, ${\bf D}_\Lambda(P_2)$ is only part of the hyperplane $(1,1)^\perp$ since $S_1\subset P_2$:
\[
	{\bf D}_\Lambda(P_2)=\{(x,y)\in \RR^2\,|\, (x,y)\cdot \undim P_2=x+y=0,\ (x,y)\cdot \undim S_1=x\le0\}.
\]
These walls are indicated in Figure \ref{A2 example 1 fig a}. This figure also includes the position of the $g$-vectors of the five indecomposable objects $P_1=S_1,P_2,S_2,P_1[1],P_2[1]$ of the cluster category of $\Lambda$.
\end{eg}

\subsection{Cluster category}

We recall \cite{BMRRT} that the \emph{cluster category} $\cC_\Lambda$ of $\Lambda$ is the orbit category of the bounded derived category of $mod\text-\Lambda$ by the functor $F=\tau^{-1}[1]$:
\[
	\cC_\Lambda=\cD^b(mod\text-\Lambda)/\tau^{-1}[1]
.\]
The indecomposable objects of $\cC_\Lambda$ are represented by objects in $\cD^b(mod\text-\Lambda)$ which are either indecomposable modules $M$ or shifted indecomposable projective modules $P_i[1]$.

\begin{defn}
A \emph{cluster-tilting object} in $\cC_\Lambda$ is defined to be an object $T\in \cC_\Lambda$ with $n$ nonisomorphic components $T=\bigoplus T_i$ so that $\Ext^1(T,T)=0$ in $\cC_\Lambda$. This is equivalent to saying that $T_i$ are either indecomposable modules, say $M_1,\cdots,M_m$, or shifted projective modules, say $P_{j_{m+1}}[1],\cdots,P_{j_{n}}[1]$, so that $\Ext_\Lambda^1(M_k,M_\ell)$ for all $k,l\le m$ and $\Hom_\Lambda(P_{j_i},M_k)=0$ for all $k\le m<i$. The module $M=M_1\oplus\cdots\oplus M_m$ is called a \emph{support tilting module}.
\end{defn}

Cluster-tilting objects and stability sets are related by the following results \cite{IOTW2}.

\begin{lem}
For exceptional modules $X,M$, $Dg(X)\in {\bf D}_\Lambda(M)$ if and only if
\[
	\Hom_\Lambda(X,M)=0=\Ext_\Lambda^1(X,M).
\]
Also, $-Dg(P_i)=Dg(P_i[1])\in {\bf D}_\Lambda(M)$ if and only if $\Hom_\Lambda(P_i,M)=0$.
\end{lem}

\begin{thm}\label{thm: cluster picture}
Let $T=T_1\oplus \cdots\oplus T_n$ be a cluster-tilting object in $\cC_\Lambda$. Then, there are unique exceptional modules $X_1,\cdots,X_n$ with the property that $Dg(T_i)\in {\bf D}_\Lambda(X_j)$ if and only if $i\neq j$. Futhermore, any positive linear combination of the $g$-vectors $Dg(T_i)$ does not lie in any ${\bf D}_\Lambda(M)$.
\end{thm}

\begin{proof} The first statement is well-known \cite{IOTW2}. The second statement is \cite{PartI}, Lemma A.
\end{proof}

Theorem \ref{thm: cluster picture} has the following interpretation. Given a cluster-tilting object $T=\bigoplus T_i$, the set of nonnegative linear combinations $\sum a_iDg(T_i)$, $a_i\ge0$ is bounded by $n$ walls ${\bf D}_\Lambda(X_i)$ and no walls meet the interior of this region. For example, in Figure \ref{A2 example 1 fig a}, the five regions correspond to the 5 cluster-tilting objects (ordered counterclockwise from lower left):
\[
	P_1[1]\oplus P_2[1],\quad P_2[1]\oplus P_1,\quad P_1\oplus P_2,\quad P_2 \oplus S_2,\quad S_2\oplus P_1[1].
\]

\begin{figure}[htbp] 
\begin{center}
\begin{tikzpicture}
\clip (-6,-2.3) rectangle (6,2.3);
\coordinate (S1) at (1.7,0);
\coordinate (S1s) at (-1.7,0);
\coordinate (S2) at (-1.7,1.7);
\coordinate (P2) at (0,1.7);
\coordinate (P2s) at (0,-1.7);
\draw[fill] (S1) circle[radius=3pt] node[below]{$g(P_1)$};
\draw[fill] (S1s) circle[radius=3pt] node[below]{$g(P_1[1])$};
\draw[fill] (S2) circle[radius=3pt] node[left]{$g(S_2)$};
\draw[fill] (P2) circle[radius=3pt] node[right]{$g(P_2)$};
\draw[fill] (P2s) circle[radius=3pt] node[right]{$g(P_2[1])$};
\begin{scope}
	\draw[thick] (-4,0)--(4,0) (-4,0) node[above]{${\bf D}_\Lambda(S_2)=(0,1)^\perp$};
	\draw[thick] (0,-3)--(0,3) (0,2.6) node[right]{${\bf D}_\Lambda(S_1)=(1,0)^\perp$};
	\draw[thick] (0,0)--(-2.5,2.5) node[left]{${\bf D}_\Lambda(P_2)$};
\end{scope}
\end{tikzpicture}
\caption{$g$-vectors of the components of each cluster-tilting object lie on the boundary of the corresponding region. For example, the upper right region corresponds to $P_1\oplus P_2$ with $g(P_1)$ on one wall and $g(P_2)$ on the other. In this example, $D=I_2$. So, $g(X)=Dg(X)$.}
\label{Figure11}\label{A2 example 1 fig a}
\end{center}
\end{figure}
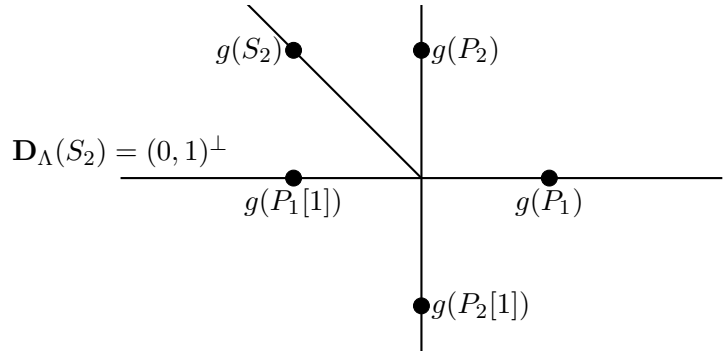
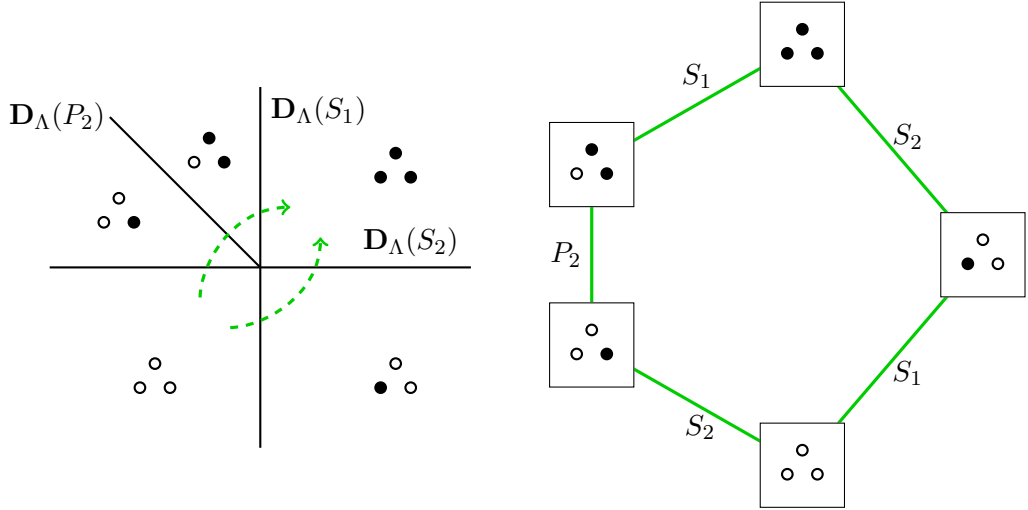
\begin{figure}[htbp]
\begin{tikzpicture}[scale=.8]
\clip (-4.3,-4) rectangle (4.5,3.1);
\begin{scope}
	\draw[thick] (-3.5,0)--(3.5,0) (2.5,0) node[above]{${\bf D}_\Lambda(S_2)$};
	\draw[thick] (0,-3)--(0,3) (0,2.6) node[right]{${\bf D}_\Lambda(S_1)$};
	\draw[thick] (0,0)--(-2.5,2.5) (-2.4,2.5) node[left]{${\bf D}_\Lambda(P_2)$};
\end{scope}
\begin{scope}[scale=.5,xshift=-4cm,yshift=-4cm]
	\draw[thick] (0,0) circle[radius=5pt];
	\draw[thick] (0.5,.8) circle[radius=5pt];
	\draw[thick] (1,0) circle[radius=5pt];
	\draw[very thick,dashed, color=green!80!black,->] (3,2) .. controls (4,2) and (6,3) .. (6,5);
	\draw[very thick,dashed, color=green!80!black,->] (2,3) .. controls (2,4) and (3,6) .. (5,6);
\end{scope}
\begin{scope}[scale=.5,xshift=4cm,yshift=-4cm]
	\draw[thick,fill] (0,0) circle[radius=5pt];
	\draw[thick] (0.5,.8) circle[radius=5pt];
	\draw[thick] (1,0) circle[radius=5pt];
\end{scope}
\begin{scope}[scale=.5,xshift=4cm,yshift=3cm]
	\draw[thick,fill] (0,0) circle[radius=5pt];
	\draw[thick,fill] (0.5,.8) circle[radius=5pt];
	\draw[thick,fill] (1,0) circle[radius=5pt];
\end{scope}
\begin{scope}[scale=.5,xshift=-2.2cm,yshift=3.5cm]
	\draw[thick] (0,0) circle[radius=5pt];
	\draw[thick,fill] (0.5,.8) circle[radius=5pt];
	\draw[thick,fill] (1,0) circle[radius=5pt];
\end{scope}
\begin{scope}[scale=.5,xshift=-5.2cm,yshift=1.5cm]
	\draw[thick] (0,0) circle[radius=5pt];
	\draw[thick] (0.5,.8) circle[radius=5pt];
	\draw[thick,fill] (1,0) circle[radius=5pt];
\end{scope}
\end{tikzpicture}
\begin{tikzpicture}[scale=.8]
\coordinate (S1) at (1.7,0);
\coordinate (S2) at (-1.7,1.7);
\coordinate (P2) at (0,1.7);
\coordinate (A) at (0,0);
\coordinate (B) at (-3.5,2);
\coordinate (C) at (-3.5,5);
\coordinate (D) at (3,3.5);
\coordinate (E) at (0,7);
\begin{scope}[xshift=.7cm,yshift=.7cm]
\draw[very thick,color=green!80!black] (0,0)--(-3.5,2)--(-3.5,5)--(0,7)--(3,3.5)--(0,0);
\draw (-1.7,1)node[below]{$S_2$};
\draw (-3.5,3.5)node[left]{$P_2$};
\draw (-1.75,6.1)node[above]{$S_1$};
\draw (1.75,1.9)node[below]{$S_1$};
\draw (1.75,5.1)node[above]{$S_2$};
\end{scope}
\foreach\x in {A,B,C,D,E} \draw[fill,color=white] (\x) rectangle +(1.4,1.4);
\foreach\x in {A,B,C,D,E} \draw (\x) rectangle +(1.4,1.4);
\begin{scope}[xshift=.1cm,yshift=.2cm] 
\begin{scope}[scale=.5,xshift=.7cm,yshift=.7cm] 
	\draw[thick] (0,0) circle[radius=5pt];
	\draw[thick] (0.5,.8) circle[radius=5pt];
	\draw[thick] (1,0) circle[radius=5pt];
\end{scope}
\begin{scope}[scale=.5,xshift=6.7cm,yshift=7.7cm] 
	\draw[thick,fill] (0,0) circle[radius=5pt];
	\draw[thick] (0.5,.8) circle[radius=5pt];
	\draw[thick] (1,0) circle[radius=5pt];
\end{scope}
\begin{scope}[scale=.5,xshift=.7cm,yshift=14.7cm] 
	\draw[thick,fill] (0,0) circle[radius=5pt];
	\draw[thick,fill] (0.5,.8) circle[radius=5pt];
	\draw[thick,fill] (1,0) circle[radius=5pt];
\end{scope}
\begin{scope}[scale=.5,xshift=-6.3cm,yshift=10.7cm] 
	\draw[thick] (0,0) circle[radius=5pt];
	\draw[thick,fill] (0.5,.8) circle[radius=5pt];
	\draw[thick,fill] (1,0) circle[radius=5pt];
\end{scope}
\begin{scope}[scale=.5,xshift=-6.3cm,yshift=4.7cm] 
	\draw[thick] (0,0) circle[radius=5pt];
	\draw[thick] (0.5,.8) circle[radius=5pt];
	\draw[thick,fill] (1,0) circle[radius=5pt];
\end{scope}
\end{scope}
\end{tikzpicture}
\caption{The poset of torsion classes on the right is dual to the semi-invariant picture on the left. The wall ${\bf D}_\Lambda(M)$ separates two regions on the left iff the brick $M$ labels the corresponding edge of the Hasse diagram. Maximal green sequences for $\Lambda$ are maximal chains in the poset, equivalent to ``green paths'' from lower left to upper right in the picture. See \cite{Woolf}.}
\label{Figure12}\label{A2 example 1 fig b}
\end{figure}
To understand maximal green sequences (green dashed arrows in Figure \ref{A2 example 1 fig b}), it helps to know that cluster-tilting objects for $\Lambda$ are in bijection with certain torsion classes. 

\begin{defn}
A \emph{torsion class} for $\Lambda$ is a full subcategory $\cG$ of $mod\text-\Lambda$ with the property that $\cG$ is closed under extension and quotient objects. We say that $\cG$ is \emph{finitely generated} if there exists a single module $G$ in $\cG$ so that $\cG$ is the class of all quotients of direct sums of $G$. This conditions is often stated as ``$\cG$ is \emph{functorially finite}'' which means that the inclusion functor $\cG\hookrightarrow mod\text-\Lambda$ has both a left and right adjoint.
\end{defn}

The condition of being closed under extensions and quotients implies that the ``trace'' of $\cG$ in any module $M$ is also an object of $\cG$. This is the sum of all images of all maps $X\to M$ where $X\in\cG$. This gives a functorial short exact sequence:
\[
	0\to A\to M\to B\to 0
\]
where $A\in \cG$ and $B\in\cF$ where
\[
	\cF:=\{B\in mod\text-\Lambda\,|\, \Hom_\Lambda(X,B)=0, \forall X\in\cG\}.
\]
This implies that $\cG\hookrightarrow mod\text-\Lambda$ always has a right adjoint, namely, $M\mapsto A$. 

For $\Lambda$ of finite type, all torsion classes are clearly finitely generated. For $\Lambda$ of infinite type, the full subcategory of preinjective $\Lambda$-modules is a torsion class which is not finitely generated.

\begin{thm}\cite{ST}
There is a 1-1 correspondence between (isomorphism classes of) cluster-tilting objects in $\cC_\Lambda$ and functorially finite torsion classes for $\Lambda$.
\end{thm}

The following chart \eqref{eq: chart for A2} gives the bijection in our example $A_2:1\leftarrow 2$.
\begin{equation}\label{eq: chart for A2}
\begin{array}{|c|c|c|}
\hline
\text{cluster tilting object} & \text{torsion class} & \text{spot diagram}\\
 \hline
 P_1[1]\oplus P_2[1] & 0 & %
\begin{tikzpicture}[scale=.5]
	\draw[thick] (0,0) circle[radius=5pt];
	\draw[thick] (0.5,.8) circle[radius=5pt];
	\draw[thick] (1,0) circle[radius=5pt];
\clip (-.5,-.3) rectangle (1.5,1.3);
\end{tikzpicture}\\
\hline
 P_2[1]\oplus P_1 & Gen\,S_1 & \begin{tikzpicture}[scale=.5]
	\draw[thick,fill] (0,0) circle[radius=5pt];
	\draw[thick] (0.5,.8) circle[radius=5pt];
	\draw[thick] (1,0) circle[radius=5pt];
\clip (-.5,-.3) rectangle (1.5,1.3);
\end{tikzpicture}\\
\hline
 P_1\oplus P_2 & mod\text-\Lambda &\begin{tikzpicture}[scale=.5]
	\draw[thick,fill] (0,0) circle[radius=5pt];
	\draw[thick,fill] (0.5,.8) circle[radius=5pt];
	\draw[thick,fill] (1,0) circle[radius=5pt];
\clip (-.5,-.3) rectangle (1.5,1.3);
\end{tikzpicture}\\
\hline
 P_2 \oplus S_2 & Gen\,P_2 & \begin{tikzpicture}[scale=.5]
	\draw[thick] (0,0) circle[radius=5pt];
	\draw[thick,fill] (0.5,.8) circle[radius=5pt];
	\draw[thick,fill] (1,0) circle[radius=5pt];
\clip (-.5,-.3) rectangle (1.5,1.3);
\end{tikzpicture}\\
\hline
S_2\oplus P_1[1] & Gen\,S_2 & \begin{tikzpicture}[scale=.5]
	\draw[thick] (0,0) circle[radius=5pt];
	\draw[thick] (0.5,.8) circle[radius=5pt];
	\draw[thick,fill] (1,0) circle[radius=5pt];
\clip (-.5,-.3) rectangle (1.5,1.3);
\end{tikzpicture}\\
\hline
 \end{array}
\end{equation}

In this chart, $Gen\,M$ is the torsion class generated by $M$, i.e., the class of all quotients of $M^k$ for all $k$. The spot diagrams, also used by Iyama in his lecture, indicate the indecomposable objects (dark spots) which lie in the torsion class. For example, 

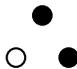
\begin{center}
\begin{tikzpicture}[scale=.7]
	\draw[thick] (0,0) circle[radius=5pt];
	\draw[thick,fill] (0.5,.8) circle[radius=5pt];
	\draw[thick,fill] (1,0) circle[radius=5pt];
\end{tikzpicture}
\end{center}
Indicates the torsion class with indecomposable objects $P_2,S_2$ excluding $S_1=P_1$.

\subsection{Maximal green sequences}

We give several equivalent definitions of a maximal green sequence. For historical remarks including the relationship between MGSs, quantum dilogarithms and Bridgeland stability conditions, see Section \ref{quantum} at the end of these notes.

\begin{thm}\label{thm: def of MGS}
Let $\Lambda$ be a finite dimensional hereditary algebra over a field $K$ of rank $n$.
Let $\beta_1,\cdots,\beta_m\in \NNN^n$. Then the following are equivalent.
\begin{enumerate}
\item $\beta_1,\cdots,\beta_m$ are the $c$-vectors of a maximal green sequence for $\Lambda$ defined by a Fomin-Zelevinsky mutation sequence on positive $c$-vectors.
\item $\beta_i$ are (positive) real Schur roots and the unique modules $M_i$ with $\undim M_i=\beta_i$ have the following properties.
	\begin{enumerate}
		\item $\Hom_\Lambda(M_i,M_j)=0$ for $i<j$.
		\item $(M_i)$ is maximal with property (a), i.e., for any other indecomposable module $M$, there exist $ i<j$ so that $\Hom(M_i,M)\neq0$ and $\Hom(M,M_j)\neq0$.
	\end{enumerate}
\item There is a generic green path in $\RR^n$ which crosses the semi-invariant domains ${\bf D}_\Lambda(M_i)$ for $i=1,\cdots,m$ in increasing order of $i$.
\item There is a finite maximal chain in the poset of functorially finite torsion classes on $\Lambda$ labeled with bricks $M_1,\cdots,M_m$.
\end{enumerate}
\end{thm}

By a \emph{generic green path} we mean a $C^1$ path $\gamma:\RR\to\RR^n$ which starts in the negative octant (all coordinates of $\gamma(t)$ are negative for $t<<0$), ends in the positive octant and crosses only finitely many walls ${\bf D}_\Lambda(M)$ at distinct times in the positive direction, i.e., so that $\frac{d\gamma}{dt}(t_0)\cdot \undim M>0$ when $\gamma(t_0)\in {\bf D}_\Lambda(M)$. See \cite{PartI}.

\begin{proof}
The equivalence between (1),(2),(3) is shown in \cite{PartI}. The equivalence of these with (4), leaving aside the brick labels for a moment, is well-known and due to Speyer and Thomas \cite{ST}. Figures similar to Figure \ref{Figure12} appear in \cite{Woolf}. The fact that the modules $M_k$ with $\undim M_k=\beta_k$ are the brick labels, call them $B_k$, of the corresponding edge in the Hasse diagram of functorially finite torsion classes follows from a lemma mentioned by Iyama in his lecture:
\[
	B_k=T_k/rad_{\End(T)}T_k.
\]	
Here $T=\bigoplus T_i$ is a support tilting module (same as support $\tau$-tilting module in the hereditary case) corresponding to the torsion class generated by the extension closure of $\{M_1,\cdots,M_k\}$. Since $T_k$ is an exceptional $\Lambda$-module, $rad_{\End(T)}T_k$ is equal to the trace in $T_k$ of the other components $T_i$, $i\neq k$ of $T$. This formula implies that $\Hom_\Lambda(T_i,B_k)=0$ for $i\neq k$. Also, $\Ext_\Lambda^1(T_i,B_k)=0$ for $i\neq k$ by right exactness of $\Ext_\Lambda^1$. Finally, the support of $B_k$ is contained in the support of $T_k$. So, $\Ext^1(P_s[1],B_k)=0$ for all $s$ not in the support of $T_k$. This implies the stability wall ${\bf D}_\Lambda(B_k)$ contains ($D$ times) the $n-1$ $g$-vectors of $T_i,P_s[1]$. (See \cite{IOTW2}.). Thus ${\bf D}_\Lambda(B_k)={\bf D}_\Lambda(M_k)$ making $B_k\in add\, M_k$. Since $B_k$ is a brick, it must be equal $M_k$.
\end{proof}

There are two very useful consequences of the wall-crossing description of a maximal green sequence. One is an observation due to Greg Muller \cite{Muller}:

\begin{cor} Let $\Lambda=KQ$ be a path algebra and let $\Lambda'$ be an algebra obtained from $\Lambda$ by deleting arrows from $KQ$ and adding relations. Then, any generic green path for $\Lambda$ is also a generic green path for $\Lambda'$. In particular, when $\Lambda'$ has no relations, the $c$-vectors of any maximal green sequence for $\Lambda$ give the $c$-vectors of a maximal green sequence for $\Lambda'$ when those vectors which are not real Schur roots of $\Lambda'$ are deleted.
\end{cor}

\begin{proof}
Since $Q'\subseteq Q$ be the quiver of $\Lambda'$. Then any $\Lambda'$-module $M$ is also a $KQ'$-module which extends to a $KQ$-module by letting the other arrows act as 0. We have $\End_{\Lambda'}(M)=\End_\Lambda(M)$ and the $\Lambda'$-submodules $M'\subseteq M$ give $\Lambda$-submodules of $M$. Therefore:
\[
	{\bf D}_{\Lambda'}(M)={\bf D}_\Lambda(M).
\]
In other words, every wall for $\Lambda'$ is a wall for $\Lambda$ with the same dimension vector associated to it. The other walls of $\Lambda$ simply disappear. So, a generic green path $\gamma$ for $\Lambda$ will be a (generally shorter) generic green path for $\Lambda'$ passing through a subset of the original set of walls.
\end{proof}

Another valuable insight is the visualization of maximal green sequences \cite{BHIT}, \cite{ITW16}, \cite{Muller}. Figure \ref{FigA2tilde} shows the union $\bigcup {\bf D}_\Lambda(M)$ of all stability sets for $\Lambda=KQ$ of type $\tilde A_2$ (excluding an infinite number of stability sets accumulating onto the red line):
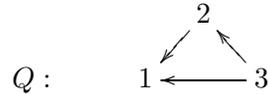
\[
\xymatrixrowsep{10pt}\xymatrixcolsep{10pt}
\xymatrix{
&& & 2\ar[dl]\\
Q: &&1 && 3\ar[ll]\ar[ul]
	}
\]
The idea behind this picture was generalized in the paper \cite{BHIT} to show that Figure \ref{FigA2tilde} is typical for quivers of tame type and thus there are only finitely many MGSs.
\begin{figure}[htbp]
\begin{center}
\begin{tikzpicture}
\allblackC
\end{tikzpicture}
\caption{Perspective view of the stereographic projection of the intersection with $S^2$ of the ``picture'' $\bigcup {\bf D}_\Lambda(M)\subset \RR^3$ for $\Lambda=KQ$ of type $\tilde A_2$. The red line is the domain of the null root. This is the accumulation set of an infinite family of stability sets making the shaded region impassible. The curvature of the lines shows the positive direction in which green paths must be pointed. So, we  see that there are only 5 maximal green sequences: A,B,C,D,E.}
\label{FigA2tilde}
\end{center}
\end{figure}
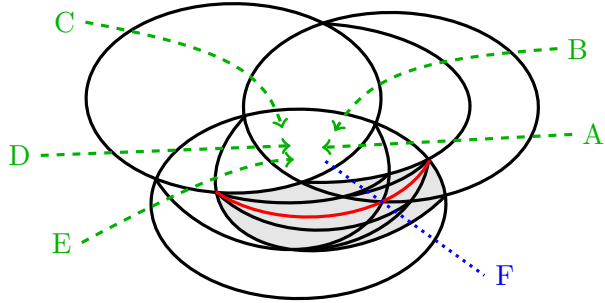

{
Chart \eqref{eq: 5 MGS for A2 tilde} lists the five MGSs for $\tilde A_2$ shown in green in Figure \ref{FigA2tilde} and the infinite path $F$ which is blue in that figure. The mutation sequences are indicated. For example, $B$ is the mutation sequence $\mu_2,\mu_1,\mu_2,\mu_3$. The modules listed are those that label the walls crossed by the paths. For example, $C$ crosses the walls ${\bf D}_\Lambda(S_1),{\bf D}_\Lambda(S_2),{\bf D}_\Lambda(S_3)$. The infinite path $F$ crosses the walls ${\bf D}_\Lambda(M)$ for all preinjective modules followed by the regular walls ${\bf D}_\Lambda(S_2),{\bf D}_\Lambda\binom31$ and the null wall being crossed simultaneously, then ${\bf D}_\Lambda(M)$ for all preprojective $M$.
\begin{equation}\label{eq: 5 MGS for A2 tilde}
\begin{array}{|c|l|c|l|}
\hline
 & \text{mutation sequence} &\text{length} & \text{sequence of modules}\\
 \hline
 A & 2,1,3,2,3 & 5 & S_2, P_2, S_3, \binom31, S_1\\
 \hline
 B & 2,1,2,3 & 4 & S_2, P_2, S_1, S_3\\
  \hline
C & 1,2,3 & 3 & S_1, S_2, S_3\\
 \hline
 D & 1,3,2,3 & 4 & S_1, S_3, I_2, S_2\\
  \hline
E & 3,1,3,2,1 & 5 & S_3, \binom31, S_1, I_2, S_2\\
 \hline
 F & 3,2,3,2,3,\cdots, 1,2,1,2 & \infty & S_3, I_2, I_1, \tau S_3,\cdots, R, \cdots, \tau^{-1}S_1, P_3,P_2,S_1\\
\hline
\end{array}
\end{equation} 
}

\begin{defn}
The \emph{semi-invariant picture} (also called ``cluster fan'' or ``scattering diagram'') for $\Lambda$ is defined to be the subset 
\[
	L(\Lambda):=\bigcup {\bf D}_\Lambda(M)\subset \RR^n
\]
where the union is over all Shurian $\Lambda$-modules $M$.
\end{defn}

Scattering diagrams have become very popular due to the work of Gross, Hacking, Keel and Kontsevich who used them to prove the positivity conjecture, sign coherence and other conjectures for cluster algebras \cite{GHKK}. An elementary proof of the existence of a consistent scattering diagram, i.e., a semi-invariant picture, for acyclic valued quivers is given in \cite{IOTW3}.

\section{$m$-maximal green sequences}

We continue to let $\Lambda$ denote a basic finite dimensional hereditary algebra over a field $K$. Let $m\ge1$. There are several equivalent definitions of an $m$-maximal green sequence. In my lecture I used the technical definition that it is a maximal sequence of ``backward mutations'' in the derived category of $mod$-$\Lambda$ starting with $\Lambda[m]$ and ending with $\Lambda[0]$. (See \cite{KQ}, Sec 4 or \cite{Q}, Sec 3.3.). After subsequent lectures by Iyama, Demonet, Buan and discussion with Zhu Han it became clear that the definition in terms of $t$-structures might be more appropriate. I will give both of these here. 

\subsection{$m$-clusters} We review the definitions of an $m$-cluster and two other types of objects which are in bijection with $m$-clusters. The aim is to explain the definition of an $m$-maximal green sequence and the reason that each term in such a sequence has three types of labels. A popular reference is \cite{KY}, but we take the point of view found in \cite{BRT} and \cite{KQ}. We use the base case $m=0$ as a very easy example of each.

The original definition of the \emph{$m$-cluster category} of $\Lambda$ \cite{Tm} is as an orbit category of the bounded derived category of $mod$-$\Lambda$:
\[
	\cC_m(\Lambda):=\cD^b(mod\text-\Lambda)/\tau^{-1}[m]
\]
where $\tau$ is the Auslander-Reiten translation. An \emph{$m$-cluster-tilting object} is an object $T$ of this category with $n$ nonisomorphic components so that
\[
	\Ext_{\cD^b}(T,T[\ell])=0
\]
for $0<\ell\le m$. The key point is that, if one component $T_k$ of $T$ is removed, there are, up to isomorphism, exactly $m$ ways to replace it with something different, i.e., there are $m$ mutations of $T$ in the $k$-direction. For $m=1$, this describes cluster-tilting objects in the cluster category $\cC_1(\Lambda)=\cC(\Lambda)$ of $\Lambda$ \cite{BMRRT}.

One convenient way to look at $m$-cluster-tilting objects is to represent them by objects of the fundamental domain $\cF D_m\subset \cD^b(mod\text-\Lambda)$ of the functor $\tau^{-1}[m]$ given by
\[
	\cF D_m:=mod\text-\Lambda\smallcoprod mod\text-\Lambda[1]\smallcoprod \cdots \smallcoprod mod\text-\Lambda[m-1]\smallcoprod \Lambda[m].
\]
This is the additive full subcategory of $\cD^b(mod\text-\Lambda)$ whose indecomposable objects are $M[\ell]$ where $M$ is an indecomposable $\Lambda$-module and $0\le \ell\le m$ with the additional restriction that $M$ is projective when $\ell=m$. We call $\ell$ the \emph{level} of $M[\ell]$. Recall that $M[\ell]$ is \emph{exceptional} if $M$ is an exceptional $\Lambda$-module.

\begin{defn}
An \emph{$m$-cluster} (also called a \emph{silting object} in $\cF D_m$) is an object $T\in\cF D_m$ having $n$ nonisomorphic exceptional summands $T_k$ so that $\Hom_{\cD^b}(T,T[\ell])=0$ for $\ell>0$.
\end{defn}

In the base case $m=0$, $\cF D_0$ consists only of projective modules in level 0 and $T=\Lambda$ is the only 0-cluster.

In \cite{BRT} it is shown that (isomorphism classes of) $m$-clusters are in 1-1 correspondence with ``$m$-$\Hom_{\le0}$-configurations'' defined as follows.

\begin{defn}
An \emph{$m$-$\Hom_{\le 0}$-configuration}, or simply \emph{$m$-configuration}, (also called a \emph{simple minded collection}) is a set of $n$ exceptional objects $E_1[\ell_1],\cdots,E_n[\ell_n]$ in $\bigcup_{\ell=0}^m mod\text-\Lambda[\ell]$ so that, for all $i\neq j$ and all $k\le0$ we have
\[
	\Hom_{\cD^b}(E_i[\ell_i],E_j[\ell_j+k])=0
\]
and so that the exceptional $\Lambda$-modules $E_1,\cdots,E_n$ form an exceptional sequence in some order.
\end{defn}

In the base case $m=0$, $\ell_k=0$ for all $k$ and $E_k\in mod\text-\Lambda$ must be the $n$ simple $\Lambda$-modules.

Finally, we recall the definition of a bounded $t$-structure.

\begin{defn}
A \emph{$t$-structure} in a triangulated category $\cD$ is a pair of additive full subcategory $(\cP,\cQ)$ so that
\begin{enumerate}
\item $\cP[1]\subset \cP$ and $\cQ\subset \cQ[1]$
\item $\Hom_\cD(P,Q)=0$ for all $P\in\cP,Q\in\cQ$ and
\item For every $X\in \cD$ there is a distinguished triangle
\[
	P\to X\to Q\to P[1]
\]
where $P\in\cP,Q\in\cQ$.
\end{enumerate}
\end{defn}

$\cP$ is called an \emph{aisle} \cite{KV}. It determines the $t$-structure since $\cQ$ is the full subcategory of all $X\in \cD$ so that $\Hom_\cD(P,X)=0$ for all $P\in\cP$. We use the aisle to give the partial ordering on $t$-structures. Thus $(\cP,\cQ)\le (\cP',\cQ')$ iff $\cP\subseteq\cP'$.

The \emph{heart} of a $t$-structure $(\cP,\cQ)$ is $\cP\cap \cQ[1]$. This is always an abelian category \cite{BBD}. A $t$-structure is call \emph{bounded} if every object of $\cD$ is in $\cP[-k]\cap \cQ[k]$ for some integer $k$.

We are interested in the $t$-structures whose hearts $\cH$ are \emph{length categories}, i.e., so that the Jordan-H\"older Theorem holds in $\cH$. In $\cD^b(mod\text-\Lambda)$, there is a 1-1 correspondence between bounded $t$-structures with length hearts and simple-minded collections. The correspondence sends a $t$-structure to the collection of simple objects in its heart. See \cite{KY}.

\subsection{Graded tropical duality} We need graded $g$-vectors and ``slope vectors'' to formulate the graded version of the tropical duality formula of \cite{NZ}.

\begin{defn}
The \emph{graded $g$-vector} of $M[k]$ is defined to be the vector $\tilde g(M[k])\in\ZZ[t]^n$ whose $i$th coordinate is $t^k(n_i-m_i)$ where
\[
	\coprod P_i^{m_i}\to \coprod P_i^{n_i}\to M\to 0
\]
is the {minimal} projective presentation of $M$. Thus
\[
	D\tilde g(M[k])=t^kf_i(n_i-m_i)_i\in \ZZ[t]^n
\]
where $D$ is the diagonal matrix with diagonal entries $f_i=\dim_K\End_\Lambda(P_i)$.
\end{defn}

\begin{defn}
The \emph{graded $c$-vector} of $N[k]$ is
\[
	\tilde c(N[k])=t^{m-k}\undim N\in \ZZ[t]^n.
\]
We define the \emph{slope} of $N[k]$ to be $m-k$ and $k$ is its \emph{level} (or \emph{degree}) if $N\in mod\text-\Lambda$. (The name comes from the fact that $m-k_i$ is the slope of the $i$th edge in the $m$-noncrossing tree corresponding to an $m$-cluster \cite{next}.)
\end{defn}

\begin{thm}[graded tropical duality]\label{thm: graded tropical}
Let $T=\coprod T_i$ be an $m$-cluster tilting object (``silting object'') and let $X=\coprod X_j$ be the corresponding $m$-$\Hom_{\le 0}$-configuration (``simple minded collection''). Then, we have:
\begin{equation}\label{eq: graded tropical}
	\tilde g(T_i)^tD \tilde c(X_j)=t^mf_i\delta_{ij}\text{ or }-t^{m-1}f_i\delta_{ij}.
\end{equation}
Furthermore, $T$ and $X$ uniquely determine each other by this relation.
\end{thm}

\begin{rem}
In my lecture, I assumed, for simplicity of notation, that $K=\overline K$ and $D=I_n$, the identity matrix. Then this formula is a dot product:
\begin{equation}\label{eq: graded tropical, simply laced}
	\tilde g(T_i)\cdot \tilde c(X_j)=t^m\delta_{ij}\text{ or }-t^{m-1}\delta_{ij}
\end{equation}
which implies that the \emph{$g$-vectors} (the values of $\tilde g(T_i)$ at $t=-1$) lie in the hyperplane perpendicular to $c(X_j)$ for $i\neq j$. In general, it is $D g(T_i)$ which lies in this hyperplane.
\end{rem}

\begin{proof}
As stated, this theorem follows from the proof in \cite{BRT}. In the Garside braid move which transforms $T_i$ into $X_i$, the degree of $X_i$ changes at most once and when it does it increases by one. (When $X_i$ changes level, it becomes relatively projective and stays that way. When $T_i$ is projective, the degree cannot shift by this process.) So, the slope decreases by zero or one which gives the two cases. (See \cite{enumerate} for more details.)

There is a very nice explanation of this in \cite{KQ} in the simply laced case. King and Qiu show in that case that $T_i$ are the projective objects of the heart $\cH$. Since $X_i$ are the simple objects of $\cH$ we get:
\begin{equation}\label{KQ duality}
	\Hom_{\cD^b}(T_i,X_j)=\delta_{ij}K
\end{equation}
in the algebraically closed case. Since $T_i=M_i[\ell_i]$ is a projective complex in degrees $\ell_i$ and $\ell_i+1$ (making $\tilde g(T_i)$ a vector in $t^{\ell_i}\ZZ^n$), the simple-minded component $X_i=N_i[k_i]$ must be in degree $k_i=\ell_i$ or $\ell_i+1$ (putting it at level $m-\ell_i$ or $m-\ell_i-1$). So, the formula \eqref{KQ duality} is equivalent to \eqref{eq: graded tropical, simply laced}.

To see that $T,X$ determine each other, we first set $t=-1$. Then the equation becomes
\[
	 g(T)^tDc(X)=(-1)^mD
\]
which gives, e.g., $g(T)=(-1)^mD(c(X)^t)^{-1}D^{-1}$. Since exceptional modules are detemined by their dimension vectors, this gives each $T_i$ up to its degree. The equation \eqref{eq: graded tropical} gives only two possibilities for the level of $T_i$ but only one of them has the correct sign. So, each $T_i$ is uniquely determined by $X$. Similarly, the $X_j$ are uniquely determined by $T$.
\end{proof}

\subsection{$m$-maximal green sequences} In parallel with the discussion of standard maximal green sequences as maximal chains in the poset of functorially finite torsion classes in $mod\text-\Lambda$, an $m$-MGS can be described as a maximal chain in the poset of bounded $t$-structures with length heart starting with (aisle equal to)
\[
	\cD^b_m(\Lambda)=\bigcup_{k\ge m} mod\text-\Lambda[k]
\]
which is the smallest aisle containing $\Lambda[m]$ and ending with $\cD^b_0(\Lambda)$.

In my lecture, I defined an $m$-maximal green sequence to be a sequence of ``negative mutations'' of $m$-clusters starting with $\Lambda[m]$ and ending with $\Lambda=\Lambda[0]$. Then I reinterpreted this as a sequence of ``positive mutations'' of $m$-configurations starting with $X=\bigoplus S_i[m]$ and ending with $\bigoplus S_i$. This is ``positive'' when using the slope $m-k$ instead of the level $k$ of $X[k]$. These three notions of an $m$-maximal green sequence are equivalent since the bijection between bounded $t$-structures with length heart, silting objects and simple-minded objects respects the partial ordering (see \cite{KY}).

Various formulas for these mutations are known (\cite[Sec 3]{BRT}, \cite[Sec 5]{IY}, \cite[Sec 7.2]{KY}). I use the numerical formula (Definition \ref{def: positive mutation of m=clusters}) equivalent to the one in \cite[Prop 5.4]{KQ}. This is a modified version of the Fomin-Zelevinsky mutation formula for extended exchange matrices \cite{FZ} using the tropical duality formula $g(T)^tDc(X)=D$ of \cite{NZ}. 

An \emph{exchange matrix} is defined to be an $n\times n$ skew symmetrizable integer matrix $B$, i.e., $DB$ is skew-symmetric for some positive diagonal matrix $D$. The matrix $D$ is fixed, but $B$ is mutable. The starting value of $B$ will be denoted $B_0$ and called the \emph{initial exchange matrix}. An \emph{extended exchange matrix} will be a $(2n+1)\times n$ integer matrix $\widetilde B$ with initial value $\widetilde B_0$:
\[
	\widetilde B=\mat{B\\
	\hline |C|\\
	\hline s},\qquad \widetilde B_0=\mat{B_0\\
	\hline I_n\\
	\hline 0}
\]
where $B$ is an exchange matrix, $|C|$, is an $n\times n$ matrix with nonnegative integer entries and the last row $s=(s_1,\cdots,s_n)$ is an integer vector with entries $0\le s_i\le m$. The \emph{$\tilde c_j$-vector} is defined to be
\[
	\tilde c_j=t^{s_j} |c_j|\in \ZZ[t]^n
\]
where $|c_j|$ is the $j$th column of $|C|$. We use the notation $c_j=(-1)^{s_j}|c_j|$ and we let $C$ be the matrix with $j$th column $c_j$. The initial extended exchange matrix $\widetilde B_0$ uses the initial exchange matrix $B_0$, $C=I_n$ the identity matrix and $0$ indicates the $1\times n$ null matrix.

The description above only gives the general format of the $(2n+1)\times n$ matrices $\widetilde B$. The mutation rules given in the following definition will produce these extended exchange matrices out of the initial matrix $\widetilde B_0$. The claim (theorem) is that, when the corresponding mutations are performed on the initial $m$-configuration consisting of the simple $\Lambda$-modules shifted by $m$: $S_i[m]\in mod\text-\Lambda[m]$, the $m$-configuration obtained will be $X=\bigoplus E_i[k_i]$, the bottom row of $\widetilde B$ will have entries $s_i:=m-k_i$ and the $i$th column of the middle part $|C|$ of $\widetilde B$ will be the dimension vector of $E_i$.

\begin{defn}\label{def: positive mutation of m=clusters}
Let $\widetilde B$ be as above so that $B=D^{-1}C^t B_0C$. When $s_k<m$, the \emph{positive mutation} $\widetilde B'=\mu_k^+\widetilde B$ of $\widetilde B$ in the \emph{$k$-th direction} is defined as follows.
\begin{enumerate}
\item $\mu_k^+$ increases the slope $s_k$ of $\tilde c_k$ by 1: $\tilde c_k'=t\tilde c_k$.
\item If $s_j\neq s_k,s_k+1$ or if $b_{kj}\le0$ then $\tilde c_j$ does not change: $\tilde c_j'=\tilde c_j$.
\item {If $s_j=s_k$ and $b_{kj}>0$, then
\[
	\tilde c_j'=\tilde c_j+\tilde c_k b_{kj}
\]
with the same slope $s_j'=s_j=s_k$.
}
\item {If $s_j=s_k+1$ and $b_{kj}>0$, there are two cases:}
\begin{enumerate}
\item {If $|c_j|-|c_k|b_{kj}>0$ then
\[
	\tilde c_j'=\tilde c_j-t\tilde c_k b_{kj}
\]
with the same slope $s_j'=s_j=s_k+1$.}
\item {If $|c_j|-|c_k|b_{kj}<0$ then
\[
	\tilde c_j'=\tilde c_k b_{kj}-t^{-1}\tilde c_j
\]
with slope $s_j'=s_k=s_j-1$.}
\end{enumerate}
\item The new value of $B$ is $B'=D^{-1}(C')^tDB_0C'$ where $C'$ is the matrix with $j$th column $c_j'$, the value of $\tilde c_j'$ at $t=-1$.
\end{enumerate}
\end{defn}

\begin{rem}
In module-theoretic terms, the mutation is given as follows. (3.2, 3.3 in \cite{BRT}.)
\begin{enumerate}
\item If the $k$th object of the $m$-configuration is $X_k=E_k[m-s_k]$, positive mutation replaces this with its negative shift: $X_k'=E_k[m-s_k-1]$.
\item Objects $X_j=E_j[\ell]$ for $\ell$ not equal to $m-s_k$ or $m-s_k-1$ are unchanged.
\item When $b_{kj}>0$, objects $X_j=E_j[m-s_k]$, with the same slope as $X_k$, are replaced with $X_j'=E_j'[m-s_k]$ where $E_j'$ is the universal extension
\[
	0\to E_j\to E_j'\to E_k^r\to 0
\]
where $r=b_{kj}=\dim_{End(E_k)}\Ext^1(E_k,E_j)$. If $b_{kj}\le0$ then $X_j'=X_j$.
\item For objects $X_j=E_j[m-s_k-1]$ with $b_{kj}\le0$ we have $X_j'=X_j$. Otherwise, consider the universal morphism
\[
	\varphi: E_k^r\to E_j 
\]
where $r=b_{kj}=\dim_{End(E_k)}\Hom(E_k,E_j)$. $\varphi$ is a monomorphism or an epimorphism.
\begin{enumerate}
	\item If $\varphi$ is a monomorphism then $E_j'=\coker\varphi$ and $X_j'=E_j'[m-s_k-1]$ has the same slope as $X_j$.
	\item If $\varphi$ is an epimorphism then $E_j'=\ker\varphi$ and $X_j'=E_j[m-s_k]$ has slope $s_j'=s_j-1$.
\end{enumerate}
\end{enumerate}
Proof: By definition of an $m$-configuration, the modules $E_j$ form an exceptional sequence in increasing order of slope (decreasing level $\ell_j$ if $X_j=E_j[\ell_j]$). Under the mutation $\mu_k^+$, the slope of $X_k=E_k[m-s_k]$ is increased by 1. This moves $E_k$ to the right in the exceptional sequence past terms of slope $s_k,s_k+1$. Thus $E_j$ of those two slopes are changed by exceptional sequence mutation as outlined in (3), (4) and those of other slopes are unchanged as stated in (2).
\end{rem}

Given a hereditary algebra $\Lambda$, its \emph{Euler matrix} $E$ is defined to have entries
\[
	e_{ij}=\dim_K\Hom_\Lambda(S_i,S_j)-\dim_K\Ext^1_\Lambda(S_i,S_j).
\]
The diagonal matrix $D$ has entries $f_i=\dim_K\End_\Lambda(S_i)$. The exchange matrix of $\Lambda$ is
\[
	B_\Lambda=D^{-1}(E^t-E).
\]
This is a skew-symmetrizable integer matrix. When restricted to $m=1$, Definition \ref{def: positive mutation of m=clusters} gives the formula for ``green'' mutations of the clusters in the cluster category of $mod\text-\Lambda$. (See \cite{IOTW2}.) We are claiming that Definition \ref{def: positive mutation of m=clusters} is the extension of this formula to $m$-clusters. For the proof of this claim, see \cite{enumerate}.

\begin{eg}
Let $Q$ be the quiver $1\leftarrow 2$. Then $D=I_2$ and $E=\mat{1 & 0\\ -1&1}$. So, $B_0=E^t-E=\mat{0 & -1\\ 1&0}$. Let $m=2$. Then we can perform $\mu_1^+$ twice:
\begin{center}
\begin{tikzpicture}
\coordinate (B) at (0,1.8);
\coordinate (C) at (0,0.5);
\coordinate (Cslope) at (0.3,-0.45);
\coordinate (Vslope) at (0.3,-2.3);
\coordinate (V) at (0,-1.4);
\coordinate (M) at (1.7,0);
\begin{scope}
\draw (3.4,.5)  node{$\xrightarrow{\mu_1^+}$};
	\draw (M) node{$\left[
	\begin{array}{rrr} 
	0 & -1 \\
	1 & 0 \\ 
	\hline
	 \color{green!66!black}1 & \color{black} 0 \\
	 \color{green!66!black}0 & \color{black} 1 \\
	\hline
	 \color{green!66!black}0 &\color{black} 0 \\
	\end{array}\right]$};
\end{scope}
\begin{scope}[xshift=3.5cm,yshift=0cm] 
	\draw (3.4,.5)  node{$\xrightarrow{\mu_1^+}$};
	\draw (1.7,0) node{$\left[
	\begin{array}{rrr} 
	0 & 1 \\
	-1 & 0 \\ 
	\hline
	\color{green!66!black}1 &  0 \\
	\color{green!66!black}0 &  1 \\
	\hline
	\color{green!66!black}1 &  0 \\
	\end{array}\right]$};
\end{scope}
\begin{scope}[xshift=7cm,yshift=0cm] 
	\draw (1.7,0) node{$\left[
	\begin{array}{rrr} 
	0 & -1 \\
	1 & 0 \\ 
	\hline
	\color{red!80!white}1 &  0 \\
	\color{red!80!white}0 &  1 \\
	\hline
	\color{red!80!white}2 &  0 \\
	\end{array}\right]$};
\end{scope}
\end{tikzpicture}
\end{center}
In the first mutation we note that the sign of $|c_1|$ does not change. Instead we increase its slope by 1. When we mutate the second time, $s_1=1$. So, the $c$-vector $c_2$, with slope $0$ (not equal to $s_1$ or $s_1+1$), is unchanged. When $s_1=2$, the mutation $\mu_1^+$ is no longer allowed.
\end{eg}

\section{Horizontal and vertical algebras}

Horizontal and vertical mutation fans are semi-invariant pictures for horizontal and vertical algebras which are associated to each $m$-configuration. We give definitions and examples.

\subsection{Definitions}

Let $X=\bigoplus X_i$ be a fixed $m$-configuration. We use round brackets to indicate slope. Thus
\[
	X_i=E_i(s_i)=E_i[m-s_i].
\]
\begin{lem}\cite{BRT}\label{lem: arrangement of m-conf}
The components $X_i=E_i(s_i)$ can be numbered so that
\begin{enumerate}
\item $(E_1,\cdots,E_n)$ is a complete exceptional sequence, i.e., $\Hom_\Lambda(E_j,E_i)=0=\Ext_\Lambda^1(E_j,E_i)$ for $i<j$.
\item $0\le s_1\le\cdots\le s_n\le m$. In particular, $s_i<s_j$ implies $i<j$.
\end{enumerate}
\end{lem}

We define the \emph{span} of an exceptional sequence $(M_1,\cdots,M_r)$ in $mod\text-\Lambda$ to be
\[
	span(M_1,\cdots,M_r)=\left(\,^{\perp}M\right)^{\perp}
\]
where $M=\bigoplus M_i$ and $\,^{\perp}M$ denotes the \emph{left hom-ext-perpendicular} subcategory of $mod\text-\Lambda$ of all $X$ so that $\Hom_\Lambda(X,M)=0=\Ext_\Lambda^1(X,M)$ and similarly for $M^{\perp}$. It is a well-known property of exceptional sequences (\cite{CB}, \cite{Ringel}) that the span of an exceptional sequence of length $r$ is equivalent to $mod\text-H$ for some hereditary algebra $H$ of rank $r$.

\begin{defn}
For $E_i(s_i)$ as above and $0\le s<t\le m$, let
\[
	\cA_{st}:=span(E_{s_k}\,|\, s<s_k\le t)
\]
and let $\Lambda_{st}(X)$ be the hereditary algebra with the property that $mod\text-\Lambda_{st}(X)\cong \cA_{st}$. For each $s$ let $H_s(X)$ and $V_s(X)$ be the algebras $H_s(X)=\Lambda_{2s,2s+1}(X)$ and $V_s(X)=\Lambda_{2s-1,2s}(X)$. We define the \emph{horizontal} and \emph{vertical algebras} of $X$ to be 
\[
	H(X)=\prod_{0\le s\le m/2} H_s(X),\quad V(X)=\prod_{0\le s\le (m+1)/2} V_s(X).
\]
\end{defn}

From Lemma \ref{lem: arrangement of m-conf} it follows that $\cA_{tn}^\perp=\cA_{0t}$ and $\cA_{tn}=\,^\perp\cA_{0t}$. This implies
\[
\cA_{st}=\cA_{tn}^\perp\cap \,^\perp\cA_{0s}.
\]
\begin{thm}
Let $X_k=E_k(s_k)$ and let $X'=\mu_k^+(X)$. Then
\begin{enumerate}
\item $H(X')=H(X)$ if $s_k$ is even.
\item $V(X')=V(X)$ if $s_k$ is odd.
\end{enumerate}
\end{thm}

\begin{proof}
This follows from the formula for the mutation $\mu_k^+$ given in Definition \ref{def: positive mutation of m=clusters}(2). Suppose $s_k$ is even. All objects $X_j$ with slope other than $s_k$ and $s_k+1$ are unchanged and all components with slope $s_k$ or $s_k+1$ are replaced with other components with slope $s_k$ or $s_k+1$. So, $\cA_{0,s_k-1}$ and $\cA_{s_k+1,n}$ are unchanged. So, $H_s(X')=H_s(X)$ for $s\neq s_k/2$. And $H_{s_k/2}(X)$ is also unchanged since $\cA_{s_k,s_k+1}=\,^\perp \cA_{0,s_k}\cap \cA_{s_k+1,n}^\perp$ is unchanged. The proof for $s_k$ odd is similar.
\end{proof}

We call the mutation $X'=\mu_k^+(X)$ (and $X=\mu_k^-(X')$) a \emph{horizontal mutation} if $s_k$ even, so that $H(X)=H(X')$. When $V(X)=V(X')$ it is a \emph{vertical mutation}.

\begin{defn}
The \emph{horizontal mutation fan} of $X$ is defined to be the semi-invariant picture $L(H(X))$ for the algebra $H(X)$. The \emph{vertical mutation fan} of $X$ is defined to be $-L(V(X))$, i.e., the set of all $x\in\RR^n$ so that $-x\in L(V(X))$.
\end{defn}

Assume that $m$ is odd. (In the examples, $m=3$. The classical case is $m=1$.) The claim is that the compartment in $L(H(X))$ corresponding to $X$ is equal as a subset of $\RR^n$ to the compartment in $-L(V(X))$ corresponding to $X$. This is because the vertices (corners) of that compartment are the $g$-vectors of the $m$-cluster $T=\bigoplus T_i$. The sign reversal for $L(V(X))$ is due to the $g$-vectors having the ``wrong'' sign. This comes from the implicit sign convention: objects in $mod\text-V(X)$ are put in degree 0 when they are actually in odd degrees by definitions. When $m$ is even, all the signs should be reversed. So, the claim still holds.

We will do two examples with $m=3$. In this case, $H(X)=H_0(X)\times \color{blue}H_1(X)$, where we color the second factor blue, and $V(X)=V_0(X)\times V_1(X)\times V_2(X)$ for any $X$.

\subsection{Example: $A_2$}

Let $\Lambda$ be the algebra of type $A_2$ from Example \ref{eg: A2 example, part 1}. Figure \ref{A2 example 1 fig a} and the left part of Figure \ref{A2 example 1 fig b} show the semi-invariant picture for $\Lambda$. There are five horizontal algebras corresponding to the five torsion classes of $A_2$. In the same order as in Chart \eqref{eq: chart for A2} they are:
\begin{enumerate}
\item $A_2\times\, \color{blue}0$ our notation for $\Lambda\times \color{blue}0$
\item $S_2\times \color{blue}S_1$
\item $\  0\ \times \color{blue}A_2$
\item $S_1\times \color{blue}P_2$
\item $P_2\times \color{blue}S_2$ which is shorthand for $H$ so that $mod\text-H=\add\, P_2\times\color{blue} add\, S_2$.
\end{enumerate}
For example, in Case (5), $X_1$ must be either $P_2$ or $P_2[1]$ and $X_2$ must be $S_2[2]$ or $S_2[3]$. Since $m=3$ any object will lie in either $mod\text-\Lambda[0]\cup mod\text-\Lambda[1]$, in which case it will contribute to $H_0$ or it will lie in $mod\text-\Lambda[2]\cup mod\text-\Lambda[3]$, in which case it will contribute to $H_1$.

Figure \ref{fig A2 horizontal fans} shows the semi-invariant pictures for these five algebras. Cases (1) and (3) are of type $A_2$ so have 5 clusters each. The other three are semi-simple of type $A_1\times \color{blue}A_1$ so have 4 clusters. The total is $5\times 2+4\times 3=22$ $m$-clusters for $m=3$. The formula for the number of $m$-clusters of type $A_n$ is the ``Fuss-Catalan number'' \cite{FR}, \cite{enumerate}:
\[
	\prod_{e=1}^n \frac{m(n+1)+e+1}{e+1}=\frac1{m(n+1)+1}\binom{(m+1)(n+1)}{n+1}
\]
equal to $22$ for $(n,m)=(2,3)$ and $140$ for $(n,m)=(3,3)$.

For each horizontal fan in Figure \ref{fig A2 horizontal fans} the walls are ${\bf D}_H(X)$ where 
\[
	X\in mod\text-(H_0\times {\color{blue}H_1})=mod\text-H_0 \coprod {\color{blue}mod\text-H_1}.
\]
The wall ${\bf D}_H(X)$ is colored blue if $X\in mod\text-H_1$.

Each compartment in each horizontal fan has two labels. The letters: $a,b,c,a^+,b^+,c^+$, $1,2,3,4,5,6$ label the 12 vertical fans. For example, the shaded regions form vertical fan (2). The spot diagram indicates a subset of the following diagram:
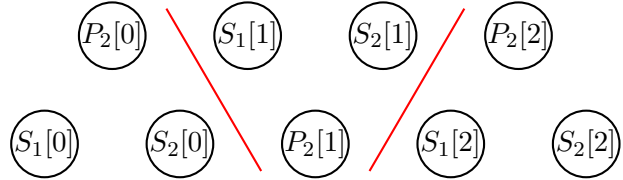
\begin{center}
\begin{tikzpicture}[scale=1.8]
\coordinate (S1) at (0,0);
\coordinate (P2) at (0.5,.8);
\coordinate (S2) at (1,0);
\coordinate (S11) at (-1.5,.8);
\coordinate (S21) at (-.5,.8);
\coordinate (P21) at (-1,0);
\coordinate (S12) at (-3,0);
\coordinate (P22) at (-2.5,.8);
\coordinate (S22) at (-2,0);
\coordinate (L1) at (-2.1,1);
\coordinate (L2) at (-1.4,-.2);
\coordinate (L3) at (-.6,-.2);
\coordinate (L4) at (0.1,1);
\begin{scope}
	\draw[thick, color=red] (L1)--(L2) (L3)--(L4);
	\draw[thick] (S1) circle[radius=7pt] node{$S_1[2]$}; 
	\draw[thick] (P2) circle[radius=7pt]node{$P_2[2]$};
	\draw[thick] (S2) circle[radius=7pt]node{$S_2[2]$};
	\draw[thick] (S11) circle[radius=7pt] node{$S_1[1]$};
	\draw[thick] (P21) circle[radius=7pt]node{$P_2[1]$};
	\draw[thick] (S21) circle[radius=7pt]node{$S_2[1]$};	
	\draw[thick] (S12) circle[radius=7pt] node{$S_1[0]$};
	\draw[thick] (P22) circle[radius=7pt]node{$P_2[0]$};
	\draw[thick] (S22) circle[radius=7pt]node{$S_2[0]$};
\end{scope}
\end{tikzpicture}
\end{center}
Filled spots indicate that the object is in the aisle corresponding to the $m$-cluster of the compartment. Also all objects in $mod\text-\Lambda[k]$ for $k\ge 3$ lie in all aisles. And all aisles are disjoint from $mod\text-\Lambda[k]$ for all negative $k$. The green lines in Figure \ref{fig A2 horizontal fans} indicate the partial ordering of the shaded regions which are the compartments in the vertical fan $0\times A_2\times 0$ assembled in Figure \ref{vertical fan A2}.


%
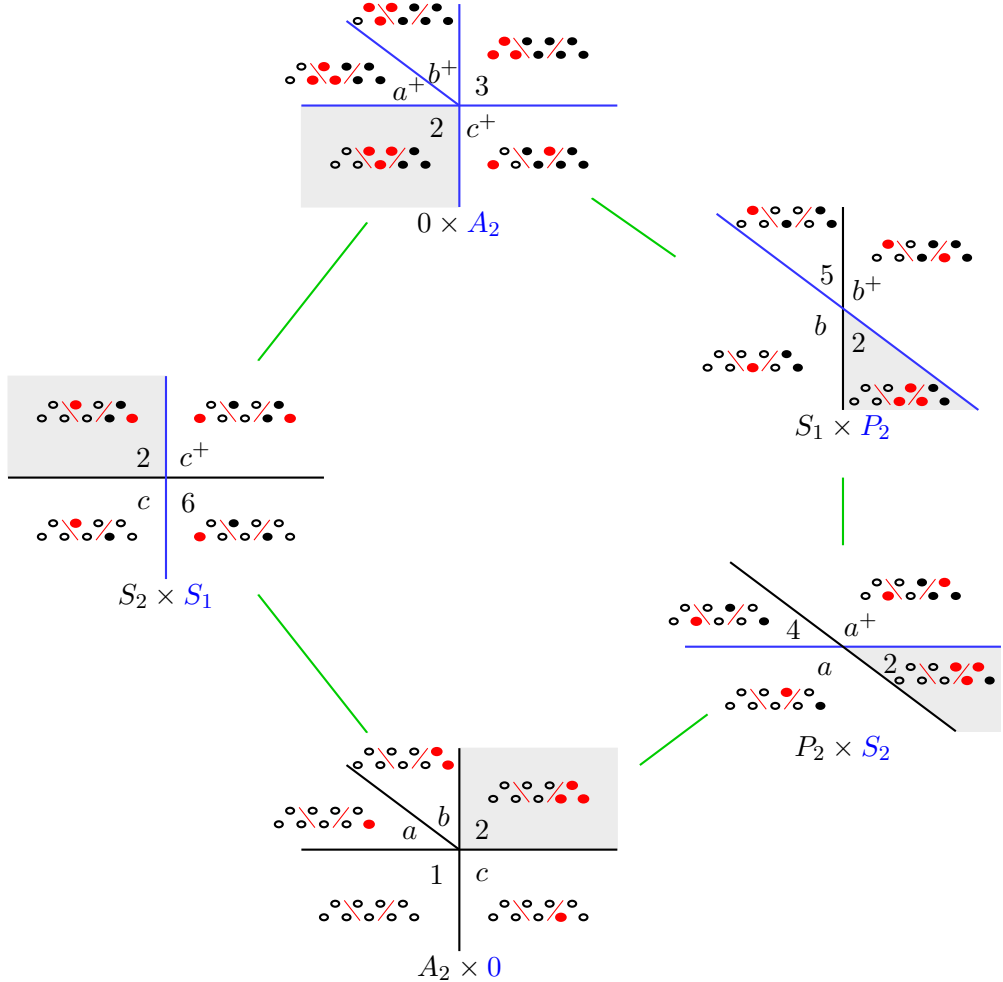
\begin{figure}[htbp] 
\begin{center}
\begin{tikzpicture}[xscale=.6, yscale=.45]
\draw[thick,color=green!80!black] (6,4)--(-1.5,11)--(-8,0)--(-1.5,-11);
\begin{scope}[xshift=-1.5cm, yshift=-11cm]
	\draw[fill,color=white] (0,0) ellipse [x radius=4cm, y radius=4cm];
	\AtwoLL{gray!15} 
\end{scope}
\begin{scope}[xshift=-1.5cm, yshift=11cm]
	\draw[fill,color=white] (0,0) ellipse [x radius=4cm, y radius=4cm];
	\AtwoUR{gray!15} 
\end{scope}
\begin{scope}[xshift=7cm, yshift=-5cm]
	\draw[fill,color=white] (0,0) ellipse [x radius=4cm, y radius=4cm];
	\AtwoLRb{gray!15} 
\end{scope}
\begin{scope}[xshift=7cm, yshift=5cm]
	\draw[fill,color=white] (0,0) ellipse [x radius=4cm, y radius=4cm];
	\AtwoLRc{gray!15} 
\end{scope}
\begin{scope}[xshift=-8cm, yshift=0cm]
	\draw[fill,color=white] (0,0) ellipse [x radius=4cm, y radius=4cm];
	\AtwoUL{gray!15} 
\end{scope}
\draw[thick,color=green!80!black] (7,-2)--(7,0) (2.5,-8.5)--(4,-7);
\end{tikzpicture}
\caption{These are 5 horizontal fans for $A_2: 1\leftarrow 2$ which contain the 22 $m$-clusters for $m=3$. The figure shows the corresponding bounded $t$-structures (with heart in red) which increase as we go northeast in each horizontal fan. The five shared regions form vertical fan (2). (See Figure \ref{vertical fan A2}).}
\label{fig A2 horizontal fans}
\end{center}
\end{figure}

$A_2:1\leftarrow 2$ has the following 12 vertical fans whose pieces are labeled in Figure \ref{fig A2 horizontal fans}. For example, vertical fan $(a)$ has two $m$-clusters which appear in horizontal fans $A_2\times\color{blue} 0$ and $P_2\times \color{blue}S_2$. But only one of the vertical fans is complete, namely $0\times A_2\times 0$, shown in Figure \ref{vertical fan A2}.
\begin{enumerate} 
\item $A_2\times \ 0\  \times \ 0\ $ with 1 $m$-cluster
\item $\ 0\ \times A_2 \times \ 0\ $ with 5 $m$-clusters
\item $\ 0\ \times \ 0\  \times A_2$ with 1 $m$-cluster
\item[$(a)$] $P_2\times\, S_2\times \ 0\ $ with 2 $m$-clusters
\item[$(b)$] $S_1\times\, P_2\times \ 0\ $ with 2 $m$-clusters
\item[$(c)$] $S_2\times\, S_1\times \ 0\ $ with 2 $m$-clusters
\item[$(a^+)$] $\ 0\ \times P_2\times S_2 $ with 2 $m$-clusters
\item[$(b^+)$] $\ 0\ \times S_1\times P_2$ with 2 $m$-clusters
\item[$(c^+)$] $\ 0\ \times S_2\times S_1$ with 2 $m$-clusters
\item $P_2\times \ 0\ \times S_2$ with 1 $m$-cluster
\item $S_1\times \ 0\ \times P_2 $ with 1 $m$-cluster
\item $S_2\times\ 0\ \times S_1 $ with 1 $m$-cluster
\end{enumerate}

\begin{figure}[htbp] 
\begin{center}
\begin{tikzpicture}[scale=.9]
\AtwoVert{gray!15}
\end{tikzpicture}
\caption{Vertical fan for $0\times A_2\times 0$. Indicated aisles are the union of torsion classes in $mod\text-\Lambda[1]$ with $mod\text-\Lambda[k]$ for all $k\ge2$. Objects in the heart are red.}
\label{vertical fan A2}
\end{center}
\end{figure}
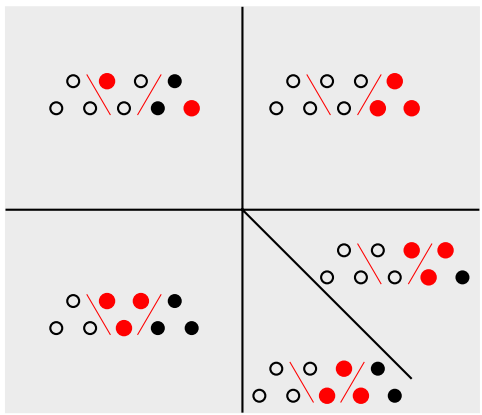

The $m$-maximal green sequence of maximal length is easy to visualize in Figures \ref{fig A2 horizontal fans} and \ref{vertical fan A2}: Start at the lower left in horizontal fan $A_2\times\color{blue}0$, go clockwise to (2) at the upper right. This is in the vertical fan $0\times A_2\times 0$ in Figure \ref{vertical fan A2}. Go clockwise around that fan to the lower left. This is in horizontal fan $0\times \color{blue}A_2$. In that fan, go clockwise to the maximum chamber in the upper right. This has 9 steps which add the 9 objects in the Auslander-Reiten quiver from right to left. The following example of an $m$-MGS was explained in the lecture by pointing to relevant parts of Figure \ref{fig A2 horizontal fans}.\vs2
\begin{center}
\begin{tikzpicture}[scale=.95]
\coordinate (B) at (0,1.8);
\coordinate (C) at (0,0.5);
\coordinate (Cslope) at (0.3,-0.45);
\coordinate (Vslope) at (0.3,-2.3);
\coordinate (V) at (0,-1.4);
\coordinate (M) at (1.7,0);
\begin{scope}
\draw (3.4,.5)  node{$\xrightarrow{\mu_2^+}$};
\draw (3.4,0) node{$H$};
\draw (1.7,-1.7) node{$(1)$};
	\draw (M) node{$\left[
	\begin{array}{rrr} 
	0 & -1 \\
	1 & 0 \\ 
	\hline
	 \color{black}1 & \color{green!66!black} 0 \\
	 \color{black}0 & \color{green!66!black} 1 \\
	\hline
	 \color{black}0 &\color{green!66!black} 0 \\
	\end{array}\right]$};
\end{scope}
\begin{scope}[xshift=3.5cm,yshift=0cm] 
	\draw (3.4,.5)  node{$\xrightarrow{\mu_2^+}$};
	\draw (3.4,0) node{$V$};
\draw (1.7,-1.7) node{$(a)$};
	\draw (1.7,0) node{$\left[
	\begin{array}{rrr} 
	0 & 1 \\
	-1 & 0 \\ 
	\hline
	\color{black}1 &\color{green!66!black}  0 \\
	\color{black}1 & \color{green!66!black} 1 \\
	\hline
	\color{black}0 &\color{green!66!black}  1 \\
	\end{array}\right]$};
\end{scope}
\begin{scope}[xshift=7cm,yshift=0cm] 
	\draw (3.4,.5)  node{$\xrightarrow{\mu_2^+}$};
	\draw (3.4,0) node{$H$};
\draw (1.7,-1.7) node{$(a)$};
	\draw (1.7,0) node{$\left[
	\begin{array}{rrr} 
	0 & -1 \\
	1 & 0 \\ 
	\hline
	\color{black}1 &\color{green!66!black}  0 \\
	\color{black}1 &\color{green!66!black}  1 \\
	\hline
	\color{black}0 &\color{green!66!black}  2 \\
	\end{array}\right]$};
\end{scope}
\begin{scope}[xshift=10.5cm,yshift=0cm] 
	\draw (3.4,.5)  node{$\xrightarrow{\mu_1^+}$};
	\draw (3.4,0) node{$H$};
\draw (1.7,-1.7) node{$(4)$};
	\draw (1.7,0) node{$\left[
	\begin{array}{rrr} 
	0 & 1 \\
	-1 & 0 \\ 
	\hline
	\color{green!66!black} 1 &\color{black} 0 \\
	\color{green!66!black} 1 & \color{black}1 \\
	\hline
	\color{green!66!black} 0 & \color{black}3 \\
	\end{array}\right]$};
\end{scope}
\begin{scope}[xshift=0cm,yshift=-4cm] 
\draw (3.4,.5)  node{$\xrightarrow{\mu_1^+}$};
\draw (3.4,0) node{$V$};
\draw (1.7,-1.7) node{$(a^+)$};
	\draw (1.7,0) node{$\left[
	\begin{array}{rrr} 
	0 & -1 \\
	1 & 0 \\ 
	\hline
	\color{green!66!black} 1 &  \color{black}0 \\
	\color{green!66!black} 1 & \color{black}1 \\
	\hline
	\color{green!66!black} 1 &\color{black}3 \\
	\end{array}\right]$};
\end{scope}
\begin{scope}[xshift=3.5cm,yshift=-4cm] 
\draw (3.4,.5)  node{$\xrightarrow{\mu_1^+}$};
\draw (3.4,0) node{$H$};
\draw (1.7,-1.7) node{$(a^+)$};
	\draw (1.7,0) node{$\left[
	\begin{array}{rrr} 
	0 & 1 \\
	-1 & 0 \\ 
	\hline
	\color{green!66!black} 1 &  \color{black}0 \\
	\color{green!66!black} 1 & \color{black}1 \\
	\hline
	\color{green!66!black} 2 &\color{black}3 \\
	\end{array}\right]$};
\end{scope}
\begin{scope}[xshift=7cm,yshift=-4cm] 
\draw (3.4,.5)  node{$\xrightarrow{\mu_1^+}$};
\draw (3.4,0) node{$H$};
\draw (1.7,-1.7) node{$(b^+)$};
	\draw (1.7,0) node{$\left[
	\begin{array}{rrr} 
	0 & -1 \\
	1 & 0 \\ 
	\hline
	\color{black}1 & \color{green!66!black}  1 \\
	\color{black}1 & \color{green!66!black} 0\\
	\hline
	\color{black}3 & \color{green!66!black}  2 \\
	\end{array}\right]$};
\end{scope}
\begin{scope}[xshift=10.5cm,yshift=-4cm] 
\draw (1.7,-1.7) node{$(3)$};
	\draw (1.7,0) node{$\left[
	\begin{array}{rrr} 
	0 & 1 \\
	-1 & 0 \\ 
	\hline
	\color{red}0 &  \color{red}1 \\
	\color{red}1 & \color{red}0\\
	\hline
	\color{red}3 & \color{red} 3 \\
	\end{array}\right]$};
\end{scope}
\end{tikzpicture}
\end{center}
\vs2
This starts at the bottom left compartment of the $A_2\times\color{blue}0$ horizontal fan. The first mutation is horizontal and goes to $(a)$ in the same horizontal fan. Then, we move southwest to compartment $(a)$ in $P_2\times\color{blue}S_2$. Then NE in that horizontal fan to $(a^+)$. Moving SW we get to $(a^+)$ in $0\times \color{blue}A_2$. Then we move NE in that horizontal fan to the maximal $t$-structure at $(3)$.

\subsection{Example}: 

$A_3$. Let $\Lambda=KQ$ where $Q$ is the quiver $1\leftarrow 2\rightarrow 3$. The poset of torsion classes is indicated in Figure \ref{A3: poset of torsion classes}. Each torsion class corresponds to a cluster for $\Lambda$.

The compartments of the vertical fan $0\times A_3\times 0$ which is shown in Figure \ref{A3: vertical fan 0 x A3 x 0} correspond to these torsion classes. But the $g$-vectors have the opposite sign because of the shift in degree.

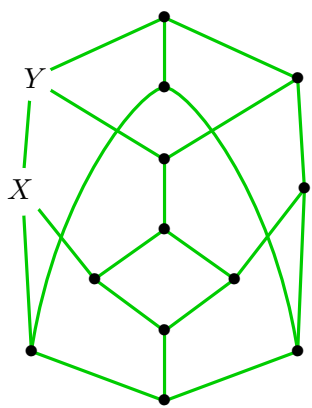
\begin{figure}[htbp] 
\begin{center}
\begin{tikzpicture}[scale=.3]
\coordinate (bottom) at (0,-3.2);
\coordinate (S1perp) at (-5.9,-1);
\coordinate (S3perp) at (5.9,-1);
\coordinate (S2perp) at (0,-.1);
\coordinate (I1perp) at (-3.1,2.2);
\coordinate (I3perp) at (3.1,2.2);
\coordinate (P2) at (0,4.4);
\coordinate (I3) at (-6.3,6.2);
\coordinate (I3right) at (-4.8,6.2);
\coordinate (I1left) at (4.8,6.2);
\coordinate (I1) at (6.2,6.2);
\coordinate (P2perp) at (0,7.5);
\coordinate (S2) at (0,10.7);
\coordinate (S2left) at (-1,10.7);
\coordinate (S2right) at (1,10.7);
\coordinate (P3) at (-5.9,11.1);
\coordinate (P1) at (5.9,11.1);
\coordinate (top) at (0,13.8);
\draw[very thick, color=green!80!black] (bottom)--(S1perp)--(I3)--(P3)--(top)--(S2);
\draw[very thick, color=green!80!black] (bottom)--(S3perp)--(I1)--(P1)--(top);
\draw[very thick, color=green!80!black] (bottom)--(S2perp)--(I1perp)--(I3) (I1perp)--(P2)--(P2perp)--(P3) (S2perp)--(I3perp)--(I1) (I3perp)--(P2) (P2perp)--(P1);
\draw[very thick, color=green!80!black] (S1perp) .. controls (I3right) and (S2left)..(S2);
\draw[very thick, color=green!80!black] (S3perp) .. controls (I1left) and (S2right)..(S2);
\draw (S2) node{$\bullet$};
\draw (S1perp) node{$\bullet$};
\draw (S2perp) node{$\bullet$};
\draw (S3perp) node{$\bullet$};
\draw (P1) node{$\bullet$};
\draw (P2) node{$\bullet$};
\draw (I1perp) node{$\bullet$};
\draw (I3perp) node{$\bullet$};
\draw (I1) node{$\bullet$};
\draw (P2perp) node{$\bullet$};
\draw (top) node{$\bullet$};
\draw (bottom) node{$\bullet$};
\begin{scope}[scale=.45,yshift=14cm,xshift=-14cm]
	\draw[fill,color=white] (0,-.5) circle [radius=2.3cm];
		\draw (-.2,-.2) node{$X$};
\end{scope}
\begin{scope}[scale=.45,yshift=25cm,xshift=-13cm]
	\draw[fill,color=white] (0,-.5) circle [radius=2cm];
		\draw[thick,color=black] (0.3,-.3) node{$Y$};
\end{scope}
\end{tikzpicture}
\caption{Poset of 14 torsion classes for $A_3: 1\leftarrow 2\to 3$.}
\label{A3: poset of torsion classes}
\end{center}
\end{figure}
%


%
\begin{figure}[htbp]
\begin{center}
\begin{tikzpicture}[scale=.9]
	\VerticalFan 
\end{tikzpicture}
\caption{The vertical fan $0\times A_3\times 0$ (the stereographic projection to $\RR^2$ of $-L(V(X))\cap S^2$). Each region corresponds to a cluster in the cluster category, e.g., $X$ corresponds to $I_3\oplus P_1[1]\oplus P_3$. But the figure shows $g$-vectors for the corresponding $m$-clusters which have the opposite sign, e.g, $g(I_3[1])=-g(I_3)$.}
\label{A3: vertical fan 0 x A3 x 0}
\end{center}
\end{figure}
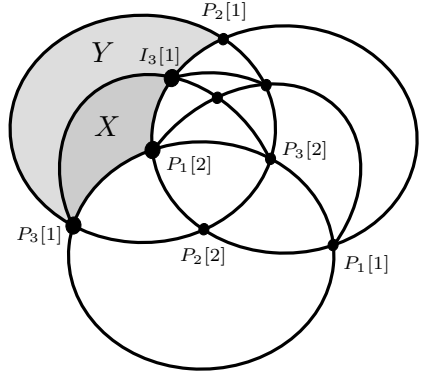

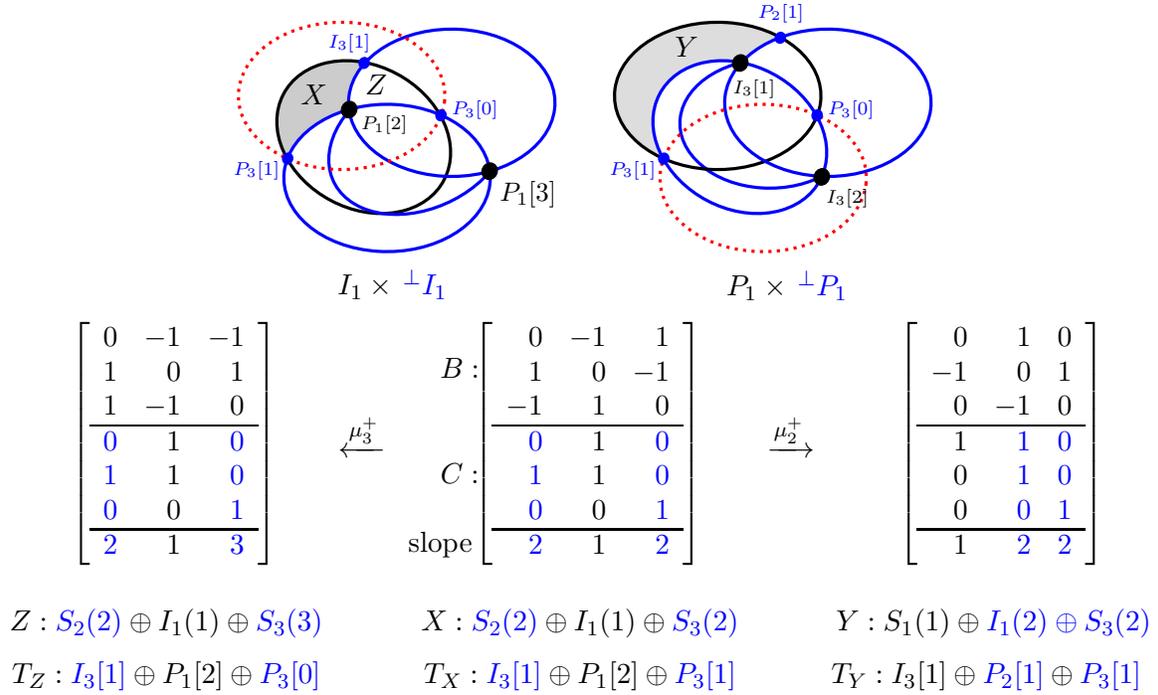
\begin{figure}[htbp]
\begin{center}
\begin{tikzpicture}
\begin{scope}[scale=.7]
\PoneBig 
\end{scope}
\begin{scope}[xshift=-5cm,scale=.7]
\IoneBig 
\end{scope}
\end{tikzpicture}
%
\begin{tikzpicture}
\coordinate (B) at (0,1.8);
\coordinate (C) at (0,0.5);
\coordinate (Cslope) at (0.3,-0.45);
\coordinate (Vslope) at (0.3,-2.3);
\coordinate (V) at (0,-1.4);
\coordinate (M) at (1.7,0);
\draw (9.9,1) node{$\xrightarrow{\mu_2^+}$};
\draw (4.2,1)  node{$\xleftarrow{\mu_3^+}$};
\begin{scope}[xshift=5.5cm,yshift=0cm]
	\draw (0,1.9) node{$B:$} ;
	\draw (0,0.5) node{$C:$} ;
	\draw (0.3,-0.45) node[left]{slope} ;
	\draw (1.6,-2.2) node{$T_X:{\color{blue} I_3[1]} \oplus P_1[2]\oplus {\color{blue}P_3[1]}$};
	\draw (1.6,-1.5) node{$X: {\color{blue} S_2(2)} \oplus I_1(1)\oplus {\color{blue}S_3(2)}$};
	\draw (1.7,.9) node{$\left[
	\begin{array}{rrr} 
	0 & -1 & 1\\
	1 & 0 & -1\\ 
	-1 & 1 & 0\\
	\hline
	 \color{blue}0 & 1 &\color{blue}0\\
	 \color{blue}1 &1 & \color{blue}0\\
	 \color{blue}0 &0 & \color{blue}1\\
	\hline
	 \color{blue}2 &1 & \color{blue}2\\
	\end{array}\right]$};
\end{scope}
\begin{scope}[xshift=11cm,yshift=0cm] 
	%
\draw (1.5,-2.2) node{$T_Y:{\color{black} I_3[1]} \oplus {\color{blue}P_2[1]}\oplus {\color{blue}P_3[1]}$};
\draw (1.6,-1.5) node{$Y:{S_1(1)} \oplus\color{blue}  I_1(2)\oplus {\color{blue}S_3(2)}$};
	\draw (1.7,.9) node{$\left[
	\begin{array}{rrr} 
	0 & 1 & 0\\
	-1 & 0 & 1\\ 
	0 & -1 &0\\
	\hline1 & \color{blue}1 & \color{blue}0\\
	0 & \color{blue}1 & \color{blue}0\\ 
	0 &\color{blue} 0 &\color{blue}1\\
	\hline
	1 & \color{blue}2 & \color{blue}2\\
	\end{array}\right]$};
\end{scope}
\begin{scope}
	\draw (1.6,-2.2) node{$T_Z:{\color{blue} I_3[1]} \oplus P_1[2]\oplus {\color{blue}P_3[0]}$};
	\draw (1.6,-1.5) node{$Z: {\color{blue} S_2(2)} \oplus I_1(1)\oplus {\color{blue}S_3(3)}$};
	\draw (1.7,.9) node{$\left[
	\begin{array}{rrr} 
	0 & -1 & -1\\
	1 & 0 & 1\\ 
	1 & -1 & 0\\
	\hline
	 \color{blue}0 & 1 &\color{blue}0\\
	 \color{blue}1 &1 & \color{blue}0\\
	 \color{blue}0 &0 & \color{blue}1\\
	\hline
	 \color{blue}2 &1 & \color{blue}3\\
	\end{array}\right]$};
\end{scope}
\end{tikzpicture}
\caption{At top: the horizontal fans $H(X)=H(Z): I_1\times\color{blue} \,^\perp I_1$ and $H(Y): P_1\times\color{blue} \,^\perp P_1$. $Y=\mu_2^+(X)$ is a vertical mutation. $Z=\mu_3(X)$ is a horizontal mutation. Matrix $B$ mutates in an unexpected way since the only change in $c$-vectors under $\mu_3^+$ is changing the sign of $c_3$.}
\label{A3: horizontal fans H(X), H(Y)}
\end{center}
\end{figure}
%


There are 55 vertical mutation fans. Only the one for $0\times A_3\times 0$ is shown.

Figure \ref{Fig:A3 Horizontal fans} shows all 14 horizontal fans for $A_3$. The regions corresponding to the $m$-cluster in the vertical fan $0\times A_3\times 0$ from Figure \ref{A3: vertical fan 0 x A3 x 0} are shaded and the green lines indicate the ordering of these shaded $m$-clusters. Thus, the horizontal fans are placed at the nodes of the Hasse diagram in Figure \ref{A3: poset of torsion classes}.

Figure \ref{preview} in the introduction shows the five horizontal fans on the left side of Figure \ref{Fig:A3 Horizontal fans}: $A_3\times \color{blue}0$, $P_3^\perp \times \color{blue}P_3$, $I_1\times \color{blue}^\perp I_1$, $P_1\times \color{blue}^\perp P_1$ and $0\times \color{blue}A_3$. One can visualize several $m$-maximal green sequences in the figure as follows. On the bottom floor $A_3\times \color{blue}0$, start in the unbounded region and go to the center. The longest such green path, of length 6, goes up to the center from below. Now, take the four stairs going up to the top floor $0\times \color{blue}A_3$. Equivalently, move in the vertical fan from the center out to the unbounded shaded region going through the shaded regions in Figure \ref{preview}. Then move to the center of the top floor. If we take the green stairs (dashed) in Figure \ref{preview}, this $m$-MGS has length 16. But the maximum number of steps is $3\times 6=18$. This is achieved by taking the longest path in the vertical fan from middle to outside.


%
\begin{figure}[htbp] 
\begin{center}
\begin{tikzpicture}
%
\coordinate (bottom) at (0,-3.2);
\coordinate (S1perp) at (-5.9,-1);
\coordinate (S3perp) at (5.9,-1);
\coordinate (S2perp) at (0,-.1);
\coordinate (I1perp) at (-3.1,2.2);
\coordinate (I3perp) at (3.1,2.2);
\coordinate (P2) at (0,4.4);
\coordinate (I3) at (-6.3,6.2);
\coordinate (I3right) at (-4.8,6.2);
\coordinate (I1left) at (4.8,6.2);
\coordinate (I1) at (6.2,6.2);
\coordinate (P2perp) at (0,7.5);
\coordinate (S2) at (0,10.7);
\coordinate (S2left) at (-1,10.7);
\coordinate (S2right) at (1,10.7);
\coordinate (P3) at (-5.9,11.1);
\coordinate (P1) at (5.9,11.1);
\coordinate (top) at (0,13.8);
\draw[very thick, color=green!80!black] (S1perp) .. controls (I3right) and (S2left)..(S2);
\draw[very thick, color=green!80!black] (S3perp) .. controls (I1left) and (S2right)..(S2);
\draw[very thick, color=green!80!black] (bottom)--(S1perp)--(I3)--(P3)--(top)--(S2);
\draw[very thick, color=green!80!black] (bottom)--(S3perp)--(I1)--(P1)--(top);
\draw[very thick, color=green!80!black] (bottom)--(S2perp)--(I1perp)--(I3) (I1perp)--(P2)--(P2perp)--(P3) (S2perp)--(I3perp)--(I1) (I3perp)--(P2) (P2perp)--(P1);
\begin{scope}[scale=.45,yshift=-7cm] 
	\draw[fill,color=white] (0,-.5) ellipse [x radius=3.3cm,y radius=2.8cm];
	\allblack
\end{scope}
\begin{scope}[scale=.45,xshift=13cm,yshift=-2cm]
	\draw[fill,color=white] (0,-.5) ellipse [x radius=3.3cm, y radius=2.8cm];
	\PonePerp
\end{scope}
\begin{scope}[scale=.45]
	\draw[fill,color=white] (0,-.5) ellipse [x radius=3.3cm, y radius=2.8cm];
	\StwoPerp
\end{scope}
\begin{scope}[scale=.45,xshift=-13cm,yshift=-2cm]
	\draw[fill,color=white] (0,-.5) ellipse [x radius=3.3cm, y radius=2.8cm];
	\PthreePerp
\end{scope}
\begin{scope}[scale=.45,yshift=5cm,xshift=7cm]
	\draw[fill,color=white] (0,-.5) ellipse [x radius=3.3cm, y radius=2.8cm];
	\IonePerp
\end{scope}
\begin{scope}[scale=.45,yshift=5cm,xshift=-7cm]
	\draw[fill,color=white] (0,-.5) ellipse [x radius=3.3cm, y radius=2.8cm];
	\IthreePerp
\end{scope}
\begin{scope}[scale=.45,yshift=14cm,xshift=14cm]
	\draw[fill,color=white] (0,-.5) ellipse [x radius=3.3cm, y radius=3cm];
	\Ithree
\end{scope}
\begin{scope}[scale=.45,yshift=10cm]
	\draw[fill,color=white] (0,-.5) ellipse [x radius=3.3cm, y radius=2.5cm];
	\Ptwo
\end{scope}
\begin{scope}[scale=.45,yshift=14cm,xshift=-14cm]
	\draw[fill,color=white] (0,-.5) ellipse [x radius=3.3cm, y radius=3cm];
	\Ione
		\draw[thick,color=black] (-1.8,2.3) node{$X$}(-1.7,1.8) --(-1.4,.5);
\end{scope}
\begin{scope}[scale=.45,yshift=17cm]
	\draw[fill,color=white] (0,-.5) ellipse [x radius=3.3cm, y radius=3cm];
	\PtwoPerp
\end{scope}
\begin{scope}[scale=.45,yshift=25cm,xshift=13cm]
	\draw[fill,color=white] (0,-.5) ellipse [x radius=3.3cm, y radius=3.3cm];
	\Pthree
\end{scope}
\begin{scope}[scale=.45,yshift=23.5cm]
	\draw[fill,color=white] (0,-.5) ellipse [x radius=3.3cm, y radius=2.4cm];
	\Stwo
\end{scope}
\begin{scope}[scale=.45,yshift=25cm,xshift=-13cm]
	\draw[fill,color=white] (0,-.5) ellipse [x radius=3.3cm, y radius=3.3cm];
	\Pone
		\draw[thick,color=black] (-1.8,2.5) node{$Y$}(-1.7,2) --(-1.2,1.1);
\end{scope}
\begin{scope}[scale=.45,yshift=30cm]
	\draw[fill,color=white] (0,-.5) ellipse [x radius=3.3cm, y radius=2.8cm];
	\allblue
\end{scope}
\end{tikzpicture}
\caption{The 14 horizontal mutation fans for the $A_3$ quiver $1\leftarrow 2\rightarrow 3$ with $m=3$. $H_0$ walls are black, $H_1$ walls are blue. The shaded regions are the 14 chambers of the vertical mutation fan for $0\times A_3\times 0$ which correspond to the torsion classes for $\Lambda=A_3$. The green lines show the partial ordering of these chambers (cut and pasted from Figure \ref{A3: poset of torsion classes}).}
\label{Fig:A3 Horizontal fans}
\end{center}
\end{figure}
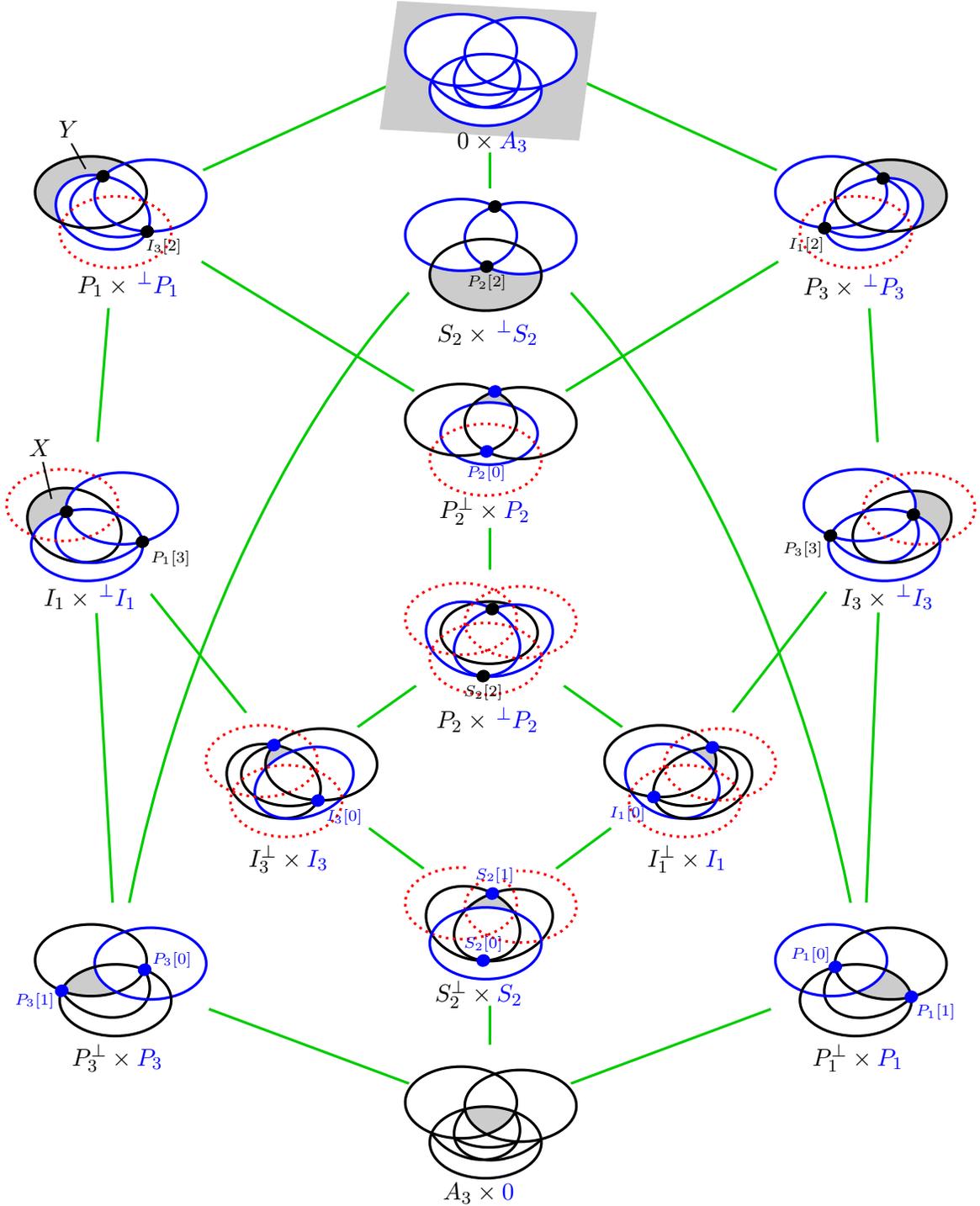
%
%


\section{Other papers}

I have several other papers which explain the ``picture'' point of view. At the request of the referee, here is a list of the papers and how they are related.
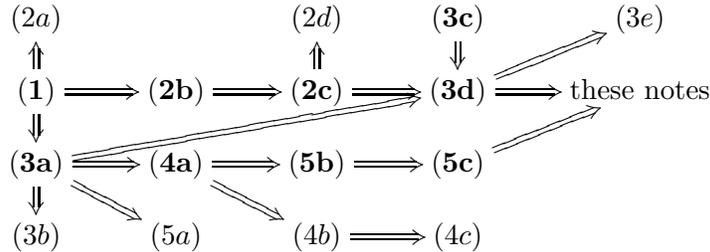
\[
\xymatrixrowsep{10pt}\xymatrixcolsep{25pt}
\xymatrix{
(2a) && (2d) & \bf(3c)\ar@{=>}[d] & (3e)\\
\bf(1) \ar@{=>}[r]\ar@{=>}[d]\ar@{=>}[u]&\bf(2b)\ar@{=>}[r]&\bf(2c)\ar@{=>}[r]\ar@{=>}[u] &\bf(3d)\ar@{=>}[r]\ar@{=>}[ru] &\text{these notes}\\
\bf(3a) \ar@{=>}[rrru]\ar@{=>}[rd]\ar@{=>}[r]\ar@{=>}[d]&\bf(4a) \ar@{=>}[r]\ar@{=>}[rd]& \bf(5b)\ar@{=>}[r]& \bf(5c)\ar@{=>}[ru]\\
(3b) &(5a)&  (4b) \ar@{=>}[r]& (4c)
	}
\]

\begin{enumerate}
\item (with G. Todorov, K. Orr, J. Weyman) ``Modulated semi-invariants''\cite{IOTW2} develops the basic cluster theory of hereditary algebras from the point of view of semi-invariant pictures.
\item Papers about the picture group.
\begin{enumerate}
	\item (with G. Todorov) ``Pictures groups and maximal green sequences''\cite{IT14}. This paper proves that, in finite type, maximal green sequences are in bijection with positive expressions for the Coxeter element of the picture group.
	\item (with G. Todorov) ``Signed exceptional sequences ...''\cite{IT13} defines the ``cluster morphism category'' and proves that its classifying space is a $K(G,1)$ for quivers of finite type where $G$ is the picture group of the quiver. A purely combinatorial version of the cluster morphism category, for type $A_n$, is given in \cite{Cat0}.
	\item (with G. Todorov and J. Weyman) ``Picture groups of finite type ...''\cite{ITW16} computes the cohomology of the picture group of type $A_n$ with any orientation.
	\item (reporting on work of Eric Hanson) ``Are finite type picture groups virtually special?''\cite{finite} generalizes pictures and picture groups to all finite dimensional algebras using recent results of \cite{BCZ}, \cite{BST}, \cite{Tr}, \cite{DIJ}.
\end{enumerate}
\item Pictures for tame quivers, also called ``propictures'', e.g., Figure \ref{FigA2tilde}.
\begin{enumerate}
	\item (with T. Br\"ustle, S. Hermes, G. Todorov) ``Semi-invariants pictures ...''\cite{BHIT} uses semi-invariant pictures to study MGSs and show that there are only finitely many MGSs for cluster-tilted algebras of tame type.
	\item (with S. Hermes) ``The no gap conjecture for tame hereditary algebras''\cite{HI} proves the ``no-gap conjecture'' for these algebras using \cite{BHIT}. This leads to an easy proof of the quantum dilogarithm identity for MGSs for tame quivers. (Theorem \ref{thm: quantum DT for tame}.)
	\item (with G. Todorov and J. Weyman) ``Periodic trees and semi-invariants''\cite{ITW14} shows that clusters of type $\tilde A_n$ can be represented by periodic trees.
	\item (with G. Todorov, M. Kim, J. Weyman) ``Periodic trees and propictures''\cite{IOTW3} defines ``propictures'' and the ``propicture groups'' which are inverse limits of picture groups. Figure \ref{FigA2tilde} above is an example. 
	\item (with M. Kim) ``Cluster propictures of type $\tilde A_n$''\cite{IK} extends definitions and theorems of \cite{IOTW3} to cluster-tilted algebras of type $\tilde A_n$.
\end{enumerate}
\item Papers on the ``linearity'' question for maximal green sequences. (See Section \ref{quantum}.)
\begin{enumerate}
	\item ``Linearity of stability conditions''\cite{PartI} gives many equivalent definitions of maximal green sequences (mostly well-known) for hereditary algebras using the corresponding Harder-Narasimhan filtration. (See Theorem \ref{thm: def of MGS} above.)
	\item ``Maximal green sequences for cluster-tilted algebras of finite type''\cite{PartII} extends the results of \cite{PartI} to cluster-tilted algebras of finite type and gives a conjectured formula for the maximum length of a MGS in these cases.
	\item (with PJ Apruzzese) ``Stability conditions for affine type $A$''\cite{AI} uses \cite{PartI}, \cite{PartII} to find the maximum length of a MGS for $\tilde A_{a,b}$ and determine which are linear.
\end{enumerate}
\item Papers about $m$-maximal green sequences:
\begin{enumerate}
\item[(a)] (with Y. Zhou) ``Tame hereditary algebras ...''\cite{IZ} gives a short module-theoretic proof that tame acyclic quivers have only finitely many $m$-maximal green sequences extending the theorem of \cite{BDP} to the $m$-cluster case.
\item[(b)] ``$m$-noncrossing trees,''\cite{next} gives the $m$-cluster version of ``cobinary trees'' \cite{IO} and introduces the mutation formula for $m$-clusters in terms of the extended exchange matrix with an additional row for ``slope'' which, in \cite{next}, is the actual slope of an edge in the ``$m$-noncrossing tree''.
\item[(c)] ``Enumerating $m$-clusters ...''\cite{enumerate} reinterprets Fomin and Reading \cite{FR} using an $m$-cluster version of signed exceptional sequences \cite{IT13}. The present lecture notes on ``horizontal and vertical fans'' are extracted from an early version of \cite{enumerate}.
\end{enumerate}
\end{enumerate}

\subsection{Quantum dilogarithm identities}\label{quantum}
Keller introduced maximal green sequences in \cite{Keller} to obtain formulas for quantum Donaldson-Thomas (DT) invariants for quivers. This was based on the earlier work of Reineke \cite{R} who obtained quantum DT-invariants of Dynkin quivers using linear stability conditions (Bridgeland \cite{B1}). Reineke conjectured \cite{R} that every Dynkin quiver admits a ``central charge'' (linear stability condition) making all roots stable. Yu Qiu \cite{Q} proved this for at least one orientation of every Dynkin quiver. My joint paper with PJ\cite{AI} answers the analogous question for quivers of the affine type $\tilde A_n$, namely: What is the length of the longest maximal green sequence and which of these are linear?

Linear stability conditions are given by straight lines in the semi-invariant picture. For example, $A, D$ and $F$ are linear and $B,C,E$ are also equivalent to linear maximal green sequences since there exist straight lines crossing the same walls in the same order.

Reineke \cite{R} showed that linear stability conditions gave quantum dilogarithm formulas for DT-invariants. Keller \cite{Keller} realized that (nonlinear) stability conditions, given by maximal green sequences, also gave the same DT-invariants. A MGS can also be viewed as a path in the space of Bridgeland stability conditions \cite{B2} which, by \cite{PartI}, is equivalent to a smooth ``green'' path in $\RR^n$ transverse to the semi-invariant picture.

The proof of the ``no-gap conjecture''\cite{HI} proves that acyclic tame quivers have well-defined DT-invariants. This is because the quantum dilogarithm $\mathbb E(M)=\mathbb E(y^{\undim M})$ satisfies the following square and pentagon identities (\cite[Thm 1.2]{Keller}).

\begin{lem}\label{quantum lemma}
Suppose $M,N$ are hom-orthogonal $KQ$-modules for an acyclic quiver $Q$ and $\Ext_{KQ}(N,M)=0$.
\begin{enumerate}
\item If $\Ext_{KQ}(M,N)=0$ then $\mathbb E(M),\mathbb E(N)$ commute.
\item If $\Ext_{KQ}(M,N)=K$ and $N\cof L\onto M$ is the nontrivial extension then
\[
	\mathbb E(N)\mathbb E(M)=\mathbb E(M)\mathbb E(L)\mathbb E(N).
\]
\end{enumerate}
\end{lem}

\begin{proof} Let $\alpha=\undim M,\beta=\undim N$. In Case (1), $y^\alpha y^\beta=y^\beta y^\alpha$ which implies that $\mathbb E(y^\alpha)$, $\mathbb E(y^\beta)$ commute. Case (2) follows from the pentagon identity for $\mathbb E$ and the equation
\[
	y^{\alpha+\beta}=q^{\frac12 (\alpha,\beta)}y^\alpha y^\beta=q^{\frac12}y^\alpha y^\beta=q^{-\frac12}y^\beta y^\alpha 
\]
where $\alpha+\beta=\undim L$ and $(\alpha,\beta)=\alpha^t B\beta=\dim \Ext_{KQ}(M,N)=1$.
\end{proof}

\begin{thm}\label{thm: quantum DT for tame}
Tame acyclic quivers have well-defined quantum DT-invariants given by the product of all $\mathbb E(M_i)$ for $M_1,\cdots,M_r$ any maximal green sequence. 
\end{thm}

\begin{proof}
By \cite[Sec 4]{HI} any two MGSs for a tame quiver differ by a sequence of square and pentagon moves. Lemma \ref{quantum lemma} shows that such ``polygonal deformations'' of MGSs do not change the product of corresponding quantum dilogarithms.
\end{proof}

After my lecture at Nankai University, someone asked if there was an $m$-cluster analogue of the quantum dilogarithm identities. That is a very interesting question that I hope to answer in another paper.




\begin{thebibliography}{aa}


\bibitem{AI} PJ Apruzzese and Kiyoshi Igusa, \emph{Stability conditions for affine type A}, arXiv:1804.09100.


\bibitem{BMRRT}
Aslak~Bakke Buan, Robert~J. Marsh, Marcus Reineke, Idun Reiten, and Gordana Todorov, \emph{Tilting theory and cluster combinatorics}, Adv. Math. \textbf{204} (2006), no.~2, 572--618. 


\bibitem{BRT}
Aslak~Bakke Buan, Idun Reiten and Hugh Thomas, \emph{From $m$-clusters to $m$-noncrossing partitions via exceptional sequences}, Mathematische Zeitschrift 271 (2012), no. 3-4, 1117--1139. 

\bibitem{BCZ} Emily Barnard, Andrew Carroll, Shijie Zhu, \emph{Minimal inclusions of torsion classes}, arXiv:1710.08837.


\bibitem{BBD} A.A. Beilinson, J. Bernstein, and P. Deligne, ``Analyse et Topologie sur les espaces singuliers Vol. I.'', Ast\'erisque, vol. 100, Soc. Math. France, 1982. 

\bibitem{B1}Tom Bridgeland, \emph{Stability conditions on triangulated categories}, Ann. Math. \textbf{166}, No. 2 (Sep., 2007), pp. 317--345. 

\bibitem{B2} Tom Bridgeland, \emph{Spaces of stability conditions}, Algebraic geometry--Seattle 2005. Part 1 (2009): 1--21. 


\bibitem{BDP} Thomas Br\"ustle, Gr\'egoire Dupont, and Matthieu P\'erotin, \emph{On maximal green sequences}, International Mathematics Research Notices 2014 (2014), no. 16, 4547--4586.

\bibitem{BHIT}
Thomas Br\"ustle, Stephen Hermes, Kiyoshi Igusa and Gordana Todorov, \emph{Semi-invariant pictures and two conjectures on maximal green sequences}, J Algebra {\bf 473}, March 2017, 80--109. 

\bibitem{BST} Thomas Br\"ustle, David Smith, Hipolito Treffinger, \emph{Stability conditions, tau-tilting theory and maximal green sequences}, arXiv:1705.08227.




\bibitem{Chindris} Calin Chindris, \emph{Cluster fans, stability conditions, and domains of semi-invariants}, Transactions of the American Mathematical Society, 363 (2011), no. 4, 2171--2190. 


\bibitem{CB}
William Crawley-Boevey, \emph{Exceptional sequences of representations of quivers}. In Representations of
algebras (Ottawa, ON, 1992), volume 14 of CMS Conf. Proc., pages 117--124. Amer. Math. Soc., Providence, RI, 1993. 

\bibitem{DIJ} Laurent Demonet, Osamu Iyama and Gustavo Jasso, \emph{$\tau $-tilting finite algebras, bricks and $g$-vectors}, Int. Math. Res. Not. (2017), 1--41.


\bibitem{DW}
Harm Derksen and Jerzy Weyman, \emph{Semi-invariants of quivers and saturation for {L}ittlewood-{R}ichardson coefficients}, J. Amer. Math. Soc. \textbf{13} (2000), no.~3, 467--479 (electronic). 

\bibitem{FR}
Sergey Fomin and Nathan Reading, \emph{Generalized cluster complexes and Coxeter combinatorics}, International Mathematics Research Notices 44 (2005), 2709--2757. 

\bibitem{FZ}
Sergey Fomin and Andrei Zelevinsky, \emph{Cluster algebras. {IV}.  {C}oefficients}, Compos. Math. \textbf{143} (2007), no.~1, 112--164. 

\bibitem{GHKK} Mark Gross, Paul Hacking, Sean Keel, and Maxim Kontsevich, \emph{Canonical bases for cluster algebras}, Journal of the American Mathematical Society 31.2 (2018): 497--608.



\bibitem{HI} Stephen Hermes and Kiyoshi Igusa, \emph{The no gap conjecture for tame hereditary algebras}, arXiv:1601.04054.


\bibitem{IK} Kiyoshi Igusa and Moses Kim, \emph{Cluster 
propictures of type $\widetilde A_n$}, in preparation.


\bibitem{IOTW3}
Kiyoshi Igusa, Moses Kim, Gordana Todorov, and Jerzy Weyman, \emph{Periodic trees and propictures}, preprint available at \href{http://people.brandeis.edu/~igusa/Papers/PeriodicProPics.pdf}{http://people.brandeis.edu/$\sim$igusa/Papers/PeriodicProPics.pdf}. 

\bibitem{IOTW2}
Kiyoshi Igusa, Kent Orr, Gordana Todorov, and Jerzy Weyman, \emph{Modulated semi-invariants}, arXiv:1507.03051. 



\bibitem{IO}
Kiyoshi Igusa and Jonah Ostroff, \emph{Mixed cobinary trees}, Journal of Algebra and Its Applications (2017),
https://doi.org/10.1142/S0219498818501700.


\bibitem{ITW14} 
Kiyoshi Igusa, Gordana Todorov, and Jerzy Weyman, 
\emph{Periodic trees and semi-invariants}, arXiv:1407.0619.

\bibitem{ITW16} 
\bysame, \emph{Picture groups of finite type and cohomology in type $A_n$}, arXiv:1609.02636. 






\bibitem{IT13} Kiyoshi Igusa and Gordana Todorov, \emph{Signed exceptional sequences and the cluster morphism category}, arXiv: 1706.2222.

\bibitem{IT14} \bysame, \emph{Picture groups and maximal green sequences}, preprint available at \href{http://people.brandeis.edu/~igusa/Papers/GreenSeq.pdf}{http://people.brandeis.edu/$\sim$igusa/Papers/GreenSeq.pdf}.

\bibitem{IZ} Kiyoshi Igusa and Ying Zhou, \emph{Tame hereditary algebras have finitely many $m$-maximal green sequences}, arXiv:1706.09118.

\bibitem{Cat0} Kiyoshi Igusa, \emph{The category of noncrossing partitions}, arXiv:1411.0196.

\bibitem{next}
\bysame, \emph{m-noncrossing trees}, Journal of Algebra and Its Applications (2017): 1850187.

\bibitem{enumerate}
\bysame, \emph{Enumerating $m$-clusters using exceptional sequences}, in preparation. 

\bibitem{PartI} \bysame, \emph{Linearity of stability conditions}, arXiv:1706.06986. 

\bibitem{PartII} \bysame, \emph{Maximal green sequences for cluster tilted algebras of finite type}, arXiv:1706.06503.

\bibitem{finite} \bysame, \emph{Are finite type picture groups virtually special?}, lecture notes available at \href{http://homepage.divms.uiowa.edu/~fbleher/CGMRT2017/Slides/Igusa2017Notes.pdf}{http://homepage.divms.uiowa.edu/$\sim$fbleher/CGMRT2017/Slides/Igusa2017Notes.pdf}.

\bibitem{IY} Osamu Iyama and Yuji Yoshino, \textit{Mutation in triangulated categories and rigid Cohen-Macaulay modules}, Inventiones mathematicae 172, no. 1 (2008): 117--168. 

\bibitem{Kase} Ryoichi Kase, \emph{Remarks on lengths of maximal green sequences for quivers of type $\tilde A_{n,1}$}, arXiv:1507.02852. 

\bibitem{Keller} {Bernhard Keller}, \emph{On cluster theory and quantum dilogarithm identities}, Representations of algebras and related topics, EMS Series of Congress Reports, European Mathematical Society, 2011, 85--116. 


\bibitem{KV} Bernhard Keller and Dieter Vossieck, \emph{Aisles in derived categories}, Bull. Soc. Math. Belg. S\'er. A 40.2 (1988): 239--253. 

\bibitem{KQ}Alastair King and Yu Qiu, \emph{Exchange graphs and Ext quivers}, Adv. Math. {\bf285} (2015), 1106--1154. 


\bibitem{KY} Steffen Koenig and Dong Yang, \textit{Silting objects, simple-minded collections, t-structures and co-t-structures for finite-dimensional algebras}, Doc. Math 19 (2014): 403--438. 

\bibitem{Muller} Greg Muller, \emph{The existence of a maximal green sequence is not invariant under quiver mutation}, Electron. J. Combin., \textbf{23}(2) (2016):Paper 2.47. 

\bibitem{NZ}
Tomoki Nakanishi and Andrei Zelevinsky, \emph{On tropical dualities in cluster algebras}, Contemp. Math \textbf{565} (2012), 217|226.


\bibitem{Q} Yu Qiu, \emph{C-sortable words as green mutation sequences,} Proceedings of the London Mathematical Society 111.5 (2015): 1052--1070.


\bibitem{Q2}\bysame, \emph{Stability conditions and quantum dilogarithm identities for Dynkin quivers}, Adv. Math. {\bf269} (2015), 220--264. 



\bibitem{R} Markus Reineke, \emph{The Harder-Narasimhan system in quantum groups and cohomology of quiver moduli}, Invent. Math. {\bf152} (2003), no. 2, 349--368. 



\bibitem{Ringel}
Claus Michael Ringel, \emph{The braid group action on the set of exceptional sequences of a hereditary Artin algebra} In Abelian group theory and related topics (Oberwolfach, 1993), volume 171 of Contemp. Math., pages 339--352. Amer. Math. Soc., Providence, RI, 1994. 


\bibitem{ST}
David Speyer and Hugh Thomas, \emph{Acyclic cluster algebras revisited}, ``Algebras, quivers and representations, Proceedings of the Abel Symposium 2011 (2013), 275--298. 


\bibitem{Tm}
 Hugh Thomas, \emph{Defining an $m$-cluster category}, J. Algebra 318 (2007), no. 1, 37--46. 
 
 \bibitem{Tr} Hipolito Treffinger, \emph{On sign coherence of $c$-vectors}, arXiv: 1711.01152.
 
\bibitem{Woolf} Jonathan Woolf, \emph{Stability conditions, torsion theories and tilting}, J. Lond. Math. Soc. (2) 82 (2010), no. 3, 663--682. 



\end{thebibliography}
\end{document}